\documentclass[a4paper,12pt, reqno]{amsart}
\usepackage{amsfonts, color}
\usepackage{geometry}
\usepackage{amstext, amsthm, amssymb, amsmath}
\usepackage{mathrsfs}
\usepackage{graphicx}
\usepackage{color}
\usepackage{todonotes}

\theoremstyle{plain}
\begingroup
\newtheorem{theorem}{Theorem}[section]
\newtheorem{lemma}[theorem]{Lemma}
\newtheorem{proposition}[theorem]{Proposition}
\newtheorem{corollary}[theorem]{Corollary}
\endgroup

\theoremstyle{definition}
\begingroup
\newtheorem{defn}{Definition}

\endgroup

\theoremstyle{remark}
\newtheorem{remark}{Remark}

\newcommand{\R}{\mathbb{R}}
\newcommand{\F}{\mathcal{F}}
\newcommand{\U}{\mathcal{U}}
\newcommand{\V}{\mathcal{V}}
\newcommand{\M}{\mathcal{M}}
\newcommand{\N}{\mathcal{N}}
\newcommand{\El}{\mathcal{L}}
\newcommand{\e}{\varepsilon}
\newcommand{\weak}{\rightharpoonup}
\newcommand{\G}{\mathcal{G}}

\newcommand{\Fb}{\F^{\beta}}
\newcommand{\be}{\beta}
\newcommand{\NN}{\mathbb{N}}

\begin{document}
\title[A gradient flow equation for optimal control
problems]
{A gradient flow equation for optimal control
problems with end-point cost}
\author[A. Scagliotti]{A. Scagliotti}

\address[A.~Scagliotti]{Scuola Internazionale Superiore di Studi Avanzati, Trieste, Italy}
\email{ascaglio@sissa.it}

\begin{abstract}
In this paper we consider a control system 
of the form $\dot x = F(x)u$,
linear in the control variable $u$. 
Given a fixed starting point, we study a
finite-horizon optimal control problem, 
where we want to minimize a weighted sum of 
an end-point cost and the 
squared $2$-norm of the control.
This functional induces a gradient flow
on the Hilbert space of admissible controls, 
and we prove a convergence result by means of the
Lojasiewicz-Simon inequality. 
Finally, we show that, if we let
the weight of the end-point cost tend to infinity,
the resulting family of functionals is
$\Gamma$-convergent, and it turns
out that the limiting problem consists in joining
the starting point and a minimizer of the
end-point cost with a horizontal length-minimizer
path. 

\subsection*{Keywords} Gradient flow, optimal control,
end-point cost, Lojasiewicz-Simon inequality,
$\Gamma$-convergence.
\end{abstract}

\maketitle

\subsection*{Acknowledgments} 
The Author acknowledges partial support from
INDAM--GNAMPA.
The Author wants to
thank Prof. A.~Agrachev and Prof. A.~Sarychev
for encouraging and for the helpful discussions.
Finally, the Author is grateful to an 
anonymous referee for the invaluable
comments that contributed to improve the
overall quality of the paper.

\maketitle

\begin{section}{Introduction}
In this paper we consider a control system of the form
\begin{equation} \label{eq:ctrl_sys_intro}
\dot x = F(x)u,
\end{equation}
where $F:\R^n \to \R^{n\times k}$ is a 
Lipschitz-continuous
function, and $u\in \R^k$ is the control variable.
If $k\leq n$, for every $x\in\R^n$
we may think of the columns 
$\{ F^i(x) \}_{i=1,\ldots,k}$ of the matrix $F(x)$
as an ortho-normal frame of vectors, defining
a sub-Riemannian structure on $\R^n$.
For a thorough introduction to the topic, we refer
the reader to the monograph \cite{ABB}.
In our framework, $\U:=L^2([0,1],\R^k)$ will be
the space of admissible controls, equipped with the 
usual Hilbert space structure. 
Given a base-point $x_0\in\R^n$, for every
$u\in\U$ we consider the absolutely
continuous trajectory $x_u:[0,1]\to\R^n$
that solves
\begin{equation} \label{eq:ctrl_cau_intro}
\begin{cases}
\dot x_u(s) = F(x_u(s))u(s) &\mbox{for a.e. }
s\in[0,1], \\
x_u(0)=x_0.
\end{cases}
\end{equation}
For every 
$\be >0$ and $x_0 \in \R^n$,
 we define the the functional
$\Fb:\U\to\R_+$ as follows:
\begin{equation}\label{eq:F_e_intro}
\Fb(u) := \frac12 ||u||_{\U}^2 + \be a(x_u(1)), 
\end{equation} 
where $a:\R^n \to \R_+$ is a non-negative 
$C^1$-regular function,
and $x_u:[0,1]\to \R^n$ is the solution of 
\eqref{eq:ctrl_cau_intro} corresponding
to the control $u\in \U$.
In this paper we want to investigate the gradient
flow induced by the functional $\Fb$ on the Hilbert
space $\U$, i.e., the evolution equation
\begin{equation} \label{eq:grad_flow_intro}
\partial_t U_t = -\G^\be[U_t],
\end{equation}   
where $\G^\be:\U\to\U$ is the vector field
on the Hilbert space $\U$ that represents
the differential $d\Fb:\U\to\U^*$ through the
Riesz's isometry. In other words, for every
$u\in \U$, we denote with 
$d_u\Fb:\U\to\R$ the differential of
$\Fb$ at $u$, and
$\G^\be[u]$ is defined as the only element
of $\U$ such that the identity
\begin{equation} \label{eq:rep_intro}
\langle \G^\be[u],v \rangle_{L^2} = d_u\Fb(v)
\end{equation}
holds for every $v\in \U$. 
In order to avoid confusion,
we use different letters to denote the
time variable in the control system
\eqref{eq:ctrl_cau_intro} and in the
evolution equation \eqref{eq:grad_flow_intro}.
Namely, the variable $s\in[0,1]$ will be exclusively
used for the control system
\eqref{eq:ctrl_cau_intro}, while 
$t\in [0,+\infty)$ will be employed
only for the gradient flow 
\eqref{eq:grad_flow_intro} and the corresponding
trajectories. Moreover, when dealing with
operators taking values in a space of
functions, we express the argument using the
square brackets.

The first part of the paper  is devoted to the
formulation of the gradient flow equation
\eqref{eq:grad_flow_intro}. In particular, we first 
study the differentiability of the functional 
$\Fb:\U\to\R_+$, then we introduce the vector field
$\G^\be:\U\to\U$ as the representation
of its differential, and finally we show that,
under suitable assumptions, 
$\G^\be$ is locally Lipschitz-continuous. 
As a matter of fact, it turns out that 
\eqref{eq:grad_flow_intro} can be treated as 
an infinite-dimensional ODE, and we prove that,
for every initial datum $U_0=u_0$, the gradient
flow equation \eqref{eq:grad_flow_intro} admits
a unique continuously differentiable
solution $U:[0,+\infty)\to\U$.
In the central part of this contribution
we focus on the asymptotic behavior
of the curves that solve
\eqref{eq:grad_flow_intro}.  
The main result states that, if the 
application $F:\R^n\to\R^{n\times k}$ 
that defines the linear-control system
\eqref{eq:ctrl_sys_intro} is real-analytic 
as well as the function
$a:\R^n\to\R_+$ that provides the end-point
term in \eqref{eq:F_e_intro}, then, for every
$u_0\in H^1([0,1],\R^k)\subset \U$, 
the curve $t\mapsto U_t$ that solves
the gradient flow equation
\eqref{eq:grad_flow_intro} with initial
datum $U_0=u_0$ satisfies
\begin{equation} \label{eq:conv_intro}
\lim_{t\to+\infty} ||U_t-u_\infty||_{L^2} =0,
\end{equation} 
where $u_\infty \in \U$ is a critical point for
$\Fb$.
Finally, in the last part of this work we prove 
a $\Gamma$-convergence result concerning the
family of functionals $(\Fb)_{\be\in\R_+}$.
In particular, we show that, when
$\be \to+\infty$, the limiting problem
consists in minimizing the $L^2$-norm of the
controls that steer the initial point
$x_0$ to the set $\{x\in\R^n:a(x) =0\}$. 
This fact can be applied, for example, to
approximate the problem of finding a 
sub-Riemannian length-minimizer curve that 
joins two assigned points.

We report below in detail
the organization of the sections.\\
In Section~\ref{sec:ass_not} we introduce the
linear-control system \eqref{eq:ctrl_sys_intro}
and we establish some preliminary results that
will be used throughout the paper.
In particular, in Subsection~\ref{subsec:1ord}
we focus on the first variation of a trajectory
when a perturbation of the corresponding control
occurs. In Subsection~\ref{subsec:2nd_ord}
we study the second variation of the trajectories
at the final evolution instant. \\
In Section~\ref{sec:well_posed} we prove that,
for every intial datum $u_0\in \U$, the
evolution equation \eqref{eq:grad_flow_intro}
gives a well-defined
Cauchy problem whose solutions exist for  
every $t\geq 0$.
To see that, we use the results obtained in
Subsection~\ref{subsec:1ord} to introduce the 
vector field $\G^\be:\U\to\U$ satisfying
\eqref{eq:rep_intro} and to prove that
it is Lipschitz-continuous when restricted to
the bounded subsets of $\U$. 
Combining this fact with the theory
of ODEs in Banach spaces (see, e.g., \cite{LL72}),
it descends that, for every choice of the
initial datum $U_0=u_0$, the evolution equation 
\eqref{eq:grad_flow_intro} admits a unique and
locally defined solution $U:[0,\alpha)\to\U$,
with $\alpha>0$.
Using the particular structure of the 
gradient flow equation \eqref{eq:grad_flow_intro},
we finally
manage to extend these solutions for every 
positive time. \\
In Section~\ref{sec:comp} we prove that, if the
Cauchy datum $u_0$ has Sobolev regularity (i.e.,
$u_0 \in H^m([0,1],\R^k)\subset \U$ for 
some positive integer $m$), then the curve
$t\mapsto U_t$ that
solves \eqref{eq:grad_flow_intro} 
and satisfies $U_0=u_0$ is pre-compact in $\U$.
The key-observation lies in the fact that,
under suitable regularity assumptions on
$F:\R^n\to\R^{n\times k}$ and
$a:\R^n\to\R_+$,
the Sobolev space $H^m([0,1],\R^k)$ is invariant
for the gradient flow \eqref{eq:grad_flow_intro}.
Moreover, we obtain that, when the Cauchy
datum belongs to $H^m([0,1],\R^k)$, the curve 
$t\mapsto U_t$
that solves \eqref{eq:grad_flow_intro}
is bounded in the $H^m$-norm. \\
In Section~\ref{sec:Loj_Sim} we establish the 
Lojasiewicz-Simon inequality for the functional
$\Fb:\U\to\R_+$, under the assumption that
$F:\R^n\to\R^{n\times k}$ and $a:\R^n\to\R_+$
are real-analytic.
This family of inequalities was first 
introduced by Lojasiewicz
in \cite{L63} for real-analytic
functions defined on a finite-dimensional domain.
The generalization of this result to 
real-analytic functionals defined on a Hilbert  
space was proposed by Simon in \cite{S83}, and
since then it has revealed to be an invaluable
tool to study convergence properties of
evolution equations (see the survey paper
\cite{C03}). Following this approach, the
Lojasiewicz-Simon inequality for the 
functional $\Fb$ is the cornerstone for the
convergence result of the subsequent section.\\    
In Section~\ref{sec:conv} we prove that,
if the Cauchy datum belongs to
$H^m([0,1],\R^k)$ for an integer $m\geq1$, 
the corresponding gradient flow trajectory
converges to a critical point of $\Fb$.
This result requires that both 
$F:\R^n\to\R^{n\times k}$ and $a:\R^n\to\R_+$
are real-analytic. Indeed, we use the 
Lojasiewicz-Simon inequality for
$\Fb:\U\to\R_+$ to show that the solutions
of \eqref{eq:grad_flow_intro} with
Sobolev-regular initial datum  have
finite length. This fact immediately yield 
\eqref{eq:conv_intro}.\\
In Section~\ref{sec:Gamma} we 
study the behavior of the minimization problem
\eqref{eq:F_e_intro} when the
positive parameter $\beta$ tends to infinity. 
We address this problem using the tools
of the $\Gamma$-convergence (see
\cite{D93} for a complete introduction to the
subject). 
In particular, we consider $\U_\rho:=\{ u\in\U:
||u||_{L^2}\leq \rho \}$ and we equip 
it with the topology of the weak convergence of 
$\U$. For every $\be>0$, we introduce the 
restrictions $\Fb_\rho:= \Fb|_{\U_\rho}$, and we
show that there exists a
functional $\F_\rho:\U_\rho\to\R_+\cup\{+\infty\}$ 
such that the family $(\Fb_\rho)_{\be\in\R_+}$
$\Gamma$-converges to $\F_\rho$ as 
$\beta\to+\infty$. 
In the case $a:\R^n\to\R_+$ admits a unique 
point $x_1\in \R^n$ such that $a(x_1)=0$,
then the limiting problem of minimizing
the functional $\F_\rho$ consists in 
finding (if it exists) 
a control $u\in\U_\rho$
with minimal $L^2$-norm
such that the corresponding curve
$x_u:[0,1]\to\R^n$ defined by
\eqref{eq:ctrl_cau_intro} satisfies
$x_u(1)=x_1$.
The final result of Section~\ref{sec:Gamma}
guarantees that the minimizers of
$\Fb_\rho$ provide $L^2$-strong
approximations of the minimizers of 
$\F_\rho$. 

\begin{subsection}{General notations}
Let us briefly introduce some basic notations
that will be used throughout the paper. 
For every $d\geq1$, we equip
the space $\R^d$ with the standard Euclidean norm
$|\cdot|_2$, i.e.,
$ |z|_2=\sqrt{\langle z,z\rangle}_{\R^d}$
for every $z\in \R^d$,
induced by the standard scalar product
\[
\langle z, w\rangle_{\R^d} := \sum_{i=1}^d z^i w^i
\]
for every $z,w\in \R^d$.
We shall often use the relation
\begin{equation} \label{eq:norm_equiv_1_2}
\frac1{\sqrt d}|z|_2\leq |z|_1 \leq \sqrt d |z|_2
\end{equation}
for every $z\in \R^d$, where 
$|z|_1:= \sum_{i=1}^d|z^i|$.
For every $d\geq 1$, if $M\in \R^{d\times d}$ is
an endomorphism of $\R^d$, we define
\begin{equation} \label{eq:norm_mat_def}
|M|_2 := \sup_{z\neq 0}\frac{|Mz|_2}{|z|_2}.
\end{equation}
We recall that in any finite-dimensional
vector space, all the norms are equivalent. 
In particular, if $|\cdot|_A, |\cdot|_B$ are
norms in $\R^{d\times d}$, then there exists
$C\geq 1$ such that
\begin{equation} \label{eq:norm_mat_equiv}
\frac1C |M|_A \leq |M|_B \leq C|M|_A
\end{equation}
for every $M\in \R^{d\times d}$. 
\end{subsection}

\end{section}

\begin{section}{Framework and preliminary results}
\label{sec:ass_not}

In this paper we consider control systems on $\R^n$
with linear dependence in the control variable
$u\in \R^k$, i.e.,
of the form   
\begin{equation} \label{eq:ctrl_sys}
\dot x = F(x)u,
\end{equation}
where $F:\R^n \to \R^{n\times k}$ is a 
Lipschitz-continuous function. 
We use the notation $F^i$ for $i=1,\ldots,k$ to
indicate the vector fields on $\R^n$ obtained by
taking the columns of $F$, and
we denote by $L>0$ the Lipschitz constant of 
these vector fields, i.e., we set
\begin{equation}\label{eq:Lipsch_const_F}
L:= \sup_{i=1,\ldots,k} \, \sup_{x,y\in \R^n}
\frac{|F^i(x)-F^i(y)|_2}{|x-y|_2}.
\end{equation}
We immediately observe that 
\eqref{eq:Lipsch_const_F} implies that 
the vector fields $F^1,\ldots,F^k$ have sub-linear
growth, i.e., there exists $C>0$ such that
\begin{equation} \label{eq:sub_lin}
\sup_{i=1,\ldots,k}|F^i(x)| \leq C(|x|_2+1)
\end{equation}
for every $x\in \R^n$. Moreover, for every
$i=1,\ldots,k$, if $F^i$ is differentiable at 
$y\in \R^n$, then from \eqref{eq:Lipsch_const_F}
we deduce that
\begin{equation} \label{eq:bound_jac_F}
\left|\frac{\partial F^i(y)}{\partial x}
\right|_2 
\leq L.
\end{equation}
We define $\U := L^2([0,1],\R^k)$
as the space of admissible controls, 
and we endow 
$\U$ with the usual Hilbert space structure, 
induced by the scalar product
\begin{equation} \label{eq:scal_prod}
\langle u,v \rangle_{L^2}
= \int_0^1 \langle u(s),v(s) \rangle_{\R^k}
\, ds.
\end{equation}
Given $x_0 \in \R^n$, for every $u\in\U$,
let 
$x_u:[0,1]\to \R^n$ be 
the absolutely continuous curve that solves 
the following Cauchy problem:
\begin{equation}\label{eq:ctrl_Cau}
\begin{cases}
\dot x_u(s) = F(x_u(s))u(s)& 
\mbox{for a.e. }s\in[0,1], \\
x_u(0) = x_0.
\end{cases}
\end{equation}
We recall that, under the condition
\eqref{eq:Lipsch_const_F},
the existence and uniqueness
of the solution of \eqref{eq:ctrl_Cau}
is guaranteed by Carath\'eodory Theorem
(see, e.g, \cite[Theorem 5.3]{H80}). 
We insist on the fact that
in this paper the Cauchy datum $x_0\in \R^n$
is assumed to be assigned.

In the remainder of this section we 
introduce auxiliary results
that will be useful in the other sections.
In Subsection~\ref{subsec:sobolev}
we recall some results concerning
Sobolev spaces in one-dimensional domains.
In Subsection~\ref{subsec:1ord} and
Subsection~\ref{subsec:2nd_ord} we investigate
the properties of the solutions of 
\eqref{eq:ctrl_Cau}.

\begin{subsection}{Sobolev spaces in one dimension} \label{subsec:sobolev}
In this subsection we recall some results
for one-dimensional Sobolev spaces. 
Since in this paper we work only in Hilbert spaces,
we shall restrict our attention to 
the Sobolev exponent $p=2$, i.e., we 
shall state the results for the Sobolev
spaces $H^m:=W^{m,2}$ with $m\geq 1$.
For a complete discussion on the topic, the
reader is referred to \cite[Chapter~8]{B11}.
Throughout the paper we use the convention 
$H^0 : =L^2$. 

For every integer $d\geq 1$,
given a compact interval $[a,b]\subset \R$,
let $C^\infty_c([a,b],\R^d)$ be the set of
the $C^\infty$-regular functions with compact
support in $[a,b]$. For every 
$\phi \in C^\infty_c([a,b],\R^d)$, we use the
symbol $\phi^{(\ell)}$ to denote the 
$\ell$-th derivative of
the function $\phi:[a,b]\to\R^d$.
For every $m\geq 1$, the function 
$u\in L^2([a,b],\R^d)$ belongs to the Sobolev
space $H^m([a,b],\R^d)$ if and only if, for every
integer $1\leq \ell\leq m$ there exists
$u^{(\ell)}\in L^2([a,b],\R^d)$ such that
the following identity holds
\[
\int_a^b \langle u(s), \phi^{(\ell)}(s)
\rangle_{\R^d}\, ds =
(-1)^\ell \int_a^b 
\langle
u^{(\ell)}(s), \phi(s)
\rangle_{\R^d}\, ds
\]
for every $\phi \in C^\infty_c([a,b],\R^d)$.
If  $u\in H^m([a,b],\R^d)$, 
then for every integer $1\leq \ell\leq m$
$u^{(\ell)}$ denotes the $\ell$-th Sobolev 
derivative of  $u$.
We recall that, for every $m\geq 1$, 
$H^m([a,b],\R^d)$ is a Hilbert space
(see, e.g., \cite[Proposition~8.1]{B11})
when it is equipped with
the norm $||\cdot||_{H^m}$
induced by the scalar product
\[
\langle u,v \rangle_{H^m} := \langle u, v \rangle_{L^2} +
\sum_{\ell = 1}^m \int_a^b \langle
u^{(\ell)}(s), v^{(\ell)}(s) \rangle_{\R^d}\,
ds.
\]
We observe that, for every $m_2> m_1\geq 0$, we 
have
\begin{equation} \label{eq:sob_ineq_triv}
||u||_{H^{m_1}} \leq  ||u||_{H^{m_2}}
\end{equation}
for every $u \in H^{m_2}([a,b],\R^l)$, i.e., the
inclusion $H^{m_2}([a,b],\R^d)  \hookrightarrow
H^{m_1}([a,b],\R^d)$ is continuous.
We recall that a linear and
continuous application $T: E_1 \to E_2$
between two Banach spaces $E_1, E_2$
is {\it compact} if, for every bounded set
$B\subset E_1$, the image $T(B)$
is pre-compact with respect to the strong
topology of $E_2$.  
In the following result we list three 
compact inclusions.

\begin{theorem} \label{thm:comp_sob_imm}
For every $m\geq 1$, the following inclusions
are compact:
\begin{equation} \label{eq:comp_sob_L2}
H^m([a,b],\R^d)  \hookrightarrow L^2([a,b],\R^d),
\end{equation}
\begin{equation} \label{eq:comp_sob_C0}
H^m([a,b],\R^d)  \hookrightarrow C^0([a,b],\R^d),
\end{equation}
\begin{equation} \label{eq:comp_sob_sob}
H^m([a,b],\R^d)  \hookrightarrow 
H^{m-1}([a,b],\R^d),
\end{equation}
\end{theorem}
\begin{proof}
When $m=1$, 
\eqref{eq:comp_sob_L2}-\eqref{eq:comp_sob_C0}
descend directly from \cite[Theorem~8.8]{B11}.
In the case $m\geq 2$, we observe that, in 
virtue of \eqref{eq:sob_ineq_triv}, the immersion
$H^m([a,b],\R^d)\hookrightarrow H^1([a,b],\R^d)$
is continuous. Recalling that the composition
of a linear continuous operator with a 
linear compact one is still compact (see, e.g.,
\cite[Proposition~6.3]{B11}),
we deduce that \eqref{eq:comp_sob_L2}-\eqref{eq:comp_sob_C0}
holds also for $m\geq 2$.

When $m=1$, \eqref{eq:comp_sob_sob}
reduces to \eqref{eq:comp_sob_L2}.
For $m\geq2$, \eqref{eq:comp_sob_sob} is proved
by induction on $m$, using \eqref{eq:comp_sob_L2}
and observing that $u\in H^m([a,b],\R^d)$
implies that $u^{(1)}\in H^{m-1}([a,b],\R^d)$.
\end{proof}

Finally, we recall the notion of {\it weak
convergence}. For every $m\geq 0$ 
(we set $H^0:=L^2$),
if $(u_n)_{n\geq1}$ is a sequence in 
$H^m([0,1],\R^d)$ and $u\in H^m([0,1],\R^d)$, then
the sequence $(u_n)_{n\geq1}$ weakly converges
to $u$ if and only if
\begin{equation*}
\lim_{n\to\infty}\langle v,u_n\rangle_{H^m} = 
\langle v,u\rangle_{H^m} 
\end{equation*}
for every $v\in H^m([0,1],\R^d)$, and we write
$u_n\weak_{H^m} u$ as $n\to\infty$.
Finally, in view of the compact inclusion 
\eqref{eq:comp_sob_sob} 
and of \cite[Remark~6.2]{B11}, for every $m\geq 1$,
if a sequence 
$(u_n)_{n\geq1}$ in $H^m([0,1],\R^d)$ 
satisfies $u_n\weak_{H^m}u$ as $n\to \infty$,
then
\begin{equation*}
\lim_{n\to\infty} ||u_n-u||_{H^{m-1}}=0.
\end{equation*}

\end{subsection}

\begin{subsection}{General properties of
the linear-control system \eqref{eq:ctrl_sys}} 
\label{subsec:1ord}
In this subsection we investigate 
basic properties of  the solutions of 
\eqref{eq:ctrl_Cau}, with a particular
focus on the relation between the
admissible control $u\in\U$ and the
corresponding trajectory $x_u$.
We start by stating a 
version of the Gr\"onwall-Bellman inequality, that
will be widely used later.
We recall that this kind of inequalities 
plays an important role in the study
of integral and differential equations.
For a complete survey on the topic, the reader
is referred to the textbook \cite{P98}.

\begin{lemma}[Gr\"onwall-Bellman Inequality] 
\label{lem:Gron}
Let $f:[a,b]\to\R_+$ be a
non-negative continuous function
and let us assume that there exists
a constant $\alpha>0$ and a non-negative 
function $\beta\in L^1([a,b],\R_+)$
such that
\[
f(s) \leq \alpha + \int_a^s\beta(\tau)f(\tau) \,d\tau
\]
for every $s\in[a,b]$. Then, for every 
$s\in[a,b]$ the following inequality holds:
\begin{equation} \label{eq:Gron_ineq}
f(s) \leq \alpha e^{||\beta||_{L^1}}.
\end{equation}
\end{lemma}
\begin{proof}
This statement follows as a particular case
of \cite[Theorem~5.1]{EK86}. 
\end{proof}
We recall that, for every 
$u\in \U :=L^2([0,1],\R^k)$ the following
inequality holds:
\begin{equation} \label{eq:norm_ineq}
||u||_{L^1} = \int_0^1 \sum_{i=1}^k|u^i(s)| \,ds
\leq \sqrt k
\sqrt{\int_0^1 \sum_{i=1}^k|u^i(s)|^2 \,ds}
= \sqrt k ||u||_{L^2}.
\end{equation}

We first show that, for every admissible 
control $u\in \U$, the corresponding 
solution of \eqref{eq:ctrl_Cau}
is bounded in the $C^0$-norm.
In our framework, given a 
continuous function $f:[0,1]\to\R^n$, we set
\begin{equation*}
||f||_{C^0}:= \sup_{s\in[0,1]}|f(s)|_2.
\end{equation*}

\begin{lemma} \label{lem:C0_bound_traj}
Let $u\in \U$ be an admissible control, and let
$x_u:[0,1]\to\R^n$ be the
solution of the
Cauchy problem \eqref{eq:ctrl_Cau}
corresponding to the control $u$.
Then, the following inequality holds:
\begin{equation}\label{eq:est_norm_traj}
||x_u||_{C^0} \leq 
\left( |x_0|_2 + \sqrt k C||u||_{L^2} 
\right) e^{\sqrt k C||u||_{L^2}},
\end{equation}
where $C>0$ is the constant of sub-linear growth
prescribed by \eqref{eq:sub_lin}.
\end{lemma}
\begin{proof}
Rewriting \eqref{eq:ctrl_Cau} in the integral form,
we obtain the following inequality
\[
|x_u(s)|_2 
\leq |x_0|_2 + \int_0^s
\sum_{i=1}^k \Big(|F^i(x_u(\tau))|_2 
|u^i(\tau)|\Big) \,d\tau
\]
for every $s\in[0,1]$.
Then, using \eqref{eq:sub_lin}, we deduce that
\[
|x_u(s)|_2 
\leq |x_0|_2 + C||u||_{L^1} 
+ C\int_0^s
|u(\tau)|_1 |x_u(\tau)|_2 \,d\tau.
\]
Finally, the thesis follows 
from Lemma~\ref{lem:Gron}
 and \eqref{eq:norm_ineq}. 
\end{proof}

In the following proposition we prove that the solution of the Cauchy
problem \eqref{eq:ctrl_Cau}
has a continuous dependence
on the admissible control.

\begin{proposition} \label{prop:cont_dep_traj}
Let us consider $u,v\in \U$ and let
$x_u, x_{u+v}:[0,1]\to\R^n$ be the
solutions of the
Cauchy problem \eqref{eq:ctrl_Cau}
corresponding, respectively,
to the controls $u$ and
$u+v$. Then, for every $R>0$ there exists
$L_R>0$ such that the inequality 
\begin{equation}\label{eq:cont_dep_traj}
||x_{u+ v} - x_{u}||_{C^0}\leq L_{R}
 ||v||_{L^2}
\end{equation}
holds
for every $u,v\in\U$ such that 
$||u||_{L^2},||v||_{L^2}\leq R$.
\end{proposition}
\begin{proof}
Using the fact that $x_u$ and $x_{u+ v}$ are
solutions of \eqref{eq:ctrl_Cau},
for every $s\in[0,1]$
we have that
\begin{align*}
|x_{u+ v}(s)-x_u(s)|_2 &\leq
 \int_0^s 
\sum_{i=1}^k
\Big( | F^i(x_{u+ v}(\tau))|_2 |v^i(\tau)|
\Big)\,d\tau \\
& \quad
+ \int_0^s \sum_{i=1}^k\Big(|F^i(x_{u+ v}(\tau))-F^i(x_{u}(\tau)|_2 |u^i(\tau)|\Big)
\,d\tau.
\end{align*}
Recalling that $||v||_{L^2}\leq R$, in virtue of
Lemma~\ref{lem:C0_bound_traj},
we obtain that there exists
$ C_{R}>0$ such that
\[
\sup_{\tau\in[0,1]} \, \sup_{i=1,\ldots,k} 
| F^i(x_{u+ v}(\tau))|_2 \leq  C_{R}.
\]
Hence, using \eqref{eq:norm_ineq}, we deduce that
\begin{equation} \label{eq:cont_dep_p1}
\int_0^s 
\sum_{i=1}^k
\Big( | F^i(x_{u+ v}(\tau))|_2 |v^i(\tau)|
\Big)\,d\tau \leq 
 C_{R} \sqrt{k} ||v||_{L^2}.
\end{equation}
On the other hand, from the Lipschitz-continuity
condition \eqref{eq:Lipsch_const_F}
 it follows that
\begin{equation} \label{eq:cont_dep_p2}
|F^i(x_{u+ v}(\tau))-F^i(x_{u}(\tau)|_2
 \leq 
L  |x_{u+ v}(\tau)-x_u(\tau)|_2
\end{equation}
for every $i=1,\ldots,k$ and for every 
$\tau\in[0,1]$.
Using \eqref{eq:cont_dep_p1} and 
\eqref{eq:cont_dep_p2}, we deduce that
\begin{equation} \label{eq:cont_dep_p3}
|x_{u+ v}(s)-x_u(s)|_2  \leq
 C_{R}\sqrt k ||v||_{L^2}
+ L \int_0^s |u(\tau)|_1|x_{u+ v}(\tau)-x_u(\tau)|_2 \,d\tau,
\end{equation}
for every $s\in[0,1]$.
By applying Lemma~\ref{lem:Gron}
to \eqref{eq:cont_dep_p3}, we obtain that
\begin{equation*}
|x_{u+ v}(s)-x_u(s)|_2
\leq  e^{L||u||_{L^1}}
 C_{R}\sqrt k ||v||_{L^2},
\end{equation*}
for every $s\in[0,1]$.
Recalling \eqref{eq:norm_ineq} and setting
\[
L_{R}:= e^{L\sqrt k R}
 C_{R} \sqrt k,
\]
we prove \eqref{eq:cont_dep_traj}.
\end{proof}
The previous result shows that the map
$u\mapsto x_u$ is Lipschitz-continuous
when restricted to any bounded set of
the space of admissible controls $\U$.
We remark that Proposition~\ref{prop:cont_dep_traj}
holds under the sole assumption that
the controlled vector fields 
$F^1,\ldots,F^k:\R^n\to\R^n$
are Lipschitz-continuous.
In the next result, by
requiring that the controlled vector
fields are $C^1$-regular,
we compute the first order
variation of the solution
of \eqref{eq:ctrl_Cau}
resulting from a perturbation in the control.

\begin{proposition} \label{prop:diff_endpoint}
Let us assume that
the vector fields $F^1,\ldots,F^k$
defining the control system \eqref{eq:ctrl_Cau} are 
$C^1$-regular.
For every $u,v \in \U$, for every $\e\in(0,1]$,
let $x_u,x_{u+\e v}:[0,1]\to\R^n$ be the
solutions of \eqref{eq:ctrl_Cau} corresponding,
respectively, to the admissible controls
$u$ and $u+\e v$.
Then, we have that
\begin{equation} \label{eq:app_traj_1ord}
||x_{u+\e v} - x_u - \e y^v_u||_{C^0} = 
o(\e)
\mbox{ as } \e\to0,
\end{equation}
where $y_u^v:[0,1]\to \R^n$ is the
solution of the following affine system:
\begin{equation} \label{eq:var_1ord}
\dot y_u^v(s) = F(x_u(s))v(s)+
\left( \sum_{i=1}^k u^i(s)
\frac{\partial F^i(x_u(s))}{\partial x} 
\right) y^v_u(s)
\end{equation}
for a.e. $s\in[0,1]$, and
with $y^v_u(0)=0$.
\end{proposition}
\begin{proof}
Setting $R:=||u||_{L^2}+||v||_{L^2}$,
we observe that 
$||u+\e v||_{L^2} \leq R$
for every $\e\in(0,1]$.
Owing to Lemma~\ref{lem:C0_bound_traj}, we deduce
that there exists a compact $K_{R}\subset\R^n$
such that $x_u(s),x_{u+\e v}(s) \in K_{R}$
for every $s\in[0,1]$ and for every $\e\in(0,1]$.
Using the fact that $F^1,\ldots,F^k$ are assumed to 
be $C^1$-regular, we deduce that they are
uniformly continuous on $K_{R}$. 
This is equivalent to say that 
there exists a non-decreasing
function $\delta:[0,+\infty)\to[0,+\infty)$
such that $\delta(0) = \lim_{r\to 0} \delta(r)=0$
and
\begin{equation*} 
\left| \frac{\partial F^i(x_1)}{\partial x}
- \frac{\partial F^i(x_2)}{\partial x}
\right|_2 \leq \delta(|x_1-x_2|)
\end{equation*}
for every $x_1,x_2\in K_{R}$ and
for every $i=1,\ldots,k$. This
fact implies that there
exists a constant $C>0$ such that for every
$i=1,\ldots, k$ and for every $x_1,x_2\in K_{R}$
the following inequality is satisfied:
\begin{equation} \label{eq:unif_err_tayl}
\left| F^i(x_2) -F^i(x_1)
- \frac{\partial F^i(x_1)}{\partial x}  
(x_2-x_1)
\right|_2 \leq C \delta(|x_1-x_2|)|x_1-x_2|.
\end{equation}
Let us consider the non-autonomous 
affine system \eqref{eq:var_1ord}.
In virtue of Carath\'eodory Theorem
(see \cite[Theorem~5.3]{H80}), we deduce that
the system \eqref{eq:var_1ord} admits
a unique absolutely continuous solution
$y_u^v:[0,1]\to\R^n$.
For every $s\in[0,1]$, let us define
\begin{equation} \label{eq:aux_def_z}
\xi(s):=  x_{u+\e v}(s) -  x_u(s) - \e 
 y_u^v(s).
\end{equation}
Therefore, in view of \eqref{eq:ctrl_Cau}
and \eqref{eq:var_1ord},
for a.e. $s\in[0,1]$ we compute
\begin{align*}
|\dot \xi(s)|_2 \leq &
 \e\sum_{i=1}^k|F^i(x_{u+ \e v}(s))
  - F^i(x_u(s))|_2
|v^i(s)| \\
& \quad  +
\sum_{i=1}^k\left|
F^i(x_{u+ \e v}(s)) - F^i(x_u(s))
- \e \frac{\partial F^i(x_u(s))}{\partial x}
y_u^v(s) \right|_2 |u^i(s)|
\end{align*}
On one hand, using 
Proposition~\ref{prop:cont_dep_traj}
and the Lipschitz-continuity assumption
\eqref{eq:Lipsch_const_F}, we deduce that
there exists $L'>0$ such that
\begin{equation} \label{eq:1ord_comp_1}
\e\sum_{i=1}^k|F^i(x_{u+ \e v}(s)) - F^i(x_u(s))|_2
\leq L'||v||_{L^2} \e^2
\end{equation}
for every $s\in[0,1]$ and for every
$\e\in(0,1]$.
On the other hand, for every $i=1,\ldots,n$,
combining Proposition~\ref{prop:cont_dep_traj},
the inequality \eqref{eq:unif_err_tayl} and
the estimate of the norm of the Jacobian 
\eqref{eq:bound_jac_F},
we obtain that there exists $L''>0$ such that
\begin{align*} 
\Bigg|
F^i(x_{u+ \e v}(s)) - &F^i(x_u(s)) 
- \e \frac{\partial F^i(x_u(s))}{\partial x}y_u^v(s) 
\Bigg|_2 \\
&\leq \Bigg|F^i(x_{u+ \e v}(s)) - F^i(x_u(s)) 
-  \frac{\partial F^i(x_u(s))}{\partial x}
\big(x_{u+ \e v}(s)-x_{u}(s)\big) 
\Bigg|_2\\
&\qquad +
\Bigg|\frac{\partial F^i(x_u(s))}{\partial x}
\big(x_{u+ \e v}(s)- x_{u}(s) -\e y^v_u(s) \big) 
\Bigg|_2\\
&\leq C\Big[ 
\delta(L''||v||_{L^2}\e)L''||v||_{L^2}\e
\Big] + L|\xi(s)|_2.
\end{align*}
for every $s\in[0,1]$ and for every $\e\in(0,1]$.
Combining the last inequality and 
\eqref{eq:1ord_comp_1}, it follows that
\begin{equation}
|\dot \xi(s)|_2 \leq 
L_R\e^2 + L_R |u(s)|_1  \delta(L_R\e)\e
+ L|u(s)|_1|\xi(s)|_2
\end{equation}
for a.e. $s\in[0,1]$ and for every $\e\in(0,1]$,
where we set $L_R:=\max\{ L', L'' \}||v||_{L^2}$.
Finally, recalling that 
$|\xi(0)|_2=|x_{u+\e v}(0) -x_u(0) -\e y_u^v(0)|_2=0$
for every $\e\in(0,1]$,
we have that
\[
|\xi(s)|_2\leq \int_0^s |\dot \xi(\tau)|_2 \,d\tau
\leq 
L_R\e^2 + L_R ||u||_{L^1}\delta(L_R\e)\e 
+ L \int_0^s |u(\tau)|_1|\xi(\tau)|_2 \,d\tau,
\]
for every $s\in[0,1]$ and for every
$\e\in (0,1]$.
Using Lemma~\ref{lem:Gron} and
\eqref{eq:aux_def_z}, we deduce
\eqref{eq:app_traj_1ord}.
\end{proof}

Let us assume that 
$F^1,\ldots,F^k$ are $C^1$-regular.
For every admissible control $u\in\U$,
let us define
$A_u\in L^2([0,1],\R^{n\times n})$
as 
\begin{equation} \label{eq:def_A}
A_u(s) := \sum_{i=1}^k\left(
u^i(s)\frac{\partial F^i(x_u(s))}{\partial x} \right)
\end{equation}
for a.e. $s\in[0,1]$. 
For every $u\in \U$, let us introduce 
the absolutely continuous curve
$M_u:[0,1]\to\R^{n\times n}$, 
defined as the solution of the 
following linear Cauchy problem:
\begin{equation} \label{eq:def_M}
\begin{cases}
\dot M_u(s) = A_u(s) M_u(s) &\mbox{for a.e. }
s\in[0,1], \\
M_u(0) = \mathrm{Id}.
\end{cases}
\end{equation}
The existence and uniqueness of the solution
of \eqref{eq:def_M} descends
once again from the Carath\'eodory Theorem.
We can prove the following result.

\begin{lemma} \label{lem:norm_M}
Let us assume that
the vector fields $F^1,\ldots,F^k$
defining the control system 
\eqref{eq:ctrl_Cau} are 
$C^1$-regular.
For every  admissible control $u\in \U$, 
let $M_u:[0,1] \to \R^{n\times n}$ be the 
solution of the Cauchy problem \eqref{eq:def_M}.
Then, for every $s\in[0,1]$, $M_u(s)$ is 
invertible, and the following estimates hold:
\begin{equation} \label{eq:norm_M_Minv_est}
|M_u(s)|_2 \leq C_u, \quad
|M_u^{-1}(s)|_2 \leq C_u,
\end{equation}
where
\[
C_u= e^{\sqrt k L ||u||_{L^2}}.
\]
\end{lemma}
\begin{proof}
Let us consider the absolutely continuous curve
$N_u:[0,1]\to\R^{n\times n}$ that solves
\begin{equation}\label{eq:def_inv_M}
\begin{cases}
\dot N_u(s) = -N_u(s)A_u(s) &\mbox{for a.e. }
s\in[0,1], \\
N_u(0) = \mathrm{Id}. 
\end{cases}
\end{equation}
The existence
and uniqueness of the solution of
\eqref{eq:def_inv_M} is guaranteed 
by Carath\'eodory Theorem.
Recalling the Leibniz rule for 
Sobolev functions (see, e.g., 
\cite[Corollary~8.10]{B11}),
a simple computation shows that 
the identity
$
N_u(s)M_u(s) = \mathrm{Id}
$
holds
for every $s\in[0,1]$.
This proves that $M_u(s)$ is invertible
and that $N_u(s) = M_u^{-1}(s)$
for every $s\in[0,1]$.
In order to prove the bound on the norm of
the matrix $M_u(s)$, we shall study
$|M_u(s)z|_2$, for $z\in\R^n$. 
Using \eqref{eq:def_M}, we deduce that
\begin{align*}
|M_u(s)z|_2 &\leq  |z|_2 + \int_0^s 
|A_u(\tau)|_2|M_u(\tau)z|_2 \,d\tau \\
& \leq |z|_2 + L\int_0^s
|u(s)|_1|M_u(\tau)z|_2 \,d\tau,
\end{align*}
where we used \eqref{eq:bound_jac_F}.
Using Lemma~\ref{lem:Gron}, and recalling
\eqref{eq:norm_mat_def} and \eqref{eq:norm_ineq},
we obtain that the
inequality \eqref{eq:norm_M_Minv_est}
holds for $M_u(s)$, for every $s\in[0,1]$.
Using \eqref{eq:def_inv_M} and applying the 
same argument, it is possible to prove that
\eqref{eq:norm_M_Minv_est} holds as well for
$N_u(s)=M_u^{-1}(s)$, for every $s\in[0,1]$.
\end{proof}

Using the curve $M_u:[0,1]\to \R^{n\times n}$
defined by \eqref{eq:def_M},
we can rewrite the solution of the
affine system \eqref{eq:var_1ord} for
the first-order variation of the trajectory.
Indeed, for every $u,v \in \U$,
a direct computation shows that the 
function $y_u^v:[0,1]\to \R^n$ that solves
 \eqref{eq:var_1ord}
can be expressed as
\begin{equation} \label{eq:dec_var_1ord}
y^v_u(s)= 
\int_0^s 
M_u(s)M_u^{-1}(\tau)F(x_u(\tau))v(\tau)
\,d\tau
\end{equation}
for every $s\in[0,1]$.
Using \eqref{eq:dec_var_1ord}
we can prove
an estimate
of the norm of $y^v_u$.

\begin{lemma} \label{lem:est_y}
Let us assume that
the vector fields $F^1,\ldots,F^k$
defining the control system 
\eqref{eq:ctrl_Cau} are 
$C^1$-regular.
Let us consider $u,v\in \U$, and let
$y_u^v:[0,1]\to \R^n$ be the solution
of the affine system \eqref{eq:var_1ord}
with $y_u^v(0)=0$.
Then, for every $R>0$ there exists $C_R>0$
such that the following inequality holds
\begin{equation} \label{eq:est_y}
|y_u^v(s)|_2 \leq  C_R||v||_{L^2}
\end{equation}
for every $s\in[0,1]$ and for every 
$u\in\U$ satisfying $||u||_{L^2}\leq R$.
\end{lemma}
\begin{proof}
In virtue of \eqref{eq:dec_var_1ord},
we have that 
\begin{equation*}
|y_u^v(s)|_2 \leq \int_0^s\left| 
M_u(s) M_u^{-1}(\tau)
F(x_u(\tau))v(\tau) \right|\,d\tau.
\end{equation*}
Using \eqref{eq:norm_M_Minv_est}, 
\eqref{eq:est_norm_traj} and 
\eqref{eq:sub_lin},
we deduce that there exists $C_R'>0$ such that
\[
|y_u^v(s)|_2 \leq C_R'\int_0^s|v(s)|_1 \,d\tau,
\]
for every $s\in[0,1]$.
Combining this with \eqref{eq:norm_ineq},
we deduce the thesis.
\end{proof}

Let us introduce the end-point map
associated to the control system 
\eqref{eq:ctrl_Cau}. 
For every $s\in [0,1]$, let us consider the
map $P_s:\U\to\R^n$ defined as
\begin{equation} \label{eq:def_end_p_map}
P_s:u\mapsto P_s(u) :=x_u(s),
\end{equation}
where $x_u:[0,1]\to\R^n$ is the solution
of \eqref{eq:ctrl_Cau} corresponding to the
admissible control $u\in\U$.
Using the results obtained before, it follows that
the end-point map is differentiable.

\begin{proposition} \label{prop:rep_diff}
Let us assume that
the vector fields $F^1,\ldots,F^k$
defining the control system \eqref{eq:ctrl_Cau} are 
$C^1$-regular.
For every $s\in [0,1]$, let $P_s:\U\to\R^n$ be the
end-point map defined by 
\eqref{eq:def_end_p_map}. Then, for every 
$u\in \U$, $P_s$ is Gateaux differentiable 
at $u$, and the differential
$D_u P_s = (D_u P_s^1,\ldots,D_uP_s^n):\U\to
\R^n $ is a linear and continuous operator. 
Moreover, using the Riesz's isometry,
for every $u\in \U$
and for every $s\in[0,1]$,
every component of the differential $D_uP_s$
can be represented as follows:
\begin{equation} \label{eq:diff_end_int} 
D_uP^j_s(v) = \int_0^1 
\left\langle g_{s,u}^j(\tau),v(\tau) \right\rangle_{\R^k} \,d\tau,
\end{equation}
where, for every $j=1,\ldots,n$, the function
$g_{s,u}^j:[0,1]\to\R^k$ is defined as
\begin{equation}\label{eq:diff_end_rep}
g_{s,u}^j(\tau) = 
\begin{cases}
\Big( (\mathbf{e}^j)^TM_u(s)M^{-1}_u(\tau)
F(x_u(\tau))
\Big)^T &\tau\in[0,s],\\
0 &\tau\in(s,1],
\end{cases}
\end{equation}
where the column vector $\mathbf{e}^j$ is the
$j$-th element of the standard basis
$\{ \mathbf{e}^1,\ldots,\mathbf{e}^n \}$ of $\R^n$.
\end{proposition}
\begin{proof}
For every $s\in[0,1]$, 
Proposition~\ref{prop:diff_endpoint}
 guarantees
that the end-point map $P_s:\U\to\R^n$
is Gateaux differentiable at every
point $u\in\U$.
In particular, for every $u,v\in \U$ and for
every $s\in[0,1]$  the following identity holds:
\begin{equation} \label{eq:diff_end_y}
D_uP_s(v) = y_u^v(s).
\end{equation}
Moreover, \eqref{eq:dec_var_1ord} 
 shows that the differential
$D_uP_s:\U\to\R^n$ is linear, 
and Lemma~\ref{lem:est_y} implies that
it is continuous.
The representation follows as well from
\eqref{eq:dec_var_1ord}.
\end{proof}

\begin{remark} \label{rmk:unif_bound_diff}
In the previous proof we used 
Lemma~\ref{lem:est_y} to deduce 
for every $u\in\U$
the continuity of the linear
operator
$D_uP_s:\U\to\R^n$. Actually, 
Lemma~\ref{lem:est_y} is slightly more informative,
since it implies that for every $R>0$
there exists $C_R>0$ such that
\begin{equation} \label{eq:unif_est_dif_endp}
|D_uP_s(v)|_2 \leq C_R||v||_{L^2}
\end{equation}
for every $v\in \U$
and for every $u\in \U$ such that
$||u||_{L^2}\leq R$.
As a matter of fact, we deduce that
\begin{equation} \label{eq:unif_est_rep_dif_endp}
||g^j_{s,u}||_{L^2}\leq C_R
\end{equation}
for every $j=1,\ldots,n$, for every
$s\in[0,1]$ and for every $u\in \U$ such that
$||u||_{L^2}\leq R$.
\end{remark}

\begin{remark} \label{rmk:reg_rep_diff_enpoint}
It is interesting to observe that, for
every $s\in(0,1]$ and for every $u\in\U$,
the function $g_{s,u}^j:[0,1]\to\R^k$
that provides the representation
 the $j$-th component of
$D_uP_s$ is absolutely continuous on the
interval $[0,s]$, being the product
of absolutely continuous matrix-valued curves.
Indeed, on one hand,
$\tau \mapsto F(x_u(\tau))$ is 
absolutely continuous, being the composition of
a $C^1$-regular function with the absolutely
continuous curve $\tau\mapsto x_u(\tau)$
(see, e.g., \cite[Corollary~8.11]{B11}).
On the other hand,
$\tau\mapsto M_u^{-1}(\tau)$
is absolutely continuous as well, since
it solves \eqref{eq:def_inv_M}.   
\end{remark}

We now prove that for every $s\in[0,1]$
the differential of the
end-point map $u\mapsto D_uP_s$ is
Lipschitz-continuous
on the bounded subsets of $\U$. 
This result requires further regularity
assumptions on the controlled vector fields.
We first establish an auxiliary result concerning
 the matrix-valued curve that
solves \eqref{eq:def_M}.

\begin{lemma} \label{lem:lip_dep_M}
Let us assume that
the vector fields $F^1,\ldots,F^k$
defining the control system 
\eqref{eq:ctrl_Cau} are 
$C^2$-regular.
For every $u,w\in \U$, 
let $M_{u},M_{u+w}:[0,1]\to \R^{n\times n}$ 
be the solutions of \eqref{eq:def_M}
corresponding to the admissible controls
$u$ and $u+w$, respectively.
Then, for every $R>0$ there exists $L_R>0$
such that, for every $u,w\in \U$ satisfying
$||u||_{L^2},||w||_{L^2}\leq R$, we have
\begin{equation} \label{eq:lips_M}
|M_{u+w}(s)-M_u(s)|_2 \leq L_R ||w||_{L^2},
\end{equation}
and
\begin{equation} \label{eq:lips_M_inv}
|M_{u+w}^{-1}(s)-M_u^{-1}(s)|_2 \leq L_R ||w||_{L^2}
\end{equation}
for every $s\in[0,1]$.
\end{lemma}
\begin{proof}
Let us consider $R>0$, and let 
$u,w\in \U$ be such that 
$||u||_{L^2},||w||_{L^2}\leq R$.
We observe that 
Lemma~\ref{lem:C0_bound_traj} implies that
there exists a compact set $K_R\subset \R^n$
such that $x_u(s),x_{u+w}(s)\in K_R$ for every
$s\in[0,1]$. The hypothesis that
$F^1,\ldots,F^2$ are $C^2$-regular implies that
there exists $L_R'>0$ such that
$\frac{\partial F^1}{\partial x},\ldots,
\frac{\partial F^k}{\partial x}$ are 
Lipschitz-continuous in $K_R$ with constant 
$L_R'$.
From \eqref{eq:def_M}, we have that
\begin{equation} \label{eq:lips_M_1}
|\dot M_{u+w}(s)- \dot M_u(s)|_2 = 
|A_{u+w}(s)M_{u+w}(s)-A_u(s)M_{u}(s)|_2,
\end{equation}
for a.e. $s\in[0,1]$.
In particular, for a.e. $s\in[0,1]$, we can compute
\begin{align*}
|A_{u+w}(s)-A_u(s)|_2 &\leq 
\sum_{i=1}^k\left|
\frac{\partial F^i(x_{u+w}(s))}{\partial x} 
-\frac{\partial F^i(x_{u}(s))}{\partial x}
\right|_2 |u^i(s)|\\
& \qquad + 
\sum_{i=1}^k\left|
\frac{\partial F^i(x_{u+w}(s))}{\partial x}
\right|_2 |w^i(s)|,
\end{align*}
and using Proposition~\ref{prop:cont_dep_traj},
the Lipschitz continuity of
$\frac{\partial F^1}{\partial x},\ldots,
\frac{\partial F^k}{\partial x}$
and \eqref{eq:bound_jac_F},
we obtain that there exists
$L_R''>0$ such that
\begin{equation} \label{eq:lips_M_2}
|A_{u+w}(s)-A_u(s)|_2 
\leq L_R''  ||w||_{L^2} |u(s)|_1
+ L|w(s)|_1,
\end{equation}
for a.e. $s\in[0,1]$. Using once again 
\eqref{eq:bound_jac_F}, we have that
\begin{equation} \label{eq:lips_M_3}
|A_u(s)|_2 \leq L|u(s)|_1,
\end{equation}
for a.e. $s\in[0,1]$.
Combining \eqref{eq:lips_M_2}-\eqref{eq:lips_M_3}
with the triangular inequality at the
right-hand side of \eqref{eq:lips_M_1},
we deduce that
\begin{align*} 
|\dot M_{u+w}(s)- \dot M_u(s)|_2 \leq & 
C'_R\big( L_R'' ||w||_{L^2} |u(s)|_1 
 + L|w(s)|_1\big)\\
& \quad + L|u(s)|_1|M_{u+w}(s)-M_u(s)|_2,
\end{align*}
for a.e. $s\in[0,1]$, where we used 
Lemma~\ref{lem:norm_M}
to deduce that there exists $C'_R>0$ such that
$|M_{u+w}(s)| \leq C'_R$ for every $s\in[0,1]$.
Recalling that the Cauchy datum of
\eqref{eq:def_M} prescribes
$M_{u+w}(0)=M_u(0)=\mathrm{Id}$,
the last inequality yields
\begin{align*}
|M_{u+w}(s)-M_u(s)|_2 & \leq
\int_0^s |\dot M_{u+w}(\tau)-\dot M_u(\tau)|_2
 \,d\tau \\
&\leq C''_R||w||_{L^2} + L
\int_0^s |u(s)|_1|M_{u+w}(\tau)-M_u(\tau)|_2
 \,d\tau,
\end{align*}
for every $s\in[0,1]$,
where we used \eqref{eq:norm_ineq}
and where $C''_R>0$ is a constant
depending only on $R$.
Finally, Lemma~\ref{lem:Gron} implies 
the first inequality of the thesis.
Recalling that
$s\mapsto M^{-1}_u(s)$ and
$s\mapsto M^{-1}_{u+w}(s)$ are absolutely continuous
curves that solve \eqref{eq:def_inv_M}, 
repeating {\it verbatim} the same argument as 
above, we deduce the second inequality
of the thesis.
\end{proof}

We are now in position to prove the
regularity result on the differential
of the end-point map.
 
\begin{proposition} \label{prop:Lipsch_diff}
Let us assume that
the vector fields $F^1,\ldots,F^k$
defining the control system 
\eqref{eq:ctrl_Cau} are 
$C^2$-regular.
Then, for every $R>0$ there exists $L_R>0$
such that, for every
$u,w\in \U$ satisfying
$||u||_{L^2},||w||_{L^2}\leq R$,
 the following inequality holds 
\begin{equation} \
|D_{u+w}P_s(v) -D_uP_s(v)|_2 \leq 
L_R ||w||_{L^2}||v||_{L^2}
\end{equation}
for every $s\in[0,1]$ and for every $v\in \U$.
\end{proposition}
\begin{proof}
In virtue of Proposition~\ref{prop:rep_diff},
it is sufficient to prove that there
exists $L_R>0$ such that
\begin{equation} \label{eq:lipsc_cont_rep_endpoint}
||g_{s,u+w}^j - g_{s,u}^j||_{L^2} \leq 
L_R||w||_{L^2}
\end{equation}
for every $j=1,\ldots,n$ and for every $u,w\in \U$
such that $||u||_{L^2},||w||_{L^2}\leq R$, where
$g_{s,u+w}^j, g_{s,u}^j$ are defined as in 
\eqref{eq:diff_end_rep}. 
Let us consider $u,w\in \U$ satisfying
$||u||_{L^2},||w||_{L^2}\leq R$.
The inequality \eqref{eq:lipsc_cont_rep_endpoint} 
will in turn follow if we show that there exists 
a constant $L_R>0$ such that
\begin{equation} \label{eq:lips_diff_goal}
|M_{u+w}(s)M^{-1}_{u+w}(\tau)F(x_{u+w}(\tau))
- M_{u}(s)M^{-1}_{u}(\tau)F(x_{u}(\tau))|_2
\leq L_R||w||_{L^2},
\end{equation}
for every $s\in[0,1]$, for every $\tau\in[0,s]$
and for every $u,w\in \U$
that satisfy $||u||_{L^2},||w||_{L^2}\leq R$.
Owing to 
Proposition~\ref{prop:cont_dep_traj} and
\eqref{eq:Lipsch_const_F},
 it follows that
there exists $L'_R>0$ such that
\begin{equation} \label{eq:lips_F}
|F(x_{u+w}(s))-F(x_u(s))|_2
 \leq L'_R||w||_{L^2},
\end{equation}
for every $s\in[0,1]$ and
for every $u,w\in \U$
satisfying $||u||_{L^2},||w||_{L^2}\leq R$.
Using the triangular inequality
in \eqref{eq:lips_diff_goal}, we compute
\begin{align*}
|M_{u+w}(s)M^{-1}_{u+w}(\tau)&F(x_{u+w}(\tau))
- M_{u}(s)M^{-1}_{u}(\tau)F(x_{u}(\tau))|_2\\
&\leq |M_{u+w}(s)- M_{u}(s)|_2
|M^{-1}_{u+w}(\tau)|_2|F(x_{u+w}(\tau))|_2\\
&\quad +|M_u(s)|_2
|M^{-1}_{u+w}(\tau)-M^{-1}_{u}(\tau)|_2|F(x_{u+w}(\tau))|_2\\
&\quad +|M_u(s)|_2
|M^{-1}_{u}(\tau)|_2| F(x_{u+w}(\tau))-
F(x_{u}(\tau))|_2
\end{align*}
for every $s\in[0,1]$ and for every $\tau\in[0,s]$.
Using \eqref{eq:lips_F}, Lemma~\ref{lem:norm_M}
and Lemma~\ref{lem:lip_dep_M}
in the last inequality, we deduce
that \eqref{eq:lips_diff_goal} holds.
This concludes the proof. 
\end{proof}

\end{subsection}

\begin{subsection}{Second-order differential of the
end-point map} \label{subsec:2nd_ord}
In this subsection we study the second-order 
variation of the end-point map 
$P_s:\U\to\R^n$ defined in
\eqref{eq:def_end_p_map}.
The main results reported here will be stated
in the case $s=1$, which
corresponds to the final evolution instant
of the control system \eqref{eq:ctrl_Cau}.
However, they can be extended
(with minor adjustments)
also in the case $s\in(0,1)$.
Similarly as done in Subsection~\ref{subsec:1ord},
we show that, under proper regularity assumptions
on the controlled vector fields $F^1,\ldots,F^k$,
the end-point map $P_1:\U\to \R^n$ is
$C^2$-regular. Therefore, for every
$u\in \U$ we can consider the second differential
$D_u^2 P_1:\U\times \U \to \R^n$, which 
turns out to be a 
bilinear and symmetric operator. 
For every $\nu \in \R^n$, we provide a 
representation of the bilinear form
$\nu \cdot D_u^2 P_1:\U\times \U \to \R$, 
and we prove that it is a compact
self-adjoint operator.

Before proceeding, we introduce some notations.
We define 
$\V:=L^2([0,1],\R^n)$, and we equip it with the
usual Hilbert space structure. In order to avoid
confusion, in the present subsection we 
denote with $||\cdot ||_\U$ and $||\cdot||_\V$
the norms of the Hilbert spaces $\U$ and
$\V$, respectively. We use a similar convention
for the respective scalar products, too.
Moreover, given an application 
$\mathcal{R}:\U\to\V$, for every $u\in \U$ we 
use the notation $\mathcal{R}[u]\in \V$ to denote
the image of $u$ through
$\mathcal{R}$. Then, for $s\in[0,1]$, we write 
$\mathcal{R}[u](s)\in\R^n$ to refer to
the value of 
(a representative of) the function $\mathcal{R}[u]$
at the point $s$. More generally,
we adopt this convention for every function-valued
operator.

It is convenient to introduce a linear operator 
that will be useful to derive the expression of the
second differential of the end-point map.
Assuming that  the controlled fields 
$F^1,\ldots,F^k$
are $C^1$-regular, for every $u\in\U$ we define
$\El_u:\U\to \V$ as follows:
\begin{equation} \label{eq:def_El}
\El_u[v](s) := y_u^v(s)
\end{equation}
for every $s\in[0,1]$, where $y_u^v:[0,1]\to\R^n$
is the curve introduced in 
Proposition~\ref{prop:diff_endpoint}
that solves the affine system \eqref{eq:var_1ord}.
Recalling \eqref{eq:dec_var_1ord}, 
we have that the identity
\begin{equation} \label{eq:def_El_int}
\El_u[v](s) = \int_0^s M_u(s) M_u^{-1}(\tau)
F(x_u(\tau)) v(\tau) \, d\tau
\end{equation} 
holds for every $s\in[0,1]$ and for every
$v\in\U$, and this shows 
that $\El_u$ is a linear operator. 
Moreover, using the standard Hilbert space 
structure of $\U$ and of $\V$,
for every $u\in \U$ we 
can introduce the adjoint of $\El_u$, namely the 
linear operator $\El_u^*:\V \to \U$ that satisfies
\begin{equation} \label{eq:adjoint_El}
\langle \El_u^*[w], v\rangle_{\U} =
\langle \El_u[v], w\rangle_{\V}
\end{equation}
for every $v\in \U$ and $w\in \V$.

\begin{remark} \label{rmk:norm_adj}
We recall a result in functional analysis 
concerning the norm of the adjoint of a bounded 
linear operator. For further details, see
\cite[Remark~2.16]{B11}. 
Given two Banach spaces 
$E_1,E_2$, let $\mathscr{L}(E_1,E_2)$ be the 
Banach 
space of the bounded linear operators from $E_1$
to $E_2$, equipped with the norm induced by
$E_1$ and $E_2$. 
Let $E_1^*, E_2^*$ be the dual spaces of 
$E_1, E_2$, respectively, and
let $\mathscr{L}(E_2^*,E_1^*)$ be defined as above.
Therefore, 
if ${A}\in \mathscr{L}(E_1,E_2)$,
then the adjoint operator satisfies
$A^* \in \mathscr{L}(E_2^*,E_1^*)$,
and the following identity holds:
\[
||A^*||_{\mathscr L (E_2^*,E_1^*)} = 
||A||_{\mathscr L (E_1,E_2)}.
\]
If $E_1,E_2$ are Hilbert spaces, using the 
Riesz's isometry it is possible to 
write $A^*$ as an element of 
$\mathscr{L}(E_2,E_1)$, and the
identity of the norms is still satisfied.
\end{remark}

We now show that $\El_u$ and $\El_u^*$ are 
bounded and compact operators.

\begin{lemma} \label{lem:bound_comp_El}
Let us assume that the vector fields 
$F^1,\ldots,F^k$ defining the control system 
\eqref{eq:ctrl_Cau} are $C^1$-regular.
Then, for every $u\in \U$, the linear operators
$\El_u:\U \to \V$ and 
$\El_u^*:\V\to \U$ defined, 
respectively, by \eqref{eq:def_El} and
\eqref{eq:adjoint_El} are bounded and compact.
\end{lemma}
\begin{proof}
It is sufficient to prove the statement for 
the operator $\El_u:\U\to \V$.
Indeed, if $\El_u$ is bounded and compact, then
$\El_u^*:\V\to \U$ is as well
Indeed, the boundedness of the adjoint
descends from 
Remark~\ref{rmk:norm_adj}, while the compactness
from \cite[Theorem~6.4]{B11}).
Using Lemma~\ref{lem:est_y} we obtain that, for
every $u\in \U$, there exists $C>0$ such that
the following inequality holds
\begin{equation} \label{eq:El_inf_norm}
||\El_u[v]||_{C^0} \leq C ||v||_{\U},
\end{equation}
for every $v\in \U$. Recalling the continuous 
inclusion 
$C^0([0,1],\R^n)\hookrightarrow \V$,
 we deduce that $\El_u$ is a continuous 
linear operator.
In view of Theorem~\ref{thm:comp_sob_imm},
in order to prove that $\El_u$ is compact,
it is sufficient 
to prove that, for every $u\in \U$,
there exists $C'>0$ such that
\begin{equation} \label{eq:El_cont_H1}
||\El_u[v]||_{H^1} \leq C'||v||_{\U}
\end{equation}
for every $v\in \U$.
However, from the definition
of $\El_u[v]$ given in \eqref{eq:def_El}, it
follows that    
\[
\frac{d}{ds}\El_u[v](s) = \dot y_u^v(s)
\]
for a.e. $s\in[0,1]$. Therefore, from 
\eqref{eq:var_1ord} and 
Lemma~\ref{lem:est_y}, we deduce that 
\eqref{eq:El_cont_H1} holds.
\end{proof}

In the next result we study
the local Lipschitz-continuity
of the correspondence
$u\mapsto \El_u$.

\begin{lemma} \label{lem:lipsh_El}
Let us assume that the vector fields 
$F^1,\ldots,F^k$ defining the control system 
\eqref{eq:ctrl_Cau} are $C^2$-regular.
Then, for every $R>0$ there exists $L_R>0$
such that
\begin{equation} \label{eq:lip_El}
||\El_{u+w}[v] - \El_u[v]||_{\V}
\leq  L_R ||w||_{\U} ||v||_{\U}
\end{equation}
for every $v\in \U$ and for every $u,w\in \U$ such 
that $||u||_{\U},||w||_{\U}\leq R$.
\end{lemma}
\begin{proof}
Recalling the continuous inclusion
$C^0([0,1],\R^n)\hookrightarrow
\V$, it is sufficient to prove that
for every $R>0$ there exists $L_R>0$
such that, for every $s\in[0,1]$, the following 
inequality is satisfied
\begin{equation} \label{eq:lip_El_prov}
|\El_{u+w}[v](s) - \El_u[v](s)|_2
\leq  L_R ||w||_{\U} ||v||_\U
\end{equation}
for every $v\in \U$ and for every $u,w\in \U$ such 
that $||u||_\U,||w||_\U\leq R$.
On the other hand, \eqref{eq:def_El_int} implies 
that
\begin{align*}
|\El_{u+w}&[v](s) - \El_u[v](s)|_2
\\ &\leq \int_0^s | M_{u+w}(s) M^{-1}_{u+w}(\tau)
F(x_{u+w}(\tau)) -
M_{u}(s) M^{-1}_{u}(\tau)
F(x_{u}(\tau))|_2 |v(\tau)|_2 \, d\tau. 
\end{align*}
However, using 
Proposition~\ref{prop:cont_dep_traj},
Lemma~\ref{lem:norm_M} and 
Lemma~\ref{lem:lip_dep_M},
we obtain that there exists $L'_R>0$ such that
\begin{equation*}
| M_{u+w}(s) M^{-1}_{u+w}(\tau)
F(x_{u+w}(\tau)) -
M_{u}(s) M^{-1}_{u}(\tau)
F(x_{u}(\tau))|_2
\leq L'_R ||w||_\U
\end{equation*} 
for every $s,\tau \in [0,1]$ and for every
$u,w\in \U$ such that 
$||u||_\U,||w||_\U\leq R$.
Combining the last two inequalities, we deduce
that \eqref{eq:lip_El_prov} holds. 
\end{proof}

\begin{remark}\label{rmk:lip_El_adj}
From Lemma~\ref{lem:lipsh_El} and
Remark~\ref{rmk:norm_adj} it follows that
the correspondence $u\mapsto \El_u^*$
is as well Lipschitz-continuous on the
bounded sets of $\U$.
\end{remark}

If the vector fields $F^1,\ldots, F^k$ are
$C^2$-regular, we write 
$\frac{\partial^2 F^1}{\partial x^2},\ldots,
\frac{\partial^2 F^k}{\partial x^2}$ to denote
their second differential, i.e., for every
$i=1,\ldots, k$ and for every $y\in \R^n$,
the application
$\frac{\partial^2 F^k(y)}{\partial x^2}:
\R^n\times \R^n \to \R^n$ is the symmetric and
bilinear operator that satisfies
\[
F^i(y+h) - F^i(y) - \frac{\partial F^i(y)}{\partial x}
(h) - \frac12 \frac{\partial^2 F^i(y)}{\partial x^2}(h,h) = o(|h|_2^2)
\]
as $|h|_2\to 0$. 

In the next result we investigate the second-order
variation of the solutions produced by
the control system \eqref{eq:ctrl_Cau}.

\begin{proposition}\label{prop:2ord_var}
Let us assume that the vector fields 
$F^1,\ldots,F^k$ defining the control system 
\eqref{eq:ctrl_Cau} are $C^2$-regular.
For every $u,v,w \in \U$, for every 
$\e\in (0,1]$, let 
$y_u^v, y_{u+\e w}^v:[0,1]\to \R^n$
be the solutions of \eqref{eq:var_1ord}
corresponding to the
first-order variation $v$ and
to the admissible
controls $u$ and $u+\e w$, respectively. 
Therefore, we have that
\begin{equation} \label{eq:2nd_diff}
\sup_{||v||_{L^2}\leq 1}
||y_{u+\e w}^v - y_u^v - \e z_u^{v,w}||_{C^0}
= o(\e)\mbox{ as } \e\to 0,
\end{equation}
where $z_u^{v,w}:[0,1]\to\R^n$ is the solution of 
the following affine system:
\begin{align} \label{al:def_z_1} 
\dot z_u^{v,w}(s) =& 
\sum_{i=1}^k \left[
v^i(s) \frac{\partial F^i(x_u(s))}{\partial x} 
y_u^w(s) 
+
w^i(s) \frac{\partial F^i(x_u(s))}{\partial x} 
y_u^v(s) \right] \\
& + \sum_{i=1}^k \label{al:def_z_2}
 u^i(s) 
\frac{\partial^2 F^i(x_u(s))}{\partial x^2}
(y_u^v(s),y_u^w(s))\\
&+ \sum_{i=1}^k u^i(s)
\frac{\partial F^i(x_u(s))}{\partial x} 
z_u^{v,w}(s) \label{al:def_z_3}
\end{align}
with $z_u^{v,w}(0)=0$, and where
$y_u^v,y_u^w:[0,1]\to\R^n$ are the solutions of
\eqref{eq:var_1ord} corresponding to the 
admissible control $u$ and to the first-order
variations $v$ and $w$, respectively. 
\end{proposition}
\begin{proof}
The proof of this result follows using the same
kind of techniques and computations as in the 
proof of Proposition~\ref{prop:diff_endpoint}.
\end{proof}

\begin{remark} \label{rmk:int_z}
Similarly as done in \eqref{eq:dec_var_1ord} for
the first-order variation, we can express the
solution of the affine system
\eqref{al:def_z_1}-\eqref{al:def_z_3}
through an integral formula. Indeed, for
every $u,v,w\in \U$, for every
$s\in[0,1]$ we have that
\begin{align}
z_u^{v,w}(s) = \label{eq:int_z_1}
\int_0^s 
M_u(s) M_u^{-1}(\tau)  &\left( 
\sum_{i=1}^k 
v^i(\tau) \frac{\partial F^i(x_u(\tau))}{\partial x} 
\El_u[w](\tau)  \right.  \\
& \quad + \sum_{i=1}^k \label{eq:int_z_2}
w^i(\tau) \frac{\partial F^i(x_u(\tau))}{\partial x} 
\El_u[v](\tau) \\ 
& \quad \left. + \sum_{i=1}^k 
 u^i(\tau) \label{eq:int_z_3}
\frac{\partial^2 F^i(x_u(\tau))}{\partial x^2}
(\El_u[v](\tau),\El_u[w](\tau))\right)
d\tau,
\end{align}
where we used the linear operator 
$\El_u:\U\to\V$ defined in \eqref{eq:def_El}.
From the previous expression it follows that, 
for every $u, v, w\in \U$, the roles of
$v$ and $w$ are interchangeable, i.e., 
for every $s\in[0,1]$ we have that 
$z_u^{v,w}(s) = z_u^{w,v}(s)$.
Moreover, we observe that, for every $s\in[0,1]$
and for every $u\in \U$, $z_u^{v,w}(s)$ is bilinear
with respect to $v$ and $w$.
\end{remark}

We are now in position to introduce the second 
differential of the end-point map 
$P_s:\U\to \R^n$ defined in 
\eqref{eq:def_end_p_map}. In view of the 
applications in the forthcoming sections, we
shall focus on the case $s=1$, i.e., we 
consider the map $P_1:\U\to\R^n$.
Before proceeding, for every $u\in \U$
we define the symmetric
and bilinear map $\mathcal B_u:\U\times \U
\to \R^n$ as follows
\begin{equation} \label{eq:def_B}
\mathcal{B}_u(v,w):= z_u^{v,w}(1).
\end{equation}

\begin{proposition} \label{prop:second_diff_endp}
Let us assume that the vector fields
$F^1,\ldots,F^k$ defining the control system
\eqref{eq:ctrl_Cau} are $C^2$-regular.
Let $P_1:\U\to \R^n$ be the end-point map
defined by \eqref{eq:def_end_p_map}, and, 
for every $u\in \U$, let $D_uP_1:\U \to \R^n$
be its differential.
Then, the correspondence $u \mapsto D_uP_1$
is Gateaux differentiable at every $u\in \U$,
namely
\begin{equation}\label{eq:G_der_diff}
\lim_{\e \to 0} \, 
\sup_{||v||_{L^2}\leq 1} \left|
\frac{D_{u+\e w}P_1(v) - D_uP_1(v)}{\e}
- \mathcal{B}_u(v,w)
\right|_2 =0,
\end{equation}
where $B_u:\U\times \U \to \R^n$ is the
bilinear, symmetric and bounded operator 
defined in \eqref{eq:def_B}. 
\end{proposition}
\begin{proof}
In view of \eqref{eq:diff_end_y},
for every $u,v,w\in \U$ and for every $\e\in(0,1]$, 
we have that $D_uP_1(v)=y_u^v(1)$ and
$D_{u+\e w}P_1(v)=y_{u+\e w}^v(1)$. 
Therefore, \eqref{eq:G_der_diff} follows directly
from \eqref{eq:2nd_diff} and from 
\eqref{eq:def_B}. 
The symmetry and the bilinearity of 
$\mathcal{B}_u:\U\times\U\to \R^n$ 
descend from the observations in 
Remark~\ref{rmk:int_z}.
Finally, we have to show that, for every $u\in \U$,
there exists $C>0$ such that
\begin{equation*}
|\mathcal{B}_u(v,w)|_2 \leq C||v||_{L^2}||w||_{L^2}
\end{equation*}
for every $v,w\in \U$.
Recalling \eqref{eq:def_B} and the integral
expression \eqref{eq:int_z_1}-\eqref{eq:int_z_3},
the last inequality follows from
the estimate \eqref{eq:El_inf_norm}, 
from Lemma~\ref{lem:norm_M}, from
Proposition~\ref{lem:C0_bound_traj}
and from the $C^2$-regularity of 
$F^1,\ldots,F^k$.
\end{proof}

In view of the previous result,
for every $u\in \U$, 
we use $D_u^2P_1:\U\times \U\to \R^n$ to denote the second
differential of the end-point map 
$P_1:\U\to\R^n$.
Moreover, for every
$u,v,w\in \U$ we have the following identities:
\begin{equation} \label{eq:2nd_diff_exp}
D_u^2 P_1(v,w) = \mathcal{B}_u(v,w) = z_u^{v,w}(1).
\end{equation}

\begin{remark}
It is possible to prove that the correspondence
$u\mapsto D_u^2P_1$ is continuous. In particular, 
under the further assumption that the controlled
vector fields $F^1,\ldots,F^k$ are 
$C^3$-regular, the application 
$u\mapsto D_u^2P_1$ is Lipschitz-continuous
on the bounded subsets of $\U$. 
Indeed, taking into account
\eqref{eq:2nd_diff_exp} and
\eqref{eq:int_z_1}-\eqref{eq:int_z_3}, 
this fact follows from 
Lemma~\ref{lem:lip_dep_M}, 
from Lemma~\ref{lem:lipsh_El}
and from the regularity of 
$F^1,\ldots,F^k$.  
\end{remark}

For every $\nu \in \R^n$
and for every $u\in \U$, we can consider the 
bilinear form $\nu\cdot D_u^2P_1:\U\times\U\to\R$,
which is defined as
\begin{equation} \label{eq:bi_form}
\nu\cdot D_u^2P_1(v,w) := \langle
\nu, D_u^2P_1(v,w)
\rangle_{\R^n}.
\end{equation}
We conclude this section by showing that,
using the scalar product of $\U$,
the bilinear form defined in \eqref{eq:bi_form}
can be represented as a self-adjoint compact
operator. Before proceeding, it is convenient to 
introduce two auxiliary linear operators.
In this part we assume that the vector fields
$F^1,\ldots,F^k$ are $C^2$-regular.
For every $u\in \U$ let
us consider the application
$\M_u^\nu:\U \to \V$
defined as follows:
\begin{equation}\label{eq:def_op_M}
\M_u^\nu[v](\tau) := \left( M_u(1)M_u^{-1}(\tau)
\sum_{i=1}^k v^i(\tau) \frac{\partial 
F^i(x_u(\tau))}{\partial x}
\right)^T \nu
\end{equation}
for a.e. $\tau \in [0,1]$, where $x_u:[0,1]
\to \R^n$ is the solution of \eqref{eq:ctrl_Cau}
and $M_u:[0,1]\to\R^{n\times n}$ is defined
in \eqref{eq:def_M}. 
We recall that, for every $i=1,\ldots, k$ and for
every $y\in\R^n$,
$\frac{\partial^2 F^i(y)}{\partial x^2}:\R^n\times
\R^n\to\R^n$ is a symmetric and bilinear function.
Hence, for every $i=1,\ldots, k$, for
every $u\in \U$ and for every 
$\tau \in [0,1]$, we have that the application
\[
(\eta_1, \eta_2)\mapsto \nu^T M_u(1)M_u^{-1}(\tau)
\frac{\partial^2 F^i(x_u(\tau))}{\partial x^2}
(\eta_1, \eta_2)
\]
is a symmetric and bilinear form on $\R^n$.
Therefore, for every $i=1,\ldots, k$, for
every $u\in \U$ and for every 
$\tau \in [0,1]$, we introduce the symmetric
matrix $S_u^{\nu, i}(\tau)\in \R^{n\times n}$
 that satisfies the identity
\begin{equation*}
\langle S_u^{\nu, i}(\tau) \eta _1 ,
\eta_2 \rangle_{\R^n}
= 
\nu^T M_u(1)M_u^{-1}(\tau)
\frac{\partial^2 F^i(x_u(\tau))}{\partial x^2}
(\eta_1, \eta_2)
\end{equation*}
for every $\eta_1, \eta_2 \in \R^n$.
We define the linear operator 
$\mathcal{S}_u^\nu: C^0([0,1],\R^n) \to \V$
as follows:
\begin{equation} \label{eq:def_op_S}
\mathcal{S}_u^\nu[v](\tau) :=
\sum_{i=1}^k u^i(\tau) S_u^{\nu, i}(\tau)
v(\tau)
\end{equation}
for every $v\in C^0([0,1],\R^n)$ and
for a.e. $\tau\in [0,1]$.

In the next result we prove that the linear
operators introduced above
are both continuous.

\begin{lemma} \label{lem:cont_op_M}
Let us assume that the vector fields 
$F^1,\ldots,F^k$ defining
the control system \eqref{eq:ctrl_Cau}
are $C^2$-regular. Therefore,
for every $u\in \U$ and for
every $\nu \in \R^n$ the linear 
operators 
$\M_u^\nu:\U\to \V$ and 
$\mathcal{S}_u^\nu: C^0([0,1],\R^n) \to \V$
defined, respectively, in \eqref{eq:def_op_M} and
\eqref{eq:def_op_S} are continuous.
\end{lemma}
\begin{proof}
Let us start with $\M_u^\nu:\U\to \V$.
Using Lemma~\ref{lem:norm_M}
and \eqref{eq:bound_jac_F}, we immediately deduce
that there exists $C_1>0$ such that
\begin{equation*}
||\M_u^\nu [v]||_\V \leq C_1 ||v||_\U
\end{equation*}
for every $v\in \U$.
As regards $\mathcal{S}^\nu:C^0([0,1],\R^n)
\to\V$,
from \eqref{eq:def_op_S} we
deduce that 
\begin{equation*}
\big|\mathcal{S}_u^\nu [v](\tau)\big|_2
 \leq \left( \sum_{i=1}^k |u^i(\tau)| 
|S_u^{\nu,i}(\tau)|_2 \right)
||v||_{C^0}
\end{equation*}
for every $v\in \U$ and for a.e. $\tau \in [0,1]$.
Moreover, from Lemma~\ref{lem:norm_M}, from 
Lemma~\ref{lem:C0_bound_traj} and the 
regularity of $F^1,\ldots,F^k$, we deduce that
there exists $C'>0$ such that
\begin{equation*}
|S_u^{\nu,i}(\tau)|_2 \leq C'
\end{equation*}
for every $\tau \in [0,1]$. Combining the last two 
inequalities and recalling that 
$u\in \U=L^2([0,1],\R^k)$, we deduce that
the linear operator 
$\mathcal{S}_u^\nu:C^0([0,1],\R^n)\to\V$
is continuous.
\end{proof}

We are now in position to represent the bilinear
form $\nu \cdot D_u^2P_1:\U \times \U\to \R$
through the scalar product of $\U$.
Indeed, recalling \eqref{eq:bi_form}
and \eqref{eq:2nd_diff_exp}, from 
\eqref{eq:int_z_1}-\eqref{eq:int_z_3}
for every $u\in \U$ we obtain
that
\begin{align*}
\nu \cdot D_u^2P_1(v,w)&=
\langle M_u^\nu [v], \El_u[w] \rangle_\V
+ \langle M_u^\nu [w], \El_u[v] \rangle_\V
+ \langle \mathcal{S}_u^\nu \El_u [v],
\El_u[w] \rangle_{\V}\\
& = \langle \El_u^* M_u^\nu[v], w\rangle_{\U}
+ \langle (\M_u^\nu)^* \El_u[v], w\rangle_{\U}
+ \langle \El_u^* \mathcal{S}_u^\nu \El_u [v],
w \rangle_{\U}
\end{align*}
for every $v,w \in \U$, where 
$(\M_u^\nu)^*:\V\to\U$ is the adjoint of the
linear operator $\M_u^\nu:\U\to\V$.
Recalling Remark~\ref{rmk:norm_adj}, we 
have that  $(\M_u^\nu)^*$ is a bounded 
linear operator.
This shows that the bilinear form 
$\nu\cdot D_u^2P_1:\U\times\U\to\R$ can be 
represented by the linear operator
$\N_u^\nu:\U\to\U$, i.e., 
\begin{equation} \label{eq:rep_bil_form}
\nu\cdot D_u^2P_1(v,w) = \langle \N_u^\nu[v]
, w\rangle_\U
\end{equation}
for every $v,w\in \U$, where 
\begin{equation}\label{eq:def_op_N}
\N_u^\nu := \El_u^* M_u^\nu
+ (\M_u^\nu)^* \El_u
+  \El_u^* \mathcal{S}_u^\nu \El_u.
\end{equation} 
We conclude this section by proving that 
$\N_u^\nu:\U\to\U$ is a bounded, compact
and self-adjoint operator.

\begin{proposition} \label{prop:comp_2nd_diff}
Let us assume that the vector fields 
$F^1,\ldots,F^k$ defining
the control system \eqref{eq:ctrl_Cau}
are $C^2$-regular. For every $u\in \U$
and for every $\nu\in \R^n$, let $\N_u^\nu:\U\to\U$
be the linear operator 
that represents the bilinear form
$\nu\cdot D_u^2P_1:\U\times \U\to \R$
through the identity \eqref{eq:rep_bil_form}.
Then $\N_u^\nu$ is continuous, 
compact and self-adjoint.
\end{proposition}
\begin{proof}
We observe that the term 
$\El_u^* M_u^\nu
+ (\M_u^\nu)^* \El_u$ at the right-hand
side of \eqref{eq:def_op_N} is continuous, 
since it is obtained as the sum and
the composition of
continuous linear operators,
as shown in Lemma~\ref{lem:bound_comp_El} and
Lemma~\ref{lem:cont_op_M}. Moreover, it is also
compact, since both $\El_u$ and $\El_u^*$ are,
in virtue of Lemma~\ref{lem:bound_comp_El}.
Finally, the fact that $\El_u^* M_u^\nu
+ (\M_u^\nu)^* \El_u$ is self-adjoint 
is immediate.
Let us consider the last term at the right-hand
side of \eqref{eq:def_op_N}, i.e., 
$\El_u^* \mathcal{S}_u^\nu \El_u$. 
We first observe that 
$\mathcal{S}_u^\nu \El_u:\U\to\V$ is continuous, 
owing to Lemma~\ref{lem:cont_op_M} and the
inequality \eqref{eq:El_inf_norm}.  
Recalling that $\El_u^*:\V\to\U$ is compact, 
the composition $\El_u^* \mathcal{S}_u^\nu \El_u
:\U\to\U$ is compact as well. Once again, the 
operator is clearly self-adjoint. 
\end{proof}

\end{subsection}

\end{section}

\begin{section}{Gradient flow: well-posedness and
global definition.} \label{sec:well_posed}

For every $\be > 0$, we consider the
functional $\F^\be :\U\to\R_+$ defined as follows:
\begin{equation} \label{eq:cost}
\Fb(u) := \frac12 ||u||_{L^2}^2 + \be a(x_u(1)),
\end{equation}
where $a:\R^n \to \R_+$ is a non-negative 
$C^1$-regular function,
and, for every $u\in \U$, 
$x_u:[0,1]\to\R^n$ is the solution of the
Cauchy problem \eqref{eq:ctrl_Cau}
corresponding to the admissible control
$u\in\U$.
In this section we want to study the gradient
flow induced by the functional $\Fb$ on the
Hilbert space $\U$.
In particular, we establish
a result that guarantees
existence, uniqueness and global definition of
the solutions of the gradient flow 
equation associated to $\Fb$.
In this section we 
adopt the approach of the monograph 
\cite{LL72}, where the theory of ODEs in
Banach spaces is developed.  

We start from the notion of differentiable 
curve in $\U$. 
We stress that in the present paper 
the time variable
$t$ is exclusively employed for curves
taking values in $\U$. In particular, we recall
that we use $s\in[0,1]$ 
to denote the time variable only
in the control system \eqref{eq:ctrl_Cau}
and in the related objects (e.g.,
admissible controls, controlled trajectories,
etc.).  
Given a curve $U:(a,b)\to\U$, we say that it
is (strongly) differentiable at $t_0\in (a,b)$
if there exists $u\in \U$ such that
\begin{equation} \label{eq:def_der_curv}
\lim_{t\to t_0}
\left|\left|
\frac{U_t-U_{t_0}}{t-t_0} - u
\right|\right|_{L^2}=0.
\end{equation}
In this case, we use the notation 
$\partial_t U_{t_0}:=u$.
In the present section
we study the well-posedness in $\U$ 
of the evolution equation 
\begin{equation} \label{eq:grad_flow}
\begin{cases}
\partial_t U_t = -\G^\beta[U_t],\\
U_0=u_0,
\end{cases}
\end{equation}   
where $\G^\be:\U\to\U$ is the representation
of the differential $d\Fb:\U\to\U^*$
through the Riesz isomorphism, i.e.,
\begin{equation} \label{eq:rep_fun}
\langle\G^\beta[u] ,v\rangle_{L^2} = d_u\Fb(v)
\end{equation}
for every $u,v\in\U$. More precisely,
for every initial datum 
$u_0 \in \U$
we prove that there exists a curve 
$t\mapsto U_t$ that solves
\eqref{eq:grad_flow}, that it is unique and
that it is defined for every $t\geq 0$.  

We first show that $d_u\Fb$
can be actually represented as
an element of $\U$, for every $u \in \U$. We 
immediately observe that this problem reduces to
study the differential of the
end-point cost, i.e., the functional 
$\mathcal{E}: \U \to \R_+$, defined as
\begin{equation} \label{eq:end_cost}
\mathcal{E}(u) := a(x_u(1)),
\end{equation}
where $x_u:[0,1]\to\R^n$ is the solution of
\eqref{eq:ctrl_Cau}
corresponding to the admissible control
$u\in\U$.

\begin{lemma} \label{lem:diff_MM}
Let us assume that the vector fields
$F^1,\ldots,F^k$ defining the control system
\eqref{eq:ctrl_Cau} are $C^1$-regular, as
well as the function $a:\R^n\to\R_+$ 
designing the end-point cost.
Then the functional $\mathcal{E}:\U\to\R_+$ 
is Gateaux
differentiable at every $u\in \U$.
Moreover, using the Riesz's isomorphism,
for every $u\in\U$ the 
differential $d_u\mathcal{E}:\U\to\R$
 can be represented as follows:
\begin{equation} \label{eq:rep_diff_endcost_int}
d_u\mathcal{E}(v) = \int_0^1 \sum_{j=1}^n
\left( \frac{\partial a(x_u(1))}{\partial x^j}
\langle g^j_{1,u}(\tau),v(\tau) \rangle_{\R^k}
\right)
\,d\tau
\end{equation}
for every $v\in \U$, where, 
for every $j=1,\ldots,n$, the function
$g^j_{1,u}\in \U$ is defined as in 
\eqref{eq:diff_end_rep}.
\end{lemma}
\begin{proof}
We observe that the functional 
$\mathcal{E}:\U\to\R_+$
is defined as the composition
\[
\mathcal{E}= a\circ P_1,
\]
where $P_1:\U\to\R^n$ is the end-point 
map defined in \eqref{eq:def_end_p_map}.
Proposition~\ref{prop:diff_endpoint}
guarantees that the end-point map $P_1$
is Gateaux
differentiable at every $u\in \U$. 
Recalling that $a:\R^n\to\R_+$ is assumed to be 
$C^1$, we deduce that,
for every $u\in\U$,
$\mathcal{E}$ is Gateaux differentiable at 
$u$ and that, for every $v\in\U$, the following
identity holds: 
\begin{equation} \label{eq:chain_diff}
d_u\mathcal{E}(v) 
= \sum_{j=1}^n \frac{\partial a (x_u(1))}{\partial
x^j}D_u P_1^j(v),
\end{equation}  
where $x_u:[0,1]\to\R^n$ is the solution of
\eqref{eq:ctrl_Cau} corresponding to the control 
$u\in\U$.
Recalling that
$D_uP_1^1,\ldots,D_uP_1^n:\U\to\R$ are 
linear and continuous functionals
for every $u\in\U$
(see Proposition~\ref{prop:rep_diff}), 
from \eqref{eq:chain_diff}
we deduce that $d_u\mathcal{E}:\U\to\R$ is as well. 
Finally, from \eqref{eq:diff_end_int} 
we obtain \eqref{eq:rep_diff_endcost_int}.
\end{proof}

\begin{remark}
Similarly as done in 
Remark~\ref{rmk:unif_bound_diff},
we can provide a uniform estimate
of the norm of $d_u\mathcal{E}$ when $u$ varies on
a bounded set. Indeed,
recalling Lemma~\ref{lem:C0_bound_traj}
and the fact that $a:\R^n\to\R_+$ is 
$C^1$-regular, 
for every $R>0$
there exists $C'_R>0$ such
that
\[
\left| 
\frac{\partial a(x_u(1))}{\partial x^j}
\right| \leq C_R'
\]
for every $j=1,\ldots,n$ and
for every $u\in \U$ such that $||u||_{L^2}\leq R$.
Combining the last inequality with
\eqref{eq:chain_diff} and
 \eqref{eq:unif_est_dif_endp},
we deduce that there exists $C_R>0$ such that
for every $||u||_{L^2}\leq R$ the
estimate
\begin{equation} \label{eq:unif_est_diff_endc}
|d_u\mathcal{E}(v)|_2 \leq C_R||v||_{L^2}
\end{equation}
holds for every $v\in\U$.
\end{remark}

\begin{remark} \label{rmk:def_lambda}
We observe that, for every $u,v\in \U$,
we can rewrite
\eqref{eq:rep_diff_endcost_int} as follows
\begin{equation} \label{eq:rep_diff_endcost_int2}
d_u\mathcal{E}(v) = \int_0^1
\left\langle F^T(x_u(\tau))\lambda_u^T(\tau),
v(\tau)
\right\rangle_{\R^k}
\,d\tau, 
\end{equation}
where $\lambda_u:[0,1]\to (\R^n)^*$ is an
absolutely continuous curve
defined for every $s\in[0,1]$
 by the relation
\begin{equation} \label{eq:def_lambda}
\lambda_u(s) :=\nabla_{x_u(1)}a \cdot  M_u(1)
M_u^{-1}(s),
\end{equation}
where $M_u:[0,1]\to \R^{n\times n}$ is defined 
as in \eqref{eq:def_M}, and 
$\nabla_{x_u(1)} a$ is understood as a row vector.
Recalling that $s\mapsto M_u^{-1}(s)$ solves
\eqref{eq:def_inv_M}, it turns out that
$s\mapsto \lambda_u(s)$ is the solution of
the following linear Cauchy problem:
\begin{equation} \label{eq:def_lambda_cau}
\begin{cases}
\dot\lambda_u(s) = -\lambda_u(s)
\sum_{i=1}^k\limits
\left(
u^i(s)\frac{\partial F^i(x_u(s))}{\partial x}
\right) & \mbox{for a.e. }s\in[0,1],\\
\lambda_u(1) = \nabla_{x_u(1)}a.
\end{cases}
\end{equation}
Finally, \eqref{eq:rep_diff_endcost_int2}
implies that, for every $u\in\U$,
we can represent $d_u\mathcal{E}$ with the 
function $h_u:[0,1]\to\R^k$ defined as
\begin{equation}\label{eq:rep_diff_endcost_rep}
h_u(s) := F^T(x_u(s))\lambda_u^T(s)
\end{equation}
for a.e. $s\in[0,1]$.
We observe that \eqref{eq:unif_est_diff_endc}
and the Riesz's isometry imply that for every
$R>0$ there exists $C_R>0$ such that
\begin{equation}
 \label{eq:unif_est_rep_diff_endc}
||h_u||_{L^2} \leq C_R
\end{equation}
for every $u\in \U$ such that $||u||_{L^2}\leq R$.
We further underline
that the representation $h_u:[0,1]\to\R^k$ of the
differential $d_u\mathcal{E}$ is actually
absolutely continuous, similarly as observed
in Remark~\ref{rmk:reg_rep_diff_enpoint}
for the representations of the components of the
differential of the end-point map.
\end{remark}

Under the assumption that the controlled vector
fields $F^1,\ldots,F^k$ and the 
function $a:\R^n\to\R_+$ are $C^2$-regular, we
can show that the differential $u\mapsto d_u\mathcal{E}$
is Lipschitz-continuous on bounded sets. 

\begin{lemma}\label{lem:lipsh_diff_endcost}
Let us assume that the vector fields
$F^1,\ldots,F^k$ defining the control system
\eqref{eq:ctrl_Cau} are $C^2$-regular, as
well as the function $a:\R^n\to\R_+$ 
designing the end-point cost.
Then, for every $R>0$
there exists $L_R>0$ such that
\begin{equation} \label{eq:lipsh_diff_cost}
||h_{u+w}-h_u||_{L^2}\leq L_R||w||_{L^2}
\end{equation}
for every $u,w\in \U$
satisfying $||u||_{L^2},||w||_{L^2}\leq R$,
where $h_{u+w},h_{u}$ are the representations,
respectively, of $d_{u+w}\mathcal{E}$ and $d_u
\mathcal{E}$
provided by \eqref{eq:rep_diff_endcost_rep}.
\end{lemma}
\begin{proof}
Let us consider $R>0$.
In virtue of \eqref{eq:rep_diff_endcost_int},
it is sufficient to prove that
there exists $ L_R>0$ such that
\begin{equation} \label{eq:lip_dif_cost_ts}
\left|\left|
\frac{\partial a(x_{u+w}(1))}{\partial x^j}
g^j_{1,u+w} -
\frac{\partial a(x_{u}(1))}{\partial x^j}
g^j_{1,u} 
\right|\right|_{L^2}  \leq  L_R ||w||_{L^2}
\end{equation}
for every $j=1,\ldots,n$
and for every $u,w\in\U$ such that
$||u||_{L^2},||w||_{L^2}\leq R$. 
Lemma~\ref{lem:C0_bound_traj} implies that
there exists 
a compact set $K_R\subset\R^n$
depending only on $R$ such that
$x_u(1),x_{u+w}(1)\in K_R$
for every $u,w\in\U$ satisfying
$||u||_{L^2},||w||_{L^2}\leq R$.
Recalling that $a:\R^n\to\R_+$ is assumed to be
$C^2$-regular, we deduce that
there exists $L_R'>0$ such that
\begin{equation*} 
\left|
\frac{\partial a(y_1)}{\partial x^j}
-
\frac{\partial a(y_2)}{\partial x^j}
\right|_2\leq L_R' |y_1-y_2|_2
\end{equation*}
for every $y_1,y_2\in K_R$.
Moreover, combining the previous inequality with
\eqref{eq:cont_dep_traj}, we deduce that
there exists $L^1_R>0$ such that
\begin{equation} \label{eq:lip_dif_cost_p1}
\left|
\frac{\partial a(x_{u+w}(1))}{\partial x^j}
-
\frac{\partial a(x_u(1))}{\partial x^j}
\right|_2\leq L^1_R||w||_{L^2}
\end{equation}
for every $u,w\in\U$ satisfying
$||u||_{L^2},||w||_{L^2}\leq R$.
On the other hand, using 
\eqref{eq:lipsc_cont_rep_endpoint},
we have that there exists $L^2_R>0$ such that
\begin{equation} \label{eq:lip_dif_cost_p2}
\left|\left|
g^j_{1,u+w} -
g^j_{1,u} 
\right|\right|_{L^2}  \leq L_R^2 ||w||_{L^2}
\end{equation}
for every $u,w\in\U$ satisfying
$||u||_{L^2},||w||_{L^2}\leq R$.
Combining \eqref{eq:lip_dif_cost_p1}
and \eqref{eq:lip_dif_cost_p2}, and recalling
\eqref{eq:unif_est_rep_dif_endp},
the triangular inequality yields 
\eqref{eq:lip_dif_cost_ts}.
\end{proof}

\begin{remark}
In Lemma~\ref{lem:diff_MM} we have computed the 
Gateaux differential $d_u\mathcal{E}$
of the functional $\mathcal{E}:\U\to\R$.
The continuity of the map 
$u\mapsto d_u\mathcal{E}$ implies that
the Gateaux differential coincides with
the Fr\'echet differential (see, e.g., 
\cite[Theorem~1.9]{AP}). 
\end{remark}

Using Lemma~\ref{lem:diff_MM} and 
Remark~\ref{rmk:def_lambda}, we can provide
an expression for the representation
map $\G^\beta:\U\to\U$ defined
in \eqref{eq:rep_fun}. Indeed,
for every $\beta>0$ we have that
\begin{equation} \label{eq:rep_map_expr}
\G^\beta[u] =  u + \beta h_u, 
\end{equation}
where $h_u:[0,1]\to\R^k$ is defined in 
\eqref{eq:rep_diff_endcost_rep}.
Before proving that the solution of
the gradient flow \eqref{eq:grad_flow}
exists and is globally defined,
we report the statement of
a local existence and uniqueness result 
for the solution of ODEs in infinite-dimensional
spaces.

\begin{theorem}\label{thm:ODEs_banach}
Let $(E,||\cdot ||_E)$ be a Banach space, and,
for every $u_0\in E$ and $R>0$, let
$B_R(u_0)$ be the set
\[
B_R(u_0):=\{ u\in E:\, ||u-u_0||_E\leq R \}.
\]
Let $\mathcal{K}:E\to E$ be a continuous 
map such that
\begin{enumerate}
\item[(i)] $||\mathcal{K}[u]||_E\leq M$ for every
$u\in B_R(u_0)$;
\item[(ii)] $||\mathcal{K}[u_1]
-\mathcal{K}[u_2]||_E \leq L||u_1-u_2||_E$
for every $u_1,u_2\in B_R(u_0)$.
\end{enumerate}
For every $t_0\in\R$, let us consider the 
following Cauchy problem:
\begin{equation} \label{eq:cau_prob_banach}
\begin{cases}
\partial_t U_t = \mathcal{K}[U_t],\\
U_{t_0}=u_0.
\end{cases}
\end{equation}
Then, setting $\alpha := \frac{R}{M}$,
the equation \eqref{eq:cau_prob_banach}
admits a unique and continuously 
differentiable 
solution $t\mapsto U_t$, which is
defined for every $t\in \mathcal{I}:=
[t_0-\alpha,t_0+\alpha]$
and satisfies $U_t\in B_R(u_0)$ for
every $t\in\mathcal{I}$.
\end{theorem}
\begin{proof}
This result descends directly from
\cite[Theorem~5.1.1]{LL72}.
\end{proof}

In the following result we show that,
whenever it exists, any solution
of \eqref{eq:grad_flow} is bounded
with respect to the $L^2$-norm.

\begin{lemma} \label{lem:bound_2_traj}
Let us assume that the vector fields
$F^1,\ldots,F^k$ defining the control system
\eqref{eq:ctrl_Cau} are $C^2$-regular, as
well as the function $a:\R^n\to\R_+$ 
designing the end-point cost.
For every initial datum $u_0\in \U$,
let $U:[0,\alpha)\to\U$ be a 
continuously differentiable solution
of  the Cauchy problem
\eqref{eq:grad_flow}.
Therefore, for every $R>0$ 
there exists $C_R>0$ such that, if
$||u_0||_{L^2}\leq R$, then 
\[
||U_t||_{L^2} \leq C_R
\]
for every $t\in [0,\alpha)$.  
\end{lemma}
\begin{proof}
Recalling \eqref{eq:grad_flow}
and using the fact that 
both $\Fb:\U\to\R_+$ and $t\mapsto U_t$
are differentiable, we observe that
\begin{equation} \label{eq:dec_fun_grflow}
\frac{d}{dt}\Fb(U_t)= 
d_{U_t}\Fb(\partial_tU_t)
=  \langle\G^\be[U_t],
\partial_t U_t \rangle_{L^2}
=-||\partial_t U_t||^2_{L^2}
\leq 0
\end{equation}
for every $t\in[0,\alpha)$, and this
immediately implies that
\begin{equation*} 
\Fb(U_t) \leq \Fb(U_0)
\end{equation*}
for every $t\in [0,\alpha)$.
Moreover, from the definition of the functional
$\Fb$ given in \eqref{eq:cost}
and recalling that the end-point term is 
non-negative,
it follows that
$\frac12 ||u||_{L^2}^2 \leq \Fb(u)$ for every 
$u\in \U$. 
Therefore, combining these facts,
 if $||u_0||_{L^2}\leq R$,
we deduce that
\[
\frac12 || U_t ||_{L^2}^2 \leq
\sup_{||u_0||_{L^2}\leq R}\Fb(u_0)\leq \frac12 R^2
+ \sup_{||u_0||_{L^2}\leq R}a(x_{u_0}(1))
\]
for every $t\in[0,\alpha)$.
Finally, using 
Lemma~\ref{lem:C0_bound_traj}
and the continuity of $a:\R^n\to\R_+$,
we deduce the thesis. 
\end{proof}

We are now in position to
prove that the gradient flow equation
\eqref{eq:grad_flow} admits a unique and globally
defined solution.

\begin{theorem}\label{thm:glob_def_gr_flow}
Let us assume that the vector fields
$F^1,\ldots,F^k$ defining the control system
\eqref{eq:ctrl_Cau} are $C^2$-regular, as
well as the function $a:\R^n\to\R_+$ 
designing the end-point cost.
For every $u_0\in\U$, let us consider 
the Cauchy problem \eqref{eq:grad_flow}
with initial datum $U_0=u_0$. 
Then, \eqref{eq:grad_flow} admits
a unique, globally defined 
and continuously differentiable 
solution
$U:[0,+\infty)\to\U$.  
\end{theorem} 
\begin{proof}
Let us fix the initial datum $u_0\in\U$, and let us 
set $R:= ||u_0||_{L^2}$. Let $C_{R}>0$ be the
constant provided by Lemma~\ref{lem:bound_2_traj}. 
Let us introduce $R':= C_{R}+1$ and let us 
consider
\begin{equation*} 
B_{R'}(0) := \{ 
u\in \U: ||u||_{L^2}\leq R' \}.
\end{equation*}
We observe that, for every $\bar u \in \U$
 such that $||\bar u ||_{L^2}
\leq C_{R}$, we have that
\begin{equation} \label{eq:incl_ball}
B_1(\bar u) \subset B_{R'}(0),
\end{equation}
where 
$B_1(\bar u):=\{ u \in \U: ||u -\bar u ||_{L^2}
\leq 1\}$.
Recalling that the vector field 
that generates the gradient flow 
\eqref{eq:grad_flow} has the form 
$\G^\be[u] = u + \be h_u$ for every
$u\in \U$, from \eqref{eq:unif_est_rep_diff_endc}
we deduce that there exists
$M_{R'}>0$ such that
\begin{equation} \label{eq:loc_bound_ODE}
||\G^\be[u]||_{L^2} \leq M_{R'}
\end{equation}
for every $u\in B_{R'}(0)$.
On the other hand, 
Lemma~\ref{lem:lipsh_diff_endcost}
implies that there exists $L_{R'}>0$
such that
\begin{equation} \label{eq:loc_lipsh_ODE}
||\G^\be[u_1]-\G^\be[u_2]||_{L^2} \leq L_{R'}
||u_1-u_2||_{L^2}
\end{equation}
for every $u_1,u_2\in B_{R'}(0)$.
Recalling the inclusion
\eqref{eq:incl_ball},
\eqref{eq:loc_bound_ODE}-\eqref{eq:loc_lipsh_ODE}
guarantee that
the hypotheses of 
Theorem~\ref{thm:ODEs_banach} are satisfied
in the ball $B_1(\bar u)$, for every
$\bar u$ satisfying $||\bar u||_{L^2}\leq C_R$.
This implies that, for every $t_0\in \R$, 
the evolution equation 
\begin{equation} \label{eq:grad_flow_aux}
\begin{cases}
\partial_t U_t = -\G^\be[U_t],\\
U_{t_0}= \bar u,
\end{cases}
\end{equation}
admits a unique and continuously differentiable
solution defined in the interval 
$[t_0-\alpha, t_0+\alpha]$, where we set 
$\alpha := \frac1{M_{R'}}$.
In particular, if we choose $t_0=0$ and 
$\bar u =u_0$ in \eqref{eq:grad_flow_aux},
we deduce that the gradient flow equation
\eqref{eq:grad_flow} with initial datum
$U_0=u_0$ admits a unique
and continuously differentiable solution
$t\mapsto U_t$ defined in
the interval $[0,\alpha]$.
We shall now prove that we can extend this local
solution to every positive time.
In virtue of Lemma~\ref{lem:bound_2_traj},
we obtain that the local solution 
$t\mapsto U_t$ satisfies 
\begin{equation} \label{eq:est_norm_max_sol}
||U_t||_{L^2}\leq C_{R}
\end{equation}
for every $t\in [0,\alpha]$.
Therefore, if we set $t_0= \frac\alpha2$
and $\bar u = U_{\frac\alpha2}$ in
\eqref{eq:grad_flow_aux}, recalling that, if
$||\bar u||_{L^2}\leq C_R$, then
\eqref{eq:grad_flow_aux} admits a unique
solution defined in $[t_0-\alpha, t_0+\alpha]$,
 it turns out that
the curve $t\mapsto U_t$ that solves 
\eqref{eq:grad_flow} with Cauchy datum
$U_0=u_0$ can be uniquely defined for every
$t\in [0,\frac32 \alpha]$. 
Since Lemma~\ref{lem:bound_2_traj} guarantees
that \eqref{eq:est_norm_max_sol} holds
whenever the solution $t\mapsto U_t$ exists, 
we can repeat recursively the argument and we can
extend the domain of the solution to the
whole half-line $[0,+\infty)$.
\end{proof}

We observe that Theorem~\ref{thm:ODEs_banach}
suggests that the solution of
the gradient flow equation
\eqref{eq:grad_flow} could be
 defined also for negative times.
In the following result we investigate this
fact.  

\begin{corollary}\label{cor:back_flow}
Under the same assumptions of 
Theorem~\ref{thm:glob_def_gr_flow},
for every $R_2>R_1>0$, 
there exists $\alpha>0$ such that,
if $||u_0||_{L^2}\leq R_1$, then the solution
$t\mapsto U_t$
of the Cauchy problem \eqref{eq:grad_flow}
with initial datum $U_0=u_0$ is defined 
for every $t\in[-\alpha,+\infty)$.
Moreover, 
$||U_t||_{L^2}\leq {R_2}$ for every 
$t\in [-\alpha,0]$.  
\end{corollary}

\begin{proof}
The fact that the solutions are defined
for every positive time descends from 
Theorem~\ref{thm:glob_def_gr_flow}.
Recalling the expression of $\G^\be:\U\to\U$
provided by \eqref{eq:rep_map_expr},
from \eqref{eq:unif_est_rep_diff_endc}
it follows that, for every $R_2>0$, there
exists $M_{R_2}$ such that
\[
||\G^\be[u]||_{L^2} \leq M_{R_2}
\]
for every 
$u\in B_{R_2}(0):=\{u\in \U: ||u||_{L^2}\leq R_2 \}
.$
On the other hand, in virtue of
Lemma~\ref{lem:lipsh_diff_endcost},
we deduce that there exists $L_{R_2}$
such that
\[
||\G^\be[u_1]-\G^\be[u_2]||_{L^2} \leq L_{R_2}
||u_1-u_2||_{L^2}
\]
for every $u_1,u_2\in B_{R_2}(0)$.
We further observe that, for every
$u_0\in\U$ such that $||u_0||_{L^2}\leq R_1$,
we have the inclusion 
$B_R(u_0):=\{ u\in \U: ||u-u_0||\leq R \}
\subset B_{R_2}(0)$,
where we set $R:=R_2-R_1$.
Therefore,
the previous inequalities guarantee that
the hypotheses of Theorem~\ref{thm:ODEs_banach}
are satisfied in $B_R(u_0)$,
whenever $||u_0||_{L^2}\leq R_1$.
Finally, in virtue of 
Theorem~\ref{thm:ODEs_banach}
and the inclusion $B_R(u_0)\subset B_{R_2}(0)$,
we obtain the thesis with
\[
\alpha = \frac{R_2-R_1}{M_{R_2}}.
\]
\end{proof}

\end{section}

\begin{section}{Pre-compactness of gradient 
flow trajectories} \label{sec:comp}
In Section~\ref{sec:well_posed} we considered the
$\Fb:\U\to\R_+$ defined
in \eqref{eq:cost} and we proved that
the gradient flow equation \eqref{eq:grad_flow}
induced on $\U$ by $\Fb$ admits a unique solution
$U:[0,+\infty)\to\U$, for every Cauchy datum
$U_0 =u_0\in \U$.
The aim of the present section is to investigate
the pre-compactness in $\U$
of the gradient flow trajectories $t\mapsto U_t$.
In order to do that, we first show that,
under suitable regularity assumptions on the
vector fields
$F^1,\ldots,F^k$ and on the function
$a:\R^n\to\R_+$,
for every 
$t\geq 0$ the value of the solution
$U_t\in \U$ has the same
Sobolev regularity as the initial datum $u_0$.
The key-fact is that, when 
$F^1,\ldots,F^k$  are
$C^{r}$-regular with $r\geq 2$
and $a:\R^n\to\R_+$ is of class $C^2$, the map
$\G^\be:H^m([0,1],\R^k)\to H^m([0,1],\R^k)$
is locally Lipschitz continuous, for every
non-negative integer $m\leq r-1$.
This implies that the gradient flow equation 
\eqref{eq:grad_flow} can be studied 
as an evolution equation in the Hilbert space
$H^m([0,1],\R^k)$.

The following result concerns the curve
$\lambda_u:[0,1]\to(\R^n)^*$ defined in
\eqref{eq:def_lambda}.

\begin{lemma}\label{lem:est_lambda}
Let us assume that the vector fields
$F^1,\ldots,F^k$ defining the control system
\eqref{eq:ctrl_Cau} are $C^2$-regular, as
well as the function $a:\R^n\to\R_+$ 
designing the end-point cost.
For every $R>0$, there exists 
$C_R>0$ such that, for every 
$u\in \U$ satisfying $||u||_{L^2}\leq R$,
the following inequality holds 
\begin{equation} \label{eq:bound_lambda}
||\lambda_u||_{C^0}\leq C_R,
\end{equation}
where the curve $\lambda_u:[0,1]\to(\R^n)^*$
is defined as in \eqref{eq:def_lambda}.
Moreover, for every $R>0$, there exists 
$L_R>0$ such that, for every 
$u,w\in \U$ satisfying 
$||u||_{L^2},||w||_{L^2}\leq R$,
for the corresponding curves 
$\lambda_u,\lambda_{u+w}:[0,1]\to(\R^n)^*$
the following
inequality holds:
\begin{equation} \label{eq:lip_lambda}
||\lambda_{u+w} - \lambda_u||_{C^0}\leq L_R||w||_{L^2}.
\end{equation}
\end{lemma}

\begin{proof}
Recalling the definition of $\lambda_u$ given in
\eqref{eq:def_lambda}, we have that
\[
|\lambda_u(s)|_2 \leq |\nabla_{x_u(1)} a|_2
|M_u(1)|_2|M_u^{-1}(s)|_2
\]
for every $s\in[0,1]$, where 
$x_u:[0,1]\to\R^n$ is solution of 
\eqref{eq:ctrl_Cau} corresponding to the control
$u\in \U$.
Lemma~\ref{lem:C0_bound_traj} implies that
there exists $C'_R>0$ such that 
$|\nabla_{x_u(1)} a|_2\leq C'_R$
for every $u\in \U$ such that $||u||_{L^2}\leq R$.
Combining this with
\eqref{eq:norm_M_Minv_est}, we deduce
\eqref{eq:bound_lambda}.

To prove \eqref{eq:lip_lambda} we first observe 
that the $C^2$-regularity of $a:\R^n\to\R_+$
and Proposition~\ref{prop:cont_dep_traj}
imply that, for every $R>0$, there exists
$L'_R>0$ such that
\[
|\nabla_{x_{u+w}(1)}a - \nabla_{x_{u}(1)}a|_2
\leq L'_R ||w||_{L^2}
\]
for every $u,w\in \U$ such that 
$||u||_{L^2},||w||_{L^2}\leq R$.
Therefore,
recalling \eqref{eq:norm_M_Minv_est} and
\eqref{eq:lips_M}-\eqref{eq:lips_M_inv},
we deduce \eqref{eq:lip_lambda}
by applying the triangular inequality
to the identity
\[
|\lambda_{u+w}(s)-\lambda_u(s)|_2
= |\nabla_{x_{u+w}(1)}a\cdot M_{u+w}(1)
M_{u+w}^{-1}(s) - \nabla_{x_{u}(1)}a\cdot M_{u}(1)
M_{u}^{-1}(s)|_2
\]
for every $s\in[0,1]$.
\end{proof}

We recall the notion of {\it Lie bracket} of
vector fields. Let $G^1,G^2:\R^n\to\R^n$ be
two vector fields such that 
$G^1\in C^{r_1}(\R^n,\R^n)$ and
$G^2\in C^{r_2}(\R^n,\R^n)$, with 
$r_1,r_2\geq 1$, and let us set
$r:=\min(r_1,r_2)$. Then the 
{\it Lie bracket of $G^1$ and $G^2$} is
the vector field
$[G^1,G^2]:\R^n\to\R^n$ defined as follows:
\[
[G^1,G^2](y) = \frac{\partial G^2(y)}{\partial x}
G^1(y) - \frac{\partial G^1(y)}{\partial x}
G^2(y).
\]
We observe that $[G^1,G^2]\in C^{r-1}(\R^n,\R^n)$.
In the following result we establish 
some estimates for vector fields obtained
via iterated Lie brackets.

\begin{lemma} \label{lem:Lie_br}
Let us assume that the vector fields
$F^1,\ldots,F^k$ defining the control system
\eqref{eq:ctrl_Cau} are $C^{m}$-regular,
with $m\geq 2$.
For every compact $K\subset \R^n$, there
exist $C>0$ and $L>0$ such that, for every
$j_1,\ldots,j_m =1,\ldots,k$, 
the vector field
\[
G:= [F^{j_m},[\ldots,[ F^{j_3},[F^{j_2},F^{j_1}]]
\ldots]:\R^n\to\R^n
\]
satisfies the following inequalities:
\begin{equation} \label{eq:bound_lie_br}
|G(x)|_2 \leq C
\end{equation}
for every $x\in K$, and 
\begin{equation} \label{eq:lip_lie_br}
|G(x)-G(y)|_2 \leq L|x-y|_2
\end{equation}
for every $x,y\in K$.
\end{lemma}
\begin{proof}
The thesis follows immediately from the fact that
the vector field $G$ is 
$C^1$-regular.
\end{proof}

The next result is the cornerstone
this section.
It concerns the regularity of the function
$h_u:[0,1]\to\R^k$
introduced in \eqref{eq:rep_diff_endcost_rep}.
We recall that, for every
$u\in\U$, $h_u$ is the representation
of the differential
$d_u\mathcal{E}$ through the
scalar product of $\U$, where the functional
$\mathcal{E}:\U\to\R_+$ is defined as in \eqref{eq:end_cost}.
We recall the convention 
$H^0([0,1],\R^k)=L^2([0,1],\R^k)=\U$.

\begin{lemma} \label{lem:sob_est_rep_endcost}
Let us assume that the vector fields
$F^1,\ldots,F^k$ defining the control system
\eqref{eq:ctrl_Cau} are $C^{r}$-regular
with $r\geq 2$, and that
the function $a:\R^n\to\R_+$ 
designing the end-point cost is
$C^2$-regular.
For every $u\in \U$,
let $h_u:[0,1]\to\R^k$ be the representation of the
differential $d_u\mathcal{E}:\U\to\R$ provided by
\eqref{eq:rep_diff_endcost_rep}.
For every integer
$1\leq m\leq r-1$, 
if $u \in H^{m-1}([0,1],\R^k)\subset \U$,
then $h_u \in H^m([0,1],\R^k)$.

Moreover, for every integer
$1\leq m\leq r-1$, for every $R>0$
there exist $C_R^m>0$ and $L_R^m>0$ such that
\begin{equation} \label{eq:sob_est_rep_dM}
|| h_u ||_{H^m} \leq C_R^m
\end{equation}
for every $u\in H^{m-1}([0,1],\R^k)$
such that $||u||_{H^{m-1}} \leq R$,
and 
\begin{equation} \label{eq:lip_sob_est_dM}
||h_{u+w} - h_{u}||_{H^m}
\leq L^m_R||w||_{H^{m-1}}
\end{equation}
for every $u,w \in H^{m-1}([0,1],\R^k)$
such that $||u||_{H^{m-1}},
||w||_{H^{m-1}} \leq R$.
\end{lemma}

\begin{proof}
It is sufficient to prove the thesis 
in the case $m=r-1$, for every integer
$r\geq2$. 
When $r=2,m=1$, we have to prove that, for every
$u\in \U$, the function
$h_u:[0,1]\to\R^k$ is in $H^1$. 
Recalling \eqref{eq:rep_diff_endcost_rep},
we have that, for every $j=1,\ldots,k$,
the $j$-th component of $h_u$ is given by the
product
\begin{equation*}
h_u^j(s) = \lambda_u(s)\cdot F^j(x_u(s))
\end{equation*}
for every $s\in[0,1]$,
where $\lambda_u:[0,1]\to(\R^n)^*$ was defined 
in \eqref{eq:def_lambda}.
Since both $s\mapsto\lambda_u(s)$ and
$s\mapsto F^j(x_u(s))$ are in $H^1$,
then their product is in $H^1$
as well (see, e.g., \cite[Corollary~8.10]{B11}).
Therefore, since $\lambda_u:[0,1]
\to(\R^n)^*$ solves \eqref{eq:def_lambda_cau},
we can compute
\begin{equation} \label{eq:expr_sob_der_1}
\dot h_u^j(s)  
 =  \lambda_u(s) \cdot
\sum_{i=1}^k [F^i,F^j]_{x_u(s)} u^i(s)
\end{equation}
for every $j=1,\ldots,k$ and for a.e.
$s\in[0,1]$.
In virtue of \eqref{eq:bound_lambda}, 
\eqref{eq:est_norm_traj} and
\eqref{eq:bound_lie_br},
for every $R>0$, 
there exists $C'_R>0$ such that 
\[
|\dot h^j_u(s)| \leq C'_R|u(s)|_1
\]
for a.e. $s\in[0,1]$, for
every $j=1,\ldots,k$ and for every
$u\in \U$ such that $||u||_{L^2}\leq R$.
Recalling \eqref{eq:norm_ineq}, 
we deduce that 
\begin{equation} \label{eq:sob_est_1}
||\dot h^j_u||_{L^2} \leq \sqrt k C'_R||u||_{L^2}
\end{equation}
for
every $j=1,\ldots,k$ and for every
$u\in \U$ such that $||u||_{L^2}\leq R$.
Finally, using
\eqref{eq:unif_est_rep_diff_endc},
we obtain that \eqref{eq:sob_est_rep_dM} holds
for $r=2,m=1$.
To prove \eqref{eq:lip_sob_est_dM}, we observe
that, for every $j=1,\ldots,k$
and for every $u,w\in \U$ we have
\begin{align*}
|\dot h_{u+w}^j(s) - \dot h^j_u(s)| &
\leq |\lambda_{u+w}(s)-\lambda_u(s)|_2
\sum_{i=1}^k \Big|[F^i,F^j]_{x_{u+w}(s)}\Big|_2 |u^i(s) + w^i(s)|\\
& \quad + |\lambda_u(s)|_2
\sum_{i=1}^k \Big|[F^i,F^j]_{x_{u+w}(s)}
-[F^i,F^j]_{x_{u}(s)}
\Big|_2 |u^i(s) + w^i(s)|\\
&\quad +|\lambda_u(s)|_2
\sum_{i=1}^k \Big|[F^i,F^j]_{x_{u}(s)}\Big|_2
|w^i(s)|
\end{align*}
for a.e. $s\in[0,1]$.
In virtue of Lemma~\ref{lem:est_lambda},
Lemma~\ref{lem:C0_bound_traj},
Proposition~\ref{prop:cont_dep_traj}
and Lemma~\ref{lem:Lie_br}, for every $R>0$
there exist $L_R'>0$ and $C_R''>0$
such that for every $j=1,\ldots,k$ 
the inequality
\[
|\dot h_{u+w}^j(s) - \dot h^j_u(s)| \leq
L'_R||w||_{L^2} |u(s)+w(s)|_1 
+ C_R''|w(s)|_1
\]
holds
for a.e. $s\in[0,1]$ and for every
$u,w\in \U$ satisfying $||u||_{L^2},||w||_{L^2}
\leq R$. Using \eqref{eq:norm_ineq}, the
previous inequality implies that there exists
$L''_R>0$ such that
\begin{equation} \label{eq:sob_lips_1}
||\dot h_{u+w}^j - \dot h^j_u||_{L^2}
\leq L''_R||w||_{L^2}
\end{equation}
for every $u,w\in \U$ such that 
 $||u||_{L^2},||w||_{L^2}\leq R$.
 Recalling \eqref{eq:lipsh_diff_cost},
we conclude that \eqref{eq:lip_sob_est_dM}
holds for $r=2,m=1$.

For $r=3,m=2$, we have to prove that,
for every $u\in H^1([0,1],\R^k)$, the
function $h_u$ belongs to $H^2([0,1],\R^k)$.
This follows if we show that $\dot h_u
\in H^1([0,1],\R^k)$ for 
for every $u\in H^1([0,1],\R^k)$.
Using the identity \eqref{eq:expr_sob_der_1},
we deduce that, whenever 
$u\in H^1([0,1],\R^k)$,
$\dot h_u^j$ is the product of three 
$H^1$-regular functions, for every
$j=1,\ldots,k$. Therefore, using again
\cite[Corollary~8.10]{B11}, we deduce that
$\dot h_u^j$ is $H^1$-regular as well.
From \eqref{eq:expr_sob_der_1}, for every
$j=1,\ldots,k$ we have that
\begin{equation*} \label{eq:expr_sob_der_2}
\ddot h_u^j(s) = \lambda_u(s)\cdot
\sum_{i_1,i_2=1}^k[F^{i_2},[F^{i_1},F^j]]_{x_u(s)}
u^{i_1}(s)u^{i_2}(s)
+ \lambda_u(s)\cdot \sum_{i_1=1}^k
[F^{i_1},F^j]_{x_u(s)} \dot u^{i_1}(s)
\end{equation*}
for a.e. $s\in[0,1]$.
Using Lemma~\ref{lem:est_lambda},
Lemma~\ref{lem:C0_bound_traj},
Lemma~\ref{lem:Lie_br}, and
recalling Theorem~\ref{thm:comp_sob_imm},
we obtain that, for every $R>0$ there exist
$C_R', C_R''>0$ such that
\begin{equation} \label{eq:sob_est_2}
||\ddot h_u^j(s)||_{L^2} \leq 
C_R' + C_R''||\dot u(s)||_{L^2}
\end{equation}
for a.e. $s\in[0,1]$, for every $j=1,\ldots,k$
and for every 
$u\in H^1([0,1],\R^k)$ such that $||u||_{H^1}
\leq R$.
Therefore, combining
\eqref{eq:unif_est_rep_diff_endc},
\eqref{eq:sob_est_1} and \eqref{eq:sob_est_2},
the inequality \eqref{eq:sob_est_rep_dM}
follows for the case $r=3,m=2$. 
In view of \eqref{eq:lipsh_diff_cost}
and \eqref{eq:sob_lips_1}, in order to 
prove \eqref{eq:lip_sob_est_dM} for
$r=3,m=2$ it is sufficient to
show that, for every
$R>0$ there exists $L_R'>0$ such that
\begin{equation} \label{eq:sob_lips_2}
||\ddot h_{u+w}^j -\ddot h_u^j||_{L^2} \leq
L_R'||w||_{H^1}
\end{equation} 
for every $u,w\in H^1([0,1],\R^k)$
such that $||u||_{H^1},||w||_{H^1}\leq R$.
The inequality \eqref{eq:sob_lips_2} can
be deduced with an argument based on the
triangular inequality,
 similarly as done in the case
$r=2, m=1$.

The same strategy works for every $r\geq 4$.
\end{proof}

The main consequence of 
Lemma~\ref{lem:sob_est_rep_endcost}  
is that, when the map $\G^\be:\U\to\U$
defined in \eqref{eq:rep_map_expr}
is restricted to $H^m([0,1],\R^k)$, 
the restriction $\G^\be:H^m([0,1],\R^k)
\to H^m([0,1],\R^k)$
is bounded and Lipschitz continuous 
on bounded sets.

\begin{proposition} \label{prop:rest_rep_map}
Let us assume that the vector fields
$F^1,\ldots,F^k$ defining the control system
\eqref{eq:ctrl_Cau} are $C^{r}$-regular
with $r\geq 2$, and that
the function $a:\R^n\to\R$ 
designing the end-point cost is
$C^2$-regular.
For every $\be>0$, let $\G^\be:\U\to\U$ be 
the representation map defined 
in \eqref{eq:rep_fun}. Then, for every
integer $1\leq m \leq r-1$, we have that
\begin{equation*} \label{eq:imag_rep_map}
\G^\be(H^m([0,1],\R^k)) \subset
H^m([0,1],\R^k).
\end{equation*}
Moreover, for every
integer $1\leq m \leq r-1$ and
 for every $R>0$ there exists $C_R^m>0$
such that
\begin{equation} \label{eq:bound_rep_rest}
||\G^\be[u]||_{H^m} \leq C^m_R
\end{equation}
for every $u\in H^m([0,1],\R^k)$
such that $||u||_{H^m}\leq R$, and there exists
 $L_R^m>0$ such that
\begin{equation} \label{eq:lipsc_rep_rest}
||\G^\be[u+w]-\G^\be[u]||_{H^m} \leq L^m_R
||w||_{H^m}
\end{equation}
for every $u,w\in H^m([0,1],\R^k)$
such that $||u||_{H^m}, ||w||_{H^m}\leq R$.
\end{proposition}
\begin{proof}
Recalling that for every $u\in \U$ we have
\[
\G^\be[u] = u + \be h_u,
\]
the thesis follows directly from 
Lemma~\ref{lem:sob_est_rep_endcost}.
\end{proof}

Proposition~\ref{prop:rest_rep_map} 
suggests that,
when the vector fields
$F^1,\ldots,F^k$ are $C^r$-regular with $r\geq 2$,
we can restrict the 
gradient flow equation \eqref{eq:grad_flow}
to the Hilbert spaces $H^m([0,1],\R^k)$,
for every integer $1\leq m\leq r-1$. 
Namely, for every integer $1\leq m \leq r-1$, we
shall introduce the application
$\G_m^\be : H^m([0,1],\R^k) \to H^m([0,1],\R^k)$
defined as the restriction 
of $\G^\be:\U\to \U$ to $H^m$, i.e., 
\begin{equation} \label{eq:def_G_m}
\G_m^\be := \G^\be|_{H^m}. 
\end{equation}
For every integer $m\geq1$,
given a curve $U:(a, b)\to H^m([0,1],\R^k)$,
we say that it is (strongly) differentiable at 
$t_0\in(a,b)$ if there exists 
$u\in H^m([0,1],\R^k)$ such that
\begin{equation} \label{eq:derivative_curve_sob}
\lim_{t\to t_0} \left| \left|
\frac{U_t-U_{t_0}}{t-t_0}- u
\right| \right|_{H^m} =0.
\end{equation}
In this case, we use the notation
$\partial_tU_{t_0}:=u$. 
For every
$\ell=1,\ldots,m$
and for every $t\in(a,b)$,
we shall write
$U_t^{(\ell)}\in H^{m-\ell}([0,1],\R^k)$
to denote the $\ell$-th Sobolev derivative
of the function $U_t:s\mapsto U_t(s)$, i.e., 
\[
\int_0^1 \langle U_t(s) , \phi^{(\ell)}(s)
\rangle_{\R^k} \,ds
= (-1)^\ell
\int_0^1 \langle U_t^{(\ell)}(s), \phi(s)
\rangle_{\R^k}\,ds
\]
for every $\phi \in C^\infty_c([0,1],\R^k)$.
It is important to observe
that, for every order of derivation
$\ell=1,\ldots,m$, 
\eqref{eq:derivative_curve_sob} implies that
\begin{equation*}
\lim_{t\to t_0} \left| \left|
\frac{U^{(\ell)}_t-U^{(\ell)}_{t_0}}{t-t_0}
- u^{(\ell)} \right| \right|_{L^2} =0,
\end{equation*}
and we use the notation $\partial_t U^{(\ell)}_{t_0}
:=u^{(\ell)}$.
In particular, for every 
$\ell=1,\ldots,m$,
it follows that
\begin{equation} \label{eq:der_evol_sob_norm}
\frac{d}{dt} || U^{(\ell)}_t ||^2_{L^2}
= 2\int_0^1 
\langle \partial_t U^{(\ell)}_t(s),
U^{(\ell)}_t(s) \rangle_{\R^k}
\,ds = 2 \langle \partial_t U_t^{(\ell)},
U_t^{(\ell)}\rangle_{L^2}.
\end{equation}
In the next result we study the
following evolution equation
\begin{equation} \label{eq:grad_flow_sob}
\begin{cases}
\partial_t U_t = -\G^\be_m [U_t], \\
U_0 = u_0,
\end{cases}
\end{equation}
with $u_0\in H^m([0,1],\R^k)$, and where
$\G_m^\be:
H^m([0,1],\R^k)\to H^m([0,1],\R^k)$ is 
defined as in \eqref{eq:def_G_m}.
Before establishing the existence, uniqueness and
global definition result for the Cauchy 
problem \eqref{eq:grad_flow_sob}, we study
the evolution of the semi-norms
$||U^{(\ell)}_t||_{L^2}$ for $\ell =1,\ldots, m$
along its  solutions. 
 
\begin{lemma} \label{lem:sob_bound}
Let us assume that the vector fields
$F^1,\ldots,F^k$ defining the control system
\eqref{eq:ctrl_Cau} are $C^{r}$-regular
with $r\geq 2$, and that
the function $a:\R^n\to\R_+$ 
designing the end-point cost is
$C^2$-regular.
For every integer $1\leq m\leq r-1$
and for every inital datum 
$u_0\in H^m([0,1],\R^k)$, 
let $U:[0,\alpha)\to H^m([0,1],\R^k)$
be a continuously differentiable
solution of the Cauchy problem
\eqref{eq:grad_flow_sob}.
Therefore, for every $R>0$
there exists $C_{R}>0$
such that, if $||u_0||_{H^m}\leq R$, then
\begin{equation} \label{eq:sob_bound_lem}
||U_t||_{H^m} \leq C_{R}
\end{equation}  
for every $t\in[0,\alpha)$.
\end{lemma}
\begin{proof}
It is sufficient to prove the statement in the case
$r\geq 2, m= r-1$. We shall use an induction 
argument on $r$.

Let us consider the case $r=2, m=1$.
We observe that if $U:[0,\alpha)\to H^1([0,1]
,\R^k)$ is a solution of \eqref{eq:grad_flow_sob}
with $m=1$,
then it solves as well the Cauchy problem
\eqref{eq:grad_flow} in $\U$.
Therefore, recalling that $||u_0||_{L^2}\leq 
||u_0||_{H^1}$,
 in virtue of 
Lemma~\ref{lem:bound_2_traj}, 
for every $R>0$
there exists $C'_{R}>0$ such that,
if $||u_0||_{H^1}\leq R$, we have that
\begin{equation}\label{eq:bound_L2_aux}
||U_t||_{L^2} \leq C'_{R}
\end{equation}
for every $t\in [0,\alpha)$.
Hence it is sufficient to provide an upper bound
to the semi-norm $||U_t^{(1)}||_{L^2}$.
From \eqref{eq:der_evol_sob_norm} and
from the fact that $t\mapsto U_t$ solves
\eqref{eq:grad_flow_sob} for $m=1$,
it follows that
\begin{align*}
\frac{d}{dt}||U_t^{(1)}||_{L^2}^2
&=2\langle \partial_tU_t^{(1)},
U_t^{(1)}\rangle_{L^2} =-2\int_0^1 \left\langle
U^{(1)}_t(s) + \be h_{U_t}^{(1)}(s), U_t^{(1)}(s) 
\right\rangle_{\R^k} \,ds\\
& \leq -2||U_t^{(1)}||_{L^2}^2 
+ 2\be || h_{U_t}^{(1)}||_{L^2}||U_t^{(1)}||_{L^2}\\
&\leq - ||U_t^{(1)}||_{L^2}^2
+{\be^2}|| h_{U_t}^{(1)}||_{L^2}^2
\end{align*}
for every $t\in[0,\alpha)$, where 
$h_{U_t}:[0,1]\to\R^k$ is the absolutely continuous
curve defined in
\eqref{eq:rep_diff_endcost_rep}, and 
$h_{U_t}^{(1)}$ is its Sobolev derivative.
Combining \eqref{eq:bound_L2_aux} with
\eqref{eq:sob_est_rep_dM}, we obtain that 
there exists $C^1_{R}>0$ such that
\begin{equation*}
\frac{d}{dt}||U_t^{(1)}||_{L^2}^2
\leq 
-||U_t^{(1)}||_{L^2}^2
+\be^2 C^1_{R}
\end{equation*}
for every $t\in[0,\alpha)$.
This implies that
\[
||U^{(1)}_t||_{L^2} \leq \max\left\{ 
||U^{(1)}_0||_{L^2}, \be \sqrt{C_R^1} 
\right\}
\]
for every $t\in[0,\alpha)$.
This proves the thesis in the case $r=2, m=1$.

Let us prove the induction step.  
We shall prove the thesis in the case 
$r, m = r-1$. Let 
$U:[0,\alpha)\to H^m([0,1],\R^k)$
be a solution of \eqref{eq:grad_flow_sob}
with $m=r-1$. We observe that $t\mapsto U_t$
solves as well
\begin{equation*}
\begin{cases}
\partial_t U_t = -\G^\be_{m-1}[U_t], \\
U_0= u_0.
\end{cases}
\end{equation*}
Using the inductive hypothesis and that
$||u_0||_{H^{m-1}}\leq ||u_0||_{H^m}$, 
for every $R>0$ there exists $C'_R>0$ such that,
if $||u_0||_{H^m}\leq R$, we have that 
\begin{equation} \label{eq:bound_sob_induct}
||U_t||_{H^{m-1}} \leq C_{R}'
\end{equation}
for every $t\in[0,\alpha)$.
Hence it is sufficient to provide an upper bound
to the semi-norm $||U_t^{(m)}||_{L^2}$.
Recalling \eqref{eq:der_evol_sob_norm}
the same computation as before yields
\begin{align*}
\frac{d}{dt}||U_t^{(m)}||_{L^2}^2
&\leq - ||U_t^{(m)}||_{L^2}^2
+{\be^2}|| h_{U_t}^{(m)}||_{L^2}^2
\end{align*}
for every $t\in[0,\alpha)$. 
Combining \eqref{eq:bound_sob_induct} with
\eqref{eq:sob_est_rep_dM}, we obtain
that 
there exists $C^1_{R}>0$ such that
\begin{equation*}
\frac{d}{dt}||U_t^{(m)}||_{L^2}^2
\leq 
-||U_t^{(m)}||_{L^2}^2
+\be^2 C^1_{R}
\end{equation*}
for every $t\in[0,\alpha)$.
This yields \eqref{eq:sob_bound_lem}
for the inductive case $r, m=r-1$. 
\end{proof}

We are now in position to prove that
the Cauchy problem \eqref{eq:grad_flow_sob}
admits a unique and globally defined solution.
The proof of the following result
follows the lines of the proof of
Theorem~\ref{thm:glob_def_gr_flow}.

\begin{theorem}\label{thm:glob_def_sob}
Let us assume that the vector fields
$F^1,\ldots,F^k$ defining the control system
\eqref{eq:ctrl_Cau} are $C^{r}$-regular
with $r\geq 2$, and that
the function $a:\R^n\to\R_+$ 
designing the end-point cost is
$C^2$-regular.
Then, for every integer $1\leq m\leq r-1$
and for every inital datum 
$u_0\in H^m([0,1],\R^k)$, the evolution
equation \eqref{eq:grad_flow_sob}
admits a unique, globally defined
and continuously differentiable
solution $U:[0,+\infty)\to H^m([0,1],\R^k)$.
Moreover, there exists $C_{u_0}>0$ such that
\begin{equation} \label{eq:sob_est_sol}
||U_t||_{H^m} \leq C_{u_0}
\end{equation}
for every $t\in[0,+\infty)$.
\end{theorem}
\begin{proof}
It is sufficient to prove the statement in the case
$r\geq 2, m= r-1$.
In virtue of 
Lemma~\ref{lem:sob_bound} and
Proposition~\ref{prop:rest_rep_map},
the global existence of the solution
of \eqref{eq:grad_flow_sob} follows from
a {\it verbatim} repetition of the
argument of the proof of 
Theorem~\ref{thm:glob_def_gr_flow}.
Finally, \eqref{eq:sob_est_sol} descends
directly from Lemma~\ref{lem:sob_bound}.
\end{proof}

\begin{remark} \label{rmk:sol_sob}
We insist on the fact that, under the 
regularity assumptions of 
Theorem~\ref{thm:glob_def_sob}, 
if the initial datum $u_0$ is $H^m$-Sobolev
regular with $m\leq r-1$,
then the solution $U:[0,+\infty)\to\U$ of 
\eqref{eq:grad_flow} does coincide
with the solution of \eqref{eq:grad_flow_sob}.
In other words, let us assume that the hypotheses
of Theorem~\ref{thm:glob_def_sob} are met, and let
us consider the evolution equation
\begin{equation} \label{eq:grad_flow_rmk}
\begin{cases}
\partial_t U_t = -\G^\be [U_t],\\
U_0=u_0,
\end{cases}
\end{equation}
where $u_0\in H^m([0,1],\R^k)$, with $m\leq r-1$.
Owing to Theorem~\ref{thm:glob_def_gr_flow},
it follows that \eqref{eq:grad_flow_rmk} admits
a unique solution $U:[0,+\infty)\to\U$.
We claim that $t\mapsto U_t$ solves as well the
evolution equation
\begin{equation} \label{eq:grad_flow_sob_rmk}
\begin{cases}
\partial_t U_t = -\G^\be_m [U_t],\\
U_0 = u_0.
\end{cases}
\end{equation}
Indeed, Theorem~\ref{thm:glob_def_sob} implies
that \eqref{eq:grad_flow_sob_rmk} admits a
unique solution 
$\tilde U:[0,+\infty)\to H^m([0,1],\R^k)$.
Moreover, any solution of \eqref{eq:grad_flow_sob_rmk}
is also a solution of \eqref{eq:grad_flow_rmk},
therefore we must have 
$U_t = \tilde U_t$ for every $t\geq 0$
by the uniqueness of the solution of 
\eqref{eq:grad_flow_rmk}.
Hence, it follows that, 
if the controlled vector fields
$F^1,\ldots,F^k$ and the
function $a:\R^n\to\R_+$ are regular enough, 
then
for every $t\in[0,+\infty)$
each point of the gradient flow trajectory
$U_t$ solving \eqref{eq:grad_flow_rmk} has the same
Sobolev regularity as the initial datum. 
\end{remark}

We now prove a 
pre-compactness result
for the gradient flow trajectories.
We recall that
we use the convention 
$H^0=L^2$.

\begin{corollary} \label{cor:comp_traj}
Under the same assumptions of 
Theorem~\ref{thm:glob_def_sob},
let us consider $u_0\in H^m([0,1],\R^k)$
with the integer $m$ satisfying $1\leq m\leq r-1$.
Let $U:[0,+\infty)\to \U$ be the solution 
of the Cauchy problem \eqref{eq:grad_flow}
with initial condition $U_0 = u_0$. 
Then the trajectory $\{ U_t: t\geq 0 \}$
is pre-compact in $H^{m-1}([0,1],\R^k)$.
\end{corollary}
\begin{proof}
As observed in 
Remark~\ref{rmk:sol_sob}, we have 
that  the solution $U:[0,+\infty)\to \U$
of \eqref{eq:grad_flow} satisfies
$U_t \in H^m([0,1],\R^k)$ for every
$t\geq 0$, and that
it solves \eqref{eq:grad_flow_sob} as well.
In virtue of Theorem~\ref{thm:comp_sob_imm}, the
inclusion $H^m([0,1],\R^k)
\hookrightarrow H^{m-1}([0,1],\R^k)$
is compact for every integer $m\geq 1$, 
therefore from \eqref{eq:sob_est_sol}
we deduce the thesis. 
\end{proof}
\end{section}

\begin{section}{Lojasiewicz-Simon inequality} \label{sec:Loj_Sim}
In this section we show that, when the controlled
vector fields $F^1,\ldots,F^k$ and the
function $a:\R^n\to\R_+$ are real-analytic,
then the functional $\Fb:\U\to\R_+$ satisfies
the Lojasiewicz-Simon inequality. This fact will
be of crucial importance for the convergence proof 
of the next section.

The first result on the 
Lojasiewicz inequality dates back to
1963, when in \cite{L63} Lojasiewicz 
 proved that, if $f:\R^d \to \R$ is a
real-analytic function, then for every $x\in \R^d$
there exist $\gamma \in (1,2]$, $C>0$ and 
$r>0$ such that
\begin{equation} \label{eq:loj_fin_dim}
|f(y)-f(x)| \leq C |\nabla f(y)|_2^\gamma
\end{equation}
for every $y\in \R^d$ satisfying $|y-r|_2<r$. 
This kind of inequalities are ubiquitous in
several branches of Mathematics. For example,
as suggested by Lojasiewicz in 
\cite{L63}, \eqref{eq:loj_fin_dim} can be employed
to study the convergence of the solutions of 
\[
\dot x = -\nabla f(x).
\] 
Another important application can be found in
\cite{P63}, where Polyak studied the convergence 
of the gradient descent algorithm for strongly
convex functions using a particular instance
of \eqref{eq:loj_fin_dim}, which is sometimes 
called Polyak-Lojasiewicz inequality.
In \cite{S83}, Simon extended \eqref{eq:loj_fin_dim}
to real-analytic functionals 
defined on Hilbert spaces, and he employed it
to establish convergence results for 
evolution equations. 
For further details, see also the lecture notes
\cite{S96}. The infinite-dimensional version
of \eqref{eq:loj_fin_dim} is often called 
Lojasiewicz-Simon inequality.
For a complete survey on the topic, we refer
the reader to the paper \cite{C03}.

In this section we prove that for every $\be>0$
the functional $\Fb:\U \to \R_+$ defined in 
\eqref{eq:cost} satisfies the Lojasiewicz-Simon 
inequality. 
We first show that, when the function
$a:\R^n\to\R_+$ involved in the definition of the
end-point cost \eqref{eq:end_cost}
and the controlled vector
fields $F^1,\ldots,F^k$ are real-analytic, 
the functional $\Fb:\U\to\R_+$ is real-analytic
as well, for every $\be>0$.
We recall the notion of  real-analytic 
application defined on a Banach space. 
For an introduction to the subject, see, 
for example, \cite{W65}.
\begin{defn} 
Let $E_1,E_2$ be Banach spaces, and let us consider
an application $\mathcal{T} : E_1 \to E_2$.
The function $\mathcal{T}$ is said to be {\it 
real-analytic at $e_0\in E_1$}
if for every $N\geq 1$ there exists a 
continuous and symmetric multi-linear application 
$l_N \in \mathscr{L}((E_1)^N,E_2)$ and if there exists $r>0$
such that, for every $e\in E_1$ satisfying
$||e-e_0||_{E_1}<r$, we have
\[
\sum_{N=1}^\infty ||l_N||_{\mathscr{L}((E_1)^N,E_2)}
\, ||e-e_0||_{E_1}^N <+\infty 
\]
and 
\[
\mathcal{T}(e)-\mathcal{T}(e_0) =
\sum_{N=1}^\infty l_N (e-e_0)^N,
\]
where, for every $N\geq 1$, we set 
$l_N(e-e_0)^N := l_N(e-e_0,\ldots,e-e_0)$.
Finally, $\mathcal{T}:E_1\to E_2$ is 
{\it real-analytic on $E_1$} if it is real-analytic
at every $e_0\in E_1$.
\end{defn}

In the next result we provide the conditions that
guarantee that $\Fb:\U\to\R$ is real-analytic.

\begin{proposition} \label{prop:real_anal_fun}
Let us assume that the vector fields 
$F^1,\ldots,F^k$ defining the control system
\eqref{eq:ctrl_Cau} are real-analytic, as well
as the function $a:\R^n\to\R_+$ designing the
end-point cost \eqref{eq:end_cost}.
Therefore, for every $\be>0$, 
the functional $\Fb:\U\to\R_+$
defined in \eqref{eq:cost} is
real-analytic.
\end{proposition}
\begin{proof}
Since $\Fb(u)=\frac12 ||u||_{L^2} + \be 
\mathcal{E}(u)$ for every $u\in\U$,
the proof reduces to show that the 
end-point cost
$\mathcal{E}:\U\to\R_+$ is real-analytic.
Recalling the definition of $\mathcal{E}$ given in
\eqref{eq:end_cost} and
the end-point map $P_1:\U\to\R^n$
introduced in
\eqref{eq:def_end_p_map}, we have that 
the former can be expressed
as the composition
\begin{equation*}
\mathcal{E}= a \circ P_1.
\end{equation*}
In the proof of \cite[Proposition~8.5]{ABB}
it is shown that $P_1$ is smooth as soon
as $F^1,\ldots,F^k$ are $C^\infty$-regular, and the
expression of the Taylor expansion of
$P_1$ at every $u\in\U$ is provided. In 
\cite[Proposition~2.1]{AG} it is proved that,
when $a:\R^n\to\R_+$ and the controlled vector
fields are real-analytic, the Taylor series of
$a\circ P_1$ is actually convergent.
\end{proof}

The previous result implies that
the differential $d\Fb:\U\to\U^*$ 
is real-analytic.

\begin{corollary} \label{cor:G_real_an}
Under the same assumptions as in
Proposition~\ref{prop:real_anal_fun},
for every $\be>0$ the  differential
$d\Fb:\U\to\U^*$ is real-analytic.
\end{corollary}
\begin{proof}
Owing to Proposition~\ref{prop:real_anal_fun},
the functional $\Fb:\U\to\R_+$ is real-analytic.
Using this fact, the thesis follows from 
\cite[Theorem~2, p.1078]{W65}.
\end{proof} 

Another key-step in view of the 
Lojasiewicz-Simon inequality is the study of the
Hessian of the functional $\Fb:\U\to\R_+$.
In our framework, the Hessian of $\Fb$ at a point
$u\in\U$ is the bounded linear operator
$\mathrm{Hess}_u\Fb:\U\to\U$ that satisfies the 
identity:
\begin{equation} \label{eq:def_Hess}
\langle\mathrm{Hess}_u\Fb[v],w\rangle_{L^2}
= d^2_u \Fb(v,w)
\end{equation}
for every $v,w\in \U$, where 
$d_u^2\Fb:\U\times\U\to\R$ is the second
differential of $\Fb$ at the point $u$.
In the next proposition we prove that, for every 
$u\in\U$, $\mathrm{Hess}_u\Fb$ has 
finite-dimensional kernel.
We stress on the fact that, unlike the
other results of the present section, we do not
have to assume that $F^1,\ldots,F^k$ and
$a:\R^n\to\R_+$ are real-analytic
to study the kernel of
$\mathrm{Hess}_u\Fb$.

\begin{proposition} \label{prop:ker_fin_dim}
Let us assume that the vector fields $F^1,\ldots,
F^k$ defining the control system \eqref{eq:ctrl_Cau}
are $C^2$-regular, as well as the function
$a:\R^n\to\R_+$ defining the end-point cost 
\eqref{eq:end_cost}. For every $u\in \U$, let 
$\mathrm{Hess}_u\Fb:\U\to\U$ be the linear operator
that represents the second differential
$d^2_u\Fb:\U\times\U\to\R$ through the identity
\eqref{eq:def_Hess}.
Then, the the kernel of $\mathrm{Hess}_u\Fb$ is
finite-dimensional.
\end{proposition}
\begin{proof}
For every $u\in\U$ we have that
\begin{equation*}
d^2_u\Fb(v,w) = \langle v, w\rangle_{L^2} 
+ \be d_u^2\mathcal{E}(v,w)
\end{equation*}
for every $v,w\in \U$. Therefore, we are reduced to
study the second differential of the end-point cost
$\mathcal{E}:\U\to\R_+$. Recalling its definition in
\eqref{eq:end_cost} and applying the chain-rule,
we obtain that
\begin{equation} \label{eq:sec_diff_fun}
d_u^2\mathcal{E}(v,w) =     
\big[D_uP_1 (v)\big]^T \nabla_{x_u(1)}^2a 
\big[ D_uP_1(w) \big] + \big(\nabla_{x_u(1)}a\big)^T
\cdot D_u^2P_1(v,w),
\end{equation}
where $P_1:\U\to\R^n$ is the end-point map 
defined in \eqref{eq:def_end_p_map}, and where
the curve $x_u:[0,1]\to\R^n$ is the solution of 
\eqref{eq:ctrl_Cau} corresponding to the control
$u\in\U$.
We recall that, for every $y\in\R^n$,
we understand $\nabla_y a$ as a row vector.
Let us set $\nu_u := \big(\nabla_{x_u(1)}a\big)^T$ and 
$H_u:= \nabla_{x_u(1)}^2 a$, where 
$H_u:\R^n\to\R^n$ is the self-adjoint linear 
operator associated to the Hessian of
$a:\R^n\to\R_+$ at the point 
$x_u(1)$.
Therefore we can write
\begin{equation}
d_u^2\mathcal{E}(v,w) =
\langle \big( D_uP_1^* \circ H_u \circ   D_uP_1 
\big)[v] , w \rangle_{L^2}
+ \nu_u \cdot  D_u^2P_1(v,w)
\end{equation} 
for every $v,w\in \U$, where $D_uP_1^*:\R^n\to\U$
is the adjoint of the differential 
$D_uP_1:\U\to\R^n$. Moreover, recalling the 
definition of the linear operator 
$\N_u^\nu:\U\to\U$ given in \eqref{eq:rep_bil_form},
we have that 
\begin{equation*}
\nu_u \cdot  D_u^2P_1(v,w) =
\langle \N_u^{\nu_u}[v],w\rangle_{L^2}
\end{equation*}
for every $v,w\in\U$. Therefore, we obtain 
\begin{equation} \label{eq:Hess_endcost}
d_u^2\mathcal{E}(v,w) =
\langle \mathrm{Hess}_u\mathcal{E}[v],w\rangle_{L^2}
\end{equation}
for every $v,w\in\U$, where $\mathrm{Hess}_u
\mathcal{E}:\U\to\U$ is the linear operator
that satisfies the identity:
\begin{equation*} 
\mathrm{Hess}_u \mathcal{E}=
D_uP_1^* \circ H_u \circ   D_uP_1 + \N_u^{\nu_u}.
\end{equation*}
We observe that $\mathrm{Hess}_u \mathcal{E}$
is a self-adjoint compact operator. 
Indeed, $N_u^{\nu_u}$ is self-adjoint and compact
in virtue of Proposition~\ref{prop:comp_2nd_diff}, 
while $D_uP_1^* \circ H_u \circ   D_uP_1$ 
has finite-rank and it self-adjoint as well.
Combining \eqref{eq:sec_diff_fun} and 
\eqref{eq:Hess_endcost}, we deduce that
\begin{equation} \label{eq:Hess_fun}
\mathrm{Hess}_u\Fb = \mathrm{Id} + \be 
\mathrm{Hess}_u\mathcal{E},
\end{equation}
where $\mathrm{Id}:\U\to\U$ is the identity.
Finally, using the Fredholm alternative
(see, e.g., \cite[Theorem~6.6]{B11}), we
deduce that the kernel of 
$\mathrm{Hess}_u\Fb$ is finite-dimensional.
\end{proof}

We are now in position to prove that the functional
$\Fb:\U\to\R_+$ satisfies the Lojasiewicz-Simon
inequality.

\begin{theorem}\label{thm:loj_ineq}
Let us assume that the vector fields 
$F^1,\ldots,F^k$ defining the control system
\eqref{eq:ctrl_Cau} are real-analytic, as well as
the function $a:\R^n\to\R_+$ defining
end-point cost \eqref{eq:end_cost}.
For every $\be>0$ and for every $u\in\U$, there
exists $r>0$, $C>0$ and $\gamma\in (1,2]$ such that
\begin{equation}\label{eq:loj_sim_ineq}
|\Fb(v)-\Fb(u)| \leq C||d_v\Fb||_{\U^*}^\gamma
\end{equation}
for every $v\in \U$ such that $||v-u||_{L^2}< r$.
\end{theorem}
\begin{proof}
If $u\in\U$ is not a critical point for $\Fb$, i.e.,
$d_u\Fb\neq 0$, then there exists $r_1>0$
and $\kappa>0$ such that
\begin{equation*}
||d_v\Fb||_{\U^*}^2 \geq \kappa
\end{equation*}
for every $v\in \U$ satisfying $||v-u||_{L^2}< r_1$.
On the other hand, by the continuity of $\Fb$,
we deduce that there exists $r_2>0$ such that
\begin{equation*}
|\Fb(v)-\Fb(u)|\leq\kappa
\end{equation*} 
for every $v\in\U$ satisfying $||v-u||_{L^2}< r_2$.
Combining the previous inequalities
and taking $r:=\min\{r_1,r_2 \}$, we deduce that,
when $d_u\Fb \neq 0$, \eqref{eq:loj_sim_ineq}
holds with $\gamma=2$.

The inequality \eqref{eq:loj_sim_ineq} in the case
$d_u\Fb = 0$ follows from 
\cite[Corollary~3.11]{C03}. We shall now verify the
assumptions of this result.
First of all, \cite[Hypothesis~3.2]{C03} 
is satisfied, being $\U$ an Hilbert space.
Moreover, \cite[Hypothesis~3.4]{C03} follows
by choosing $W=\U^*$. 
In addition, we recall that $d\Fb:\U\to\U^*$ is 
real-analytic
in virtue of Corollary~\ref{cor:G_real_an},
and that $\mathrm{Hess}_u\Fb$ has finite-dimensional
kernel owing to Proposition~\ref{prop:ker_fin_dim}.
These facts imply that the conditions
(1)--(4) of \cite[Corollary~3.11]{C03} are
verified if we set $X=\U$ and $Y=\U^*$.
\end{proof}
 
\end{section}

\begin{section}{Convergence of the gradient flow}
\label{sec:conv}
In this section we show that 
the gradient flow trajectory 
$U:[0+\infty)\to\U$ that solves 
\eqref{eq:grad_flow} is convergent to a critical
point of the functional $\Fb:\U\to\R$, provided that
the Cauchy datum $U_0 = u_0$ satisfies
$u_0\in H^1([0,1],\R^k)\subset \U$.
The Lojasiewicz-Simon inequality established in
Theorem~\ref{thm:loj_ineq} will play a crucial role
in the proof of the convergence result. 
Indeed, we use this inequality to show that
the trajectories with
Sobolev-regular initial datum
have finite length. This approach was 
first proposed in 
\cite{L63} in the finite-dimensional
framework, and in 
\cite{S83} for evolution PDEs.
In order to satisfy the assumptions of
Theorem~\ref{thm:loj_ineq},
we need to assume throughout the section
that the controlled vector
fields $F^1,\ldots,F^k$ and the function
$a:\R^n\to\R_+$ are real-analytic.

We first recall the notion of the
Riemann integral of a 
curve that takes values in $\U$. For 
general statements and further 
details, we refer the reader to 
\cite[Section~1.3]{LL72}.
Let us consider a continuous curve 
$V:[a,b]\to\U$. Therefore, using
\cite[Theorem~1.3.1]{LL72}, we can define
\begin{equation*}
\int_a^b V_t\,dt := \lim_{n\to\infty}
\frac1n \sum_{k=0}^{n-1} V_{\frac{b-a}{n}k}.
\end{equation*} 
We immediately observe that the following inequality
holds:
\begin{equation} \label{eq:norm_int_Riem}
\left|\left| \int_a^b V_t\,dt \right|\right|_{L^2}
\leq \int_a^b ||V_t||_{L^2} \,dt.
\end{equation}
Moreover, \cite[Theorem~1.3.4]{LL72}
guarantees that,
 if the curve $V:[a,b]\to\U$ is continuously
differentiable, then we have:
\begin{equation}\label{eq:fund_th_calc_Riem}
V_{b} -V_{a} = \int_{a}^{b} \partial_t V_\theta\,
d\theta,
\end{equation}
where $\partial_t V_\theta$ is
the derivative of the curve $t\mapsto V_t$
defined as in \eqref{eq:def_der_curv}
and computed at the instant $\theta\in[a,b]$.
Finally, combining
\eqref{eq:fund_th_calc_Riem} and 
\eqref{eq:norm_int_Riem}, we deduce that
\begin{equation}\label{eq:def_length}
||V_{b}-V_{a}||_{L^2} \leq \int_{a}^{b}
||\partial_t V_\theta||_{L^2}\, d\theta.
\end{equation}
We refer to the quantity at the right-hand side
of \eqref{eq:def_length} as {\it the length of the
continuously differentiable curve $V:[a,b]\to\U$}.

Let $U:[0,+\infty)\to\U$ be the solution of the
gradient flow equation \eqref{eq:grad_flow}
with initial datum $u_0\in \U$.
We say that $u_\infty\in\U$ is a {\it limiting
point} for the curve $t\mapsto U_t$ if there
exists a sequence 
$(t_j)_{j\geq 1}$ such that
$t_j\to+\infty$ and 
$||U_{t_j}-u_\infty||_{L^2}\to 0$ as $j\to\infty$.
In the next result we study the length of 
$t\mapsto U_t$ in a neighborhood of a
limiting point.

\begin{proposition} \label{prop:len_omega}
Let us assume that the vector fields 
$F^1,\ldots,F^k$ defining the control system
\eqref{eq:ctrl_Cau} are real-analytic, as well
as the function $a:\R^n\to\R_+$ designing
the end-point cost.
Let $U:[0,+\infty)\to\U$ be the solution of
the Cauchy problem
\eqref{eq:grad_flow} with initial datum
$U_0=u_0$, and let 
$u_\infty \in \U$ be any of its limiting points.
Then there exists $r>0$
such that the portion of the curve that lies in
$B_{r}(u_\infty)$ has finite length, i.e.,
\begin{equation} \label{eq:loc_fin_len}
\int_{\mathcal{I}}||\partial_t U_\theta||_{L^2} \,
d\theta <\infty,
\end{equation}
where $\mathcal{I}:=\{ t\geq 0: U_t
\in B_{r}(u_\infty) \}$, and
$
B_r(u_\infty):=\{ u\in\U: ||u-u_\infty||_{L^2}<r \}.
$
\end{proposition}
\begin{proof}
Let $u_\infty \in \U$ be a limiting point of
$t\mapsto U_t$, and let $(\bar t_j)_{j\geq 1}$ be
a sequence such that $\bar t_j\to+\infty$ and
$||U_{\bar t_j}- u_\infty||_{L^2}\to 0$ as $j\to\infty$. 
The same computation as in \eqref{eq:dec_fun_grflow}
implies that the functional $\Fb:\U\to\R_+$ 
is decreasing along the trajectory $t\mapsto U_t$,
i.e.,
\begin{equation} \label{eq:dec_fun_taj}
\Fb(U_{t'})\leq\Fb(U_{t}) 
\end{equation}
for every $t'\geq t\geq 0$.
In addition, using the continuity of $\Fb$,
it follows 
that $\Fb(U_{\bar t_j}) \to \Fb(u_\infty)$
as $j\to\infty$.
Combining these facts, we have that
\begin{equation} \label{eq:pos_diff_traj}
\Fb(U_t) - \Fb(u_\infty) \geq 0
\end{equation}  
for every $t\geq 0$.
Moreover, owing to  Theorem~\ref{thm:loj_ineq},
we deduce that there exist
$C>0$, $\gamma \in (1,2]$ and $r>0$ such that
\begin{equation} \label{eq:loj_sim_omega}
|\Fb (v) - \Fb (u_\infty)| 
\leq \frac1C|| d_v\Fb||^{\gamma}_{\U^*}
\end{equation}
for every $v\in B_r (u_\infty)$.
Let $t_1\geq 0$ be the infimum of the
instants such that 
$U_{t}\in B_r(u_\infty)$, i.e.,
\[
t_1:=\inf_{t\geq 0} \{ U_t\in B_r(u_\infty)\}.
\]
We observe that the set where we take the infimum
is nonempty, in virtue of the convergence
$||U_{\bar t_j}- u_\infty||_{L^2}\to 0$ as $j\to \infty$.
Then, there exists $t_1'\in (t_1,+\infty]$ such that
$U_t\in B_r(u_\infty)$ for every 
$t\in (t_1,t_1')$, and we take the supremum
$t_1'>t_1$
such that the previous condition is satisfied, i.e.,
\[
t_1':= \sup_{t'> t_1}\{ U_t\in B_r(u_\infty),
\forall t\in(t_1,t')\}.
\]
If $t_1'<\infty$, we set
\[
t_2:=\inf_{t\geq t_1'} \{ U_t\in B_r(u_\infty)\},
\]
and
\[
t_2':= \sup_{t'> t_2}\{ U_t\in B_r(u_\infty),
\forall t\in(t_2,t')\}.
\]
We repeat this procedure
(which terminates in a finite number of steps
if and only if there exits $\bar t>0$ such that
$U_t \in B_r(u_\infty)$ for every 
$t\geq \bar t$), and we
obtain a family of intervals
$\{ (t_j,t_j') \}_{j=1,\ldots,N}$,
where $N\in \mathbb{N}\cup \{\infty \}$.
We observe that
$\bigcup_{j=1}^N(t_j,t_j') = \mathcal{I}$, 
where we set
$\mathcal{I}:=\{ t\geq 0: U_t
\in B_{r}(u_\infty) \}$.

Without loss of generality, we may assume that
$\mathcal{I}$
is a set of infinite Lebesgue measure.
Indeed, if this is not the case, we would
have the thesis:
\[
\int_{\mathcal{I}}||\partial_t 
U_\theta||_{L^2}
\,d\theta = \int_{\mathcal{I}}
||\G^\be[U_\theta]||_{L^2} \,d\theta <\infty,
\]
since $||\G^\be[u]||_{L^2}$ is bounded on 
the bounded subsets of $\U$, as shown in
\eqref{eq:loc_bound_ODE}.
Therefore, we focus on
the case when the Lebesgue measure of
$\mathcal{I}$ is infinite.
Let us introduce the following sequence:
\begin{equation} \label{eq:def_taus}
 \tau_0 = t_1, \,\,\,
\tau_1 = t_1', \,\,\, \tau_2 = \tau_1 +(t_2'-t_2),
\,\,\,
\ldots, \,\,\,
 \tau_j = \tau_{j-1} + (t'_j-t_j), \,\,\,
 \ldots,
\end{equation}
where $t_1,t'_1,\ldots$ are the extremes of the
intervals $\{ (t_j,t_j') \}_{j=1,\ldots,N}$
constructed above. Finally, we define
the function $\sigma:[\tau_0,+\infty)\to
[\tau_0,+\infty)$ as follows:
\begin{equation} \label{eq:def_sigma}
\sigma(t) :=
\begin{cases}
t &\mbox{if }\tau_0\leq t<\tau_1, \\
t-\tau_1 + t_2 &\mbox{if }
\tau_1\leq t<\tau_2, \\
t-\tau_2 + t_3 
&\mbox{if } \tau_2\leq t<
\tau_3,\\
\cdots & \cdots
\end{cases}
\end{equation}
We observe that $\sigma:[\tau_0,+\infty)\to
[\tau_0,+\infty)$ is 
piecewise affine and  it is monotone increasing. 
In particular, we have that
\begin{equation} \label{eq:mon_sigma}
\sigma(\tau_j) = t_{j+1} \geq t'_j =
\lim_{t\to \tau_j^-} \sigma(t).
\end{equation}
Moreover, from \eqref{eq:def_taus} and from the
definition of the intervals
$\{ (t_j,t_j') \}_{j\geq1}$, it follows that
\begin{equation} \label{eq:incl_traj_ball}
U_{\sigma(t)}\in B_r(u_\infty)
\end{equation}
for every $t\in[\tau_0,+\infty)$.
Let us define the function 
$g:[\tau_0, +\infty) \to \R_+$ as follows:
\begin{equation} \label{eq:def_g_t}
g(t) :=
\Fb(U_{\sigma(t)}) -\Fb(u_\infty),
\end{equation}
where we used \eqref{eq:pos_diff_traj} to
deduce that $g$ is always non-negative.
From \eqref{eq:def_sigma},
we obtain that the restriction 
$g|_{(\tau_j,\tau_{j+1})}$ is $C^1$-regular,
for every $j\geq 0$. Therefore,
using the fact that 
$\dot \sigma|_{(\tau_j,\tau_{j+1})}\equiv 1$,
we compute
\begin{equation*} 
\dot g(t) = \frac{d}{dt}\big(
\Fb(U_{\sigma(t)})-\Fb(u_\infty) \big)
= -d_{U_{\sigma(t)}}\Fb \big(\G^\be[U_{\sigma(t)}]
\big)
\end{equation*}
for every $t\in (\tau_j,\tau_{j+1})$ and for every
$j\geq 0$.
Recalling that $\G^\be:\U\to\U$
is the Riesz's representation of the differential
$d\Fb:\U\to\U^*$, 
it follows that 
\begin{equation} \label{eq:det_time_g}
\dot g(t) = -||d_{U_{\sigma(t)}}\Fb||_{\U^*}^2
\end{equation}
for every $t\in (\tau_j,\tau_{j+1})$ and for every
$j\geq 0$.
Moreover, owing to the Lojasiewicz-Simon
inequality \eqref{eq:loj_sim_omega},
from \eqref{eq:incl_traj_ball} we deduce that
\begin{equation} \label{eq:der_g_est}
\dot g(t) \leq -Cg^{\frac2\gamma}(t)
\end{equation}
for every $t\in (\tau_j,\tau_{j+1})$ and for every
$j\geq 0$.
Let $h:[\tau_0,\infty) \to [0,+\infty)$ be the solution
of the Cauchy problem
\begin{equation} \label{eq:h_aux_ode}
\dot h = -C h^{\frac{2}{\gamma}}, \,\,\,\, 
h(\tau_0) = g(\tau_0),
\end{equation}
whose expression is
\begin{equation*} 
h(t) = 
\begin{cases}
\left(
h(\tau_0)^{1-\frac{2}{\gamma}}
+\frac{(2-\gamma)C}{\gamma}(t-\tau_0)
\right)^{-1-\frac{2\gamma-2}{2-\gamma}}
&\mbox{if }\gamma\in(1,2),\\
h(\tau_0)e^{-Ct}&
\mbox{if } \gamma=2,
\end{cases}
\end{equation*}
for every $t\in[\tau_0,\infty)$.
Using the fact that $g|_{(\tau_0,\tau_1)}$ is
$C^1$-regular, in view of \eqref{eq:der_g_est}, we
deduce that
\begin{equation}\label{eq:ineq_g_h}
g(t) \leq h(t),
\end{equation}
for every $t\in[\tau_0,\tau_1)$. 
We shall now prove that the previous inequality
holds for every $t\in [\tau_0,+\infty)$ using an
inductive argument. Let us assume that
\eqref{eq:ineq_g_h} holds in the interval
$[\tau_0,\tau_j)$, with $j\geq 1$. From the 
definition of $g$,
combining \eqref{eq:dec_fun_taj} and
\eqref{eq:mon_sigma}, we obtain that
\begin{equation} \label{eq:jumps_g}
g(\tau_j) \leq \lim_{t\to\tau_j^-}g(t) 
\leq \lim_{t\to\tau_j^-}h(t)= h(\tau_j).
\end{equation}
Using that the restriction 
$g|_{(\tau_j,\tau_{j+1})}$ is $C^1$-regular, 
in virtue of \eqref{eq:der_g_est},
\eqref{eq:h_aux_ode} and 
\eqref{eq:jumps_g} ,
we extend the 
the inequality \eqref{eq:ineq_g_h} 
to the interval
$[\tau_0,\tau_{j+1})$.
This shows that \eqref{eq:ineq_g_h} is satisfied
for every $t\in[\tau_0,+\infty)$.

We now prove that the portion of the
trajectory that lies in $B_r(u_\infty)$ is
finite.
We observe that
\begin{equation} \label{eq:int_len_1}
\int_{\mathcal I}||\partial_t 
U_\theta||_{L^2} \,d\theta 
= \int_{\mathcal I}
||\G^\be(U_\theta)||_{L^2} \,d\theta
=
\int_{\mathcal I}
||d_{U_\theta}\Fb||_{\U^*} \,d\theta,
\end{equation}
where we recall that $\mathcal{I}=\bigcup_{j=1}^N
(t_j,t'_j)$.
For every $j\geq 1$, in the interval $(t_j,t'_j)$
we use the change of variable
$\theta = \sigma	(\vartheta)$, 
where $\sigma$ is defined in \eqref{eq:def_sigma}.
Using 
\eqref{eq:def_taus} and \eqref{eq:def_sigma},
we observe that $\sigma^{-1}\{(t_j,t'_j)\}=
(\tau_{j-1},\tau_j)$ and that 
$\dot \sigma|_{(\tau_{j-1},\tau_j)}\equiv 1$.
These facts yield
\begin{equation} \label{eq:int_len_2}
\int_{t_j}^{t'_j} 
||d_{U_\theta}\Fb||_{\U^*} \,d\theta
= \int_{\tau_{j-1}}^{\tau_j}
||d_{U_{\sigma(\vartheta)}}\Fb||_{\U^*} \,d\vartheta
= \int_{\tau_{j-1}}^{\tau_j}
\sqrt{-\dot g(\vartheta)}
\,d\vartheta
\end{equation}
for every $j\geq 1$,
where we used \eqref{eq:det_time_g} in the last 
identity.
Therefore, combining \eqref{eq:int_len_1}
and \eqref{eq:int_len_2}, we deduce that
\begin{equation} \label{eq:int_len_3}
\int_\mathcal{I} ||\partial_t U_\theta||_{L^2}\,
d\theta = \int_{\tau_0}^{+\infty}
\sqrt{-\dot g(\vartheta)}\,d\vartheta.
\end{equation}
Then the thesis reduces
to prove that the quantity at the right-hand side
of \eqref{eq:int_len_3} is finite.
Let $\delta>0$ be a positive quantity whose value
will be specified later.
From the Cauchy-Schwarz inequality, it follows
that
\begin{equation} \label{eq:int_len_4}
\int_{\tau_0}^{+\infty}
\sqrt{-\dot g(\vartheta)}\,d\vartheta
\leq 
\left(
\int_{\tau_0}^\infty {-\dot g(\vartheta)}\vartheta^{1+\delta} \, d\vartheta
\right)^{\frac12}
\left(
\int_{\tau_0}^\infty \vartheta^{-1-\delta} \, d\vartheta
\right)^{\frac12}.
\end{equation}
On the other hand, for every $j\geq 1$,
using the integration by parts
on each interval $(\tau_0,\tau_1),\ldots,
(\tau_{j-1},\tau_j)$, we have that
\begin{align*}
\int_{\tau_0}^{\tau_j} {-\dot g(\vartheta)}
\vartheta^{1+\delta} \, d\vartheta
&= \sum_{i=1}^j
\left(
\tau_{i-1}^{1+\delta}g(\tau_{i-1}) -
\tau_i^{1+\delta}g(\tau_i^-) 
+ (1+\delta)
\int_{\tau_{i-1}}^{\tau_i} {g(\vartheta)}\vartheta^{\delta} \, d\vartheta
\right)\\
&\leq 
\tau_{0}^{1+\delta}g(\tau_{0}) -
\tau_j^{1+\delta}g(\tau_j^-)
+ (1+\delta)
\int_{\tau_0}^{\tau_j} 
{h(\vartheta)}\vartheta^{\delta} \, d\vartheta\\
&\leq 
\tau_{0}^{1+\delta}g(\tau_{0})
+ (1+\delta)
\int_{\tau_0}^{\tau_j} 
{h(\vartheta)}\vartheta^{\delta} \, d\vartheta,
\end{align*}
where we introduced the notation
$g(\tau_i^-):=\lim_{\vartheta\to \tau_i^-}
g(\vartheta)$, and we used the first inequality
of \eqref{eq:jumps_g}
and the fact that $g$ is always non-negative.
Finally, if the exponent $\gamma$ in
\eqref{eq:loj_sim_omega} satisfies $\gamma=2$,
we can choose any  positive $\delta>0$.
On the other hand, if $\gamma\in(1,2)$, we choose
$\delta$ such that 
$0<\delta<\frac{2\gamma -2}{2-\gamma}$.
This choice guarantees that
that
\[
\lim_{j\to\infty} \int_{\tau_0}^{\tau_j} {-\dot g(\vartheta)}
\vartheta^{1+\delta} \, d\vartheta
=
\int_{\tau_0}^{\infty} {-\dot g(\vartheta)}
\vartheta^{1+\delta} \, d\vartheta
<\infty,
\]
and therefore, in virtue of \eqref{eq:int_len_4}
and \eqref{eq:int_len_3},
we deduce the thesis. 
\end{proof}

In the following corollary we state an immediate
(but important) consequence of 
Proposition~\ref{prop:len_omega}.

\begin{corollary} \label{cor:lim_omega}
Under the same assumptions as in
Proposition~\ref{prop:len_omega},
let the curve
$U:[0,+\infty)\to\U$ be the solution of
the Cauchy problem
\eqref{eq:grad_flow} with initial datum
$U_0=u_0$. If $u_\infty\in \U$ is a 
limiting point for the curve
$t\mapsto U_t$, then 
the whole solution converges to $u_\infty$ as 
$t\to \infty$, i.e., 
\[
\lim_{t\to \infty} ||U_t-u_\infty||_{L^2} =0.
\]
Moreover, the length of the whole solution is finite. 
\end{corollary}
\begin{proof}
We prove the statement by contradiction. Let us assume
that $t\mapsto U_t$ is not converging to 
$u_\infty$ as $t\to \infty$. Let $B_r(u_\infty)$
be the neighborhood of $u_{\infty}$ given by
Proposition~\ref{prop:len_omega}. 
Diminishing $r>0$ if necessary,
we can find
two sequences $\{ t_j \}_{j\geq 0}$ 
and $\{ t_j' \}_{j\geq 0}$ such that for every
$j\geq0$ the following conditions hold:
\begin{itemize}
\item $t_j<t_j'<t_{j+1}$;
\item $||U_{t_j} -u_\infty||_{L^2}
 \leq \frac{r}{4}$;
\item $\frac{r}{2} \leq 
||U_{t_j'} -u_\infty||_{L^2} \leq r$;
\item $U_t\in B_r(u_\infty)$ for every $t\in 
(t_j,t_j')$.
\end{itemize}
We observe that $\bigcup_{j=1}^\infty(t_j,t_j')
\subset \mathcal{I}$, where 
$\mathcal{I}:=\{ t\geq0: U_t\in B_r(u_\infty) \}$.
Moreover the inequality
\eqref{eq:def_length} and the
previous conditions imply that
\begin{equation*}
\int_{t_j}^{t_j'} ||\partial_t U_\theta||_\U
\,d\theta \geq || U_{t_k'} - U_{t_k}||_{\U}
\geq \frac{r}{4}
\end{equation*}
for every $j\geq 0$.
However, this contradicts \eqref{eq:loc_fin_len}.
Therefore, we deduce that
$|| U_t-u_\infty ||_{\U}\to 0$
as $t\to\infty$.
In particular, this means
 that there exists $\bar t \geq 0$ such
that $U_t \in B_r (u_\infty)$ for every
$t\geq \bar t$. This in turn
 implies that the whole trajectory
has finite length, since
\[
\int_0^{\bar t} ||\partial_t U_\theta||_{L^2} \,d\theta
<+\infty.
\] 
\end{proof}

We observe that in 
Corollary~\ref{cor:lim_omega} we need to assume
{\it a priori} that the solution of the
Cauchy problem \eqref{eq:grad_flow} admits
a limiting point. However, for a general
initial datum $u_0\in \U$ we cannot prove that
this is actually the case. On the other hand,
if we assume more regularity on the Cauchy datum 
$u_0$, we can use the compactness results proved
in Section~\ref{sec:comp}. We recall the notation
$H^0([0,1],\R^k) =: \U$.

\begin{theorem} \label{thm:conv_sobolev}
Let us assume that the vector fields 
$F^1,\ldots,F^k$ defining the control system
\eqref{eq:ctrl_Cau} are real-analytic, as well
as the function $a:\R^n\to\R_+$ designing
the end-point cost.
Let $U:[0,+\infty)\to\U$ be the solution of the
Cauchy problem \eqref{eq:grad_flow}
with initial datum $U_0=u_0$, and let
 $m\geq 1$ be an integer
such that $u_0$ belongs to
$H^m([0,1],\R^k)$. Then there exists $u_\infty
\in H^m([0,1],\R^k)$ such that
\begin{equation} \label{eq:strong_cov_sol}
\lim_{t\to \infty} 
||U_t-u_\infty||_{H^{m-1}} =0.
\end{equation} 
\end{theorem} 
\begin{proof}
Let us consider $u_0\in H^m([0,1],\R^k)$ and let
$U:[0,+\infty)\to\U$ be the solution of
\eqref{eq:grad_flow} satisfying
$U_0=u_0$. 
Owing to Theorem~\ref{thm:glob_def_sob}, we have 
that $U_t\in H^m([0,1],\R^k)$ for every $t\geq0$,
and that the trajectory 
$\{U_t: t\geq 0 \}$
is bounded in
$H^{m}([0,1],\R^k)$. In addition,
from Corollary~\ref{cor:comp_traj}, we 
deduce that $\{U_t: t\geq 0 \}$
is pre-compact 
with respect to the strong topology of
$H^{m-1}([0,1],\R^k)$.
Therefore,
there exist $u_\infty\in H^{m-1}([0,1],\R^k)$ and a 
sequence $(t_j)_{j\geq 1}$
such that we have $t_j\to+\infty$ and
$||U_{t_j}-u_\infty||_{H^{m-1}}\to0$
as $j\to\infty$.
In particular, this implies that
$||U_{t_j}-u_\infty||_{L^2}\to0$ as $j\to\infty$.
In virtue of Corollary~\ref{cor:lim_omega},
we deduce that
$||U_t - u_\infty||_{L^2}\to0$ as $t\to+\infty$.
Using again the pre-compactness of the
 trajectory $\{U_t: t\geq 0 \}$
with respect to the strong topology of
$H^{m-1}([0,1],\R^k)$,
the previous convergence implies that 
$||U_t-u_\infty||_{H^{m-1}}\to 0$ as $t\to+\infty$.

To conclude, we have to show that $u_\infty \in
H^{m}([0,1],\R^k)$. 
Owing to 
the compact inclusion
\eqref{eq:comp_sob_sob} in 
Theorem~\ref{thm:comp_sob_imm}, and
recalling that the trajectory
$\{U_t: t\geq 0 \}$
is pre-compact with respect to the weak topology
of $H^{m}([0,1],\R^k)$, 
the convergence \eqref{eq:strong_cov_sol} 
guarantees that
$u_\infty\in H^{m}([0,1],\R^k)$ and that
$U_t\weak_{H^m}u_\infty$ as $t\to+\infty$.
\end{proof}

In the next result we study the
regularity of the limiting points of the gradient
flow trajectories.

\begin{theorem}\label{thm:reg_crit}
Let us assume that the vector fields 
$F^1,\ldots,F^k$ defining the control system
\eqref{eq:ctrl_Cau} are real-analytic, as well
as the function $a:\R^n\to\R_+$ designing
the end-point cost.
Let $U:[0,+\infty)\to\U$ be the solution of
the Cauchy problem
\eqref{eq:grad_flow} with initial datum
$U_0=u_0$, and let $u_\infty \in \U$ be any of
its limiting points.
Then $u_\infty$ is a critical point for 
the functional $\Fb$, i.e., $d_{u_\infty}\Fb=0$.
Moreover, $u_\infty \in H^m([0,1],\R^k)$ for
every integer $m\geq 1$.
\end{theorem}
\begin{proof}
By Corollary~\ref{cor:lim_omega}, we have that
the solution $t\mapsto U_t$
converges to $u_\infty$ as $t\to+\infty$
with respect to the strong topology of $\U$.
Let us consider the radius $r>0$ prescribed by 
Proposition~\ref{prop:len_omega}. 
If $d_{u_\infty}\Fb \neq 0$,
taking a smaller $r>0$
if necessary, we have that there exists $\e>0$ such
that $|| d_u\Fb ||_{\U^*} \geq \e$ for every
$u\in B_r(u_\infty)$.
Recalling that $||U_t-u_\infty||_{\U}\to0$ 
as $t\to+\infty$, then there exists $\bar t\geq 0$
such that
$U_t \in B_r(u_\infty)$ and for every 
$t\geq \bar t$.
On the other hand, this fact implies that
$|| \partial_t U_t ||_{\U} = ||d_{U_t}\Fb||_{\U^*} 
\geq\e$ for every $t\geq \bar t$, 
but this contradicts \eqref{eq:loc_fin_len},
i.e., 
the fact that the length of the trajectory is 
finite. Therefore, we deduce that
$d_{u_\infty}\Fb=0$.
As regards the regularity of $u_\infty$, we observe 
that $d_{u_\infty}\Fb=0$ implies that
$\G^\be[u_\infty]=0$, which in turn gives
\[
u_\infty = -\be  h_{u_\infty},
\]
where the function $h_{u_\infty}:[0,1]\to\R^k$
is defined as in \eqref{eq:rep_diff_endcost_rep}.
Owing to Lemma~\ref{lem:sob_est_rep_endcost},
we deduce that
the right-hand side of the previous equality
has regularity $H^{m+1}$
whenever $u_\infty \in H^m$, for every 
integer $m\geq 0$. Using a bootstrapping
argument, this implies that $u_\infty
\in H^m([0,1],\R^k)$, for every integer $m\geq1$.
\end{proof}

\begin{remark}
We can give a further characterization of  the
critical points of the functional $\Fb$. Let 
$\hat{u}$ be such that $d_{\hat{u}}\Fb=0$.
Therefore, as seen in the
proof of Theorem~\ref{thm:reg_crit},
we have that the identity
\begin{equation*}
\hat{u}(s) = -\be h_{\hat{u}}(s)
\end{equation*}
is satisfied for every $s\in[0,1]$.
Recalling the definition of $h_{\hat{u}}:[0,1]
\to\R^k$ given in \eqref{eq:rep_diff_endcost_rep},
we observe that the previous relation
yields 
\begin{equation} \label{eq:max_cond}
\hat{u}(s) = \arg \max_{u\in \R^k}
\left\{- \be
\lambda_{\hat{u}}(s)F(x_{\hat{u}}(s)) u
- \frac 12 |u|_2^2
\right\},
\end{equation}
where $x_{\hat{u}}:[0,1]\to\R^n$ solves
\begin{equation}\label{eq:max_traj}
\begin{cases}
\dot x_{\hat{u}}(s) = 
F(x_{\hat{u}}(s))\hat{u}(s) & \mbox{for a.e. }
s\in[0,1],\\
x_{\hat{u}}(0) =x_0,
\end{cases}
\end{equation}
and $\lambda_{\hat{u}}:[0,1]\to(\R^n)^*$ 
satisfies
\begin{equation}\label{eq:max_adj}
\begin{cases}
\dot\lambda_{\hat{u}}(s) = -\lambda_{\hat{u}}(s)
\sum_{i=1}^k\limits
\left(
{\hat{u}}^i(s)\frac{\partial F^i(x_{\hat{u}}(s))}{\partial x}
\right) & \mbox{for a.e. }s\in[0,1],\\
\lambda_{\hat{u}}(1) = \nabla_{x_{\hat{u}}(1)}a.
\end{cases}
\end{equation}
Recalling the Pontryagin Maximum Principle
(see, e.g., \cite[Theorem~12.10]{AS}), 
from \eqref{eq:max_cond}-\eqref{eq:max_adj}
we deduce that the curve 
$x_{\hat{u}}:[0,1]\to\R^n$  is a 
normal Pontryagin extremal for the following
optimal control problem:
\begin{equation*}
\begin{cases}
\min_{u\in \U} \left\{
 \frac12 ||u||_{L^2}^2+\be a(x_u(1))  \right\},\\
\mbox{subject to }
\begin{cases}
\dot x_u = F(x_u) u, \\
x_u(0)=x_0.
\end{cases}
\end{cases}
\end{equation*}
\end{remark}

\end{section}

\begin{section}{$\Gamma$-convergence}
\label{sec:Gamma}
In this section we study the behavior of the 
functionals $(\Fb)_{\be\in\R_+}$ 
as $\be \to +\infty$
using the tools of the $\Gamma$-convergence.
More precisely, we show that the problem of
minimizing the functional $\Fb:\U\to\R_+$
converges as $\be\to+\infty$
(in the sense of $\Gamma$-convergence)
to a limiting minimization problem. 
A classical consequence of this fact is that
the minimizers of the functionals 
$(\Fb)_{\be\in\R_+}$ can provide an 
approximation of the solutions of the
limiting problem.
Moreover, in the present case, the limiting
functional has an important geometrical
meaning, since it is related to
the search of 
sub-Riemannian length-minimizing paths
that connect an initial point to
a target set.
The results obtained in this section
hold under mild regularity assumptions on the
vector fields $F^1,\ldots,F^k$ and on the 
end-point cost $a:\R^n\to\R_+$.
Finally, for a complete introduction to the 
theory of $\Gamma$-convergence,
we refer the reader to the monograph
\cite{D93}. 

In this section we shall work with the weak topology
of the Hilbert space $\U:=L^2([0,1],\R^k)$.
We first establish a preliminary result.
We consider a 
$L^2$-weakly convergent sequence 
$(u_m)_{m\geq 1}\subset\U$, and we study the
convergence of the sequence $(x_m)_{m\geq 1}$,
where, for every $m\geq 1$, the curve
 $x_m:[0,1]\to\R^n$
is the solution of the Cauchy problem 
\eqref{eq:ctrl_Cau} corresponding to the 
admissible control $u_m$.

\begin{lemma} \label{lem:conv_traj}
Let us assume that the vector fields
$F^1,\ldots,F^k$ defining the control system
\eqref{eq:ctrl_Cau} satisfy the
 Lipschitz-continuity condition 
 \eqref{eq:Lipsch_const_F}.
Let $(u_m)_{m\geq 1}
 \subset \U$ be a sequence 
such that $u_m\weak_{L^2} u_\infty$ as 
$m\to\infty$. For every $m\in 
\mathbb{N}\cup \{ \infty \}$,
let $x_m:[0,1]\to \R^n$ be the solution
of \eqref{eq:ctrl_Cau} corresponding to 
the control $u_m$.
Then, we have that
\[
\lim_{m\to\infty} ||x_m - x_\infty||_{C^0} =0.
\] 
\end{lemma}
\begin{proof}
Being the sequence $(u_m)_{m\geq 1}$
weakly convergent, we deduce that there exists
$R>0$ such that $||u_m||_{L^2}\leq R$ for every
$m\geq1$. The estimate established in
Lemma~\ref{lem:C0_bound_traj} implies that
there exists $C_R>0$ such that
\begin{equation} \label{eq:boud_seq_C0}
||x_m||_{C^0} \leq C_R,
\end{equation}
for every $m\geq 1$. Moreover, using the sub-linear
growth inequality \eqref{eq:sub_lin}, we have
that there exists $C>0$ such that
\[
|\dot x_m(s)| \leq \sum_{j=1}^k|F^j(x_m(s)|_2
|u_m^j(s)| \leq C(1+C_R) \sum_{j=1}^k |u_m^j(s)|,
\] 
for a.e. $s\in[0,1]$. 
Then, recalling that $||u_m||_{L^2}\leq R$
for every $m\geq1$, we deduce that
\begin{equation} \label{eq:boud_seq_der_L2}
||\dot x_m||_{L^2} \leq
C(1+C_R)kR
\end{equation}
for every $m\geq1$. Combining 
\eqref{eq:boud_seq_C0} 
and \eqref{eq:boud_seq_der_L2}, we obtain that
the sequence $(x_m)_{m\geq1}$ is pre-compact
with respect to the
weak topology of $H^1([0,1],\R^n)$.
Our goal is to prove that the set of the
$H^1$-weak limiting points of the sequence
$(x_m)_{m\geq1}$ coincides with $\{ x_\infty \}$,
i.e., that the whole sequence 
$x_m\weak_{H^1}x_\infty$ as $m\to\infty$.
Let $\hat x \in H^1([0,1],\R^n)$ be any 
$H^1$-weak limiting point of the sequence
$(x_m)_{m\geq1}$, and let 
$(x_{m_\ell})_{\ell\geq1}$ be a sub-sequence
such that $x_{m_\ell}\weak_{H^1}\hat x$
as $\ell \to \infty$.
Recalling \eqref{eq:comp_sob_C0} in
Theorem~\ref{thm:comp_sob_imm}, we have that 
the inclusion 
$H^1([0,1],\R^n)\hookrightarrow C^0([0,1],\R^n)$
is compact, and this implies that 
\begin{equation} \label{eq:C0_conv_traj}
x_{m_\ell}\to_{C^0}\hat x
\end{equation}
as $\ell \to \infty$.
From \eqref{eq:C0_conv_traj} and the assumption
\eqref{eq:Lipsch_const_F}, 
for every $j=1,\ldots,k$ it follows that  
\begin{equation} \label{eq:conv_field}
||F^j(x_{m_l}) - 
F^j(\hat x)||_{C^0} \to 0
\end{equation}
as $\ell \to\infty$.
Let us  consider a smooth and compactly supported
test function $\phi \in C^\infty_c([0,1],\R^n)$.
Therefore,
recalling that $x_{m_\ell}$ is the solution
of the Cauchy problem \eqref{eq:ctrl_Cau}
corresponding to the control $u_{m_\ell}\in \U$,
we have that
\begin{equation*}
\int_0^1 x_{m_\ell}(s)\cdot \dot \phi(s)\,ds
= -\sum_{j=1}^k\int_0^1 \left(F^j(x_{m_\ell}(s))
\cdot \phi(s) \right) u^j_{m_\ell}(s)\, ds
\end{equation*}
for every $\ell\geq1$.
Thus, passing to the limit as $\ell \to \infty$
in the previous identity, we obtain 
\begin{equation} \label{eq:weak_der}
\int_0^1 \hat x(s)\cdot \dot \phi(s)\,ds
= -\sum_{j=1}^k\int_0^1 \left(F^j(\hat x(s))
\cdot \phi(s) \right) u^j_{\infty}(s)\, ds.
\end{equation}
Indeed, the convergence of the right-hand side
is guaranteed by \eqref{eq:C0_conv_traj}.
On the other hand, for every $j=1,\ldots,k$,
from \eqref{eq:conv_field} we deduce the 
strong convergence
$F^j(x_{m_\ell})\cdot \phi \to_{L^2} 
F^j(\hat x)\cdot \phi$ as $\ell \to \infty$, while 
$u^j_{m_\ell}\weak_{L^2} u^j_\infty$ 
as $\ell \to \infty$
by the hypothesis.
Finally, observing that \eqref{eq:C0_conv_traj}
gives $\hat x(0)=x_0$, we deduce that
\begin{equation*}
\begin{cases}
\dot {\hat x}(s) = F(\hat x(s))u_\infty (s),
& {\mbox{for a.e. }s\in[0,1],}\\
\hat x(0)=x_0,
\end{cases}
\end{equation*}
that implies $\hat x \equiv x_\infty$.
This argument shows that $x_m \weak_{H^1}
x_\infty$ as $m\to \infty$.
Finally, the thesis follows
using again the compact inclusion
\eqref{eq:comp_sob_C0}.
\end{proof}

The standard theory of $\Gamma$-convergence requires
the domain of the functionals to be a metric space,
or, more generally, 
to be equipped with a first-countable 
topology (see \cite[Chapter~12]{ABM}).
 Since the weak topology of 
$\U$ is first-countable (and metrizable) only
on the bounded subsets of $\U$, we shall restrict
the functionals $(\Fb)_{\beta\in \R_+}$ to the set
\[
U_\rho := \{ u\in \U: || u ||_{L^2} \leq \rho \},
\]
where $\rho >0$.
We set 
\begin{equation*}
\Fb_\rho := \Fb|_{\U_\rho},
\end{equation*}
where $\Fb:\U\to\R_+$ is defined in
\eqref{eq:cost}.
Using Lemma~\ref{lem:conv_traj} we deduce that
for every $\be>0$ and $\rho>0$
the functional 
$\Fb_\rho:\U_\rho\to\R_+$ admits a minimizer.

\begin{proposition} \label{prop:dir_met}
Let us assume that the vector fields
$F^1,\ldots,F^k$ defining the control system
\eqref{eq:ctrl_Cau} satisfy the
Lipschitz-continuity condition 
\eqref{eq:Lipsch_const_F}, and that the
function $a:\R^n\to\R_+$
designing the end-point cost is continuous.
Then, for every $\be>0$ and $\rho>0$ there exists
$\hat u\in \U_\rho$ such that
\begin{equation*}
\Fb_\rho(\hat u) = \inf_{\U_\rho} \Fb_\rho.
\end{equation*}
\end{proposition}
\begin{proof}
Let us set $\be>0$ and $\rho>0$. If 
we show that
$\Fb_\rho:\U_\rho\to\R_+$ is sequentially coercive
and sequentially lower semi-continuous, then
the thesis will follow from the Direct Method
of calculus of variations (see, e.g., 
\cite[Theorem~1.15]{D93}).
The sequential coercivity is immediate, since
the domain $\U_\rho$ is sequentially compact,
for every $\rho>0$.
Let $(u_m)_{m\geq 1}\subset \U_\rho$ be a sequence
such that $u_m\weak_{L^2} u_\infty$ as $m\to\infty$.
On one hand, in virtue of 
Lemma~\ref{lem:conv_traj}, we have that
\begin{equation} \label{eq:conv_endcost_aux}
\lim_{m\to\infty} a(x_m(1)) = a(x_\infty(1)),
\end{equation}
where for every $m\in \NN\cup \{ \infty \}$
the curve
$x_m:[0,1]\to\R^n$ is the solution of the
Cauchy problem \eqref{eq:ctrl_Cau} corresponding
to the admissible control $u_m$. 
On the other hand, 
the $L^2$-weak convergence implies that
\begin{equation} \label{eq:L2_low_semic}
||u_\infty||_{L^2} \leq \liminf_{m\to\infty}
||u_m||_{L^2}.
\end{equation}
Therefore, combining \eqref{eq:conv_endcost_aux}
and \eqref{eq:L2_low_semic}, we deduce that
the functional
$\Fb_\rho$ is lower semi-continuous. 
\end{proof}

Before proceeding to the main result of the section,
we recall the definition of $\Gamma$-convergence.

\begin{defn} \label{def:G_conv}
The family of functionals
$(\Fb_\rho)_{\beta \in \R_+}$ is said to
 $\Gamma$-converge
to a functional 
$\F_\rho:\U_\rho\to\R_+\cup\{ +\infty \}$
with respect to the weak topology of $\U$ as $\beta
\to+\infty$ if the following conditions hold:
\begin{itemize}
\item for every 
$(u_\beta)_{\beta\in\R_+}\subset \U_\rho$ such that
$u_\beta \weak_{L^2} u$
 as $\beta \to +\infty$ we have
\begin{equation}\label{eq:Gamma_liminf}
\liminf_{\be\to +\infty} \Fb_\rho(u_\be) 
\geq \F_\rho(u);
\end{equation}
\item for every $u\in \U$ there exists a sequence
$(u_\beta)_{\beta\in\R_+}\subset \U_\rho$
called {\it recovery sequence} such that
$u_\beta \weak_{L^2} u$ as $\beta \to +\infty$
 and such that
\begin{equation}\label{eq:Gamma_limsup}
\limsup_{\be\to +\infty} \Fb_\rho(u_\be) \leq \F_\rho(u).
\end{equation}
\end{itemize}
If \eqref{eq:Gamma_liminf} and \eqref{eq:Gamma_limsup}
are satisfied, then we write $\Fb_\rho\to_\Gamma \F_\rho$
as $\beta \to +\infty$.
\end{defn}

\begin{remark} \label{rmk:rec_sequence}
Let us assume that
$\Fb_\rho \to_\Gamma \F_\rho$ as $\be\to\infty$,
and let us consider a non-decreasing sequence
$(\be_m)_{m\geq 1}$ such that 
$\be_m\to+\infty$ as $m\to\infty$. 
For every $u\in \U_\rho$ and for
every sequence 
$(u_{\be_m})_{m\geq 1}\subset \U_\rho$
such that $u_{\be_m}\weak_{L^2} u$
as $m\to\infty$, we have that
\begin{equation}\label{eq:Gamma_liminf_seq}
\F_\rho(u) \leq \liminf_{m\to\infty}
\F^{\be_m}_\rho(u_{\be_m}).
\end{equation}
Indeed, it is sufficient to 
``embed'' the sequence $(u_{\be_m})_{m\geq 1}$
into a sequence $(u_\be)_{\be\in\R_+}$ such that
$u_\be\weak_{L^2} u$ as $\be\to+\infty$, 
and to observe that
\begin{equation*}
\liminf_{\be\to+\infty} \Fb(u_\be)
\leq \liminf_{m\to\infty}
\F^{\be_m}_\rho(u_{\be_m}).
\end{equation*}
Combining the last inequality with the
$\liminf$ condition \eqref{eq:Gamma_liminf}, we
obtain \eqref{eq:Gamma_liminf_seq}.
\end{remark}

Let $a:\R^n \to \R_+$ be the non-negative function
that defines the end-point cost, and let us
assume that the set $D:=\{ x\in \R^n:\,
a(x)=0 \}$ is non-empty. 
Let us define the functional
$\F_\rho:\U_\rho\to \R \cup \{+\infty\}$ as follows:
\begin{equation}\label{eq:def_F}
\F_\rho(u) :=
\begin{cases}
 \frac12 ||u||_{L^2}^2 &\mbox{if } x_u(1)\in D, \\
+\infty &\mbox{otherwise},
\end{cases}
\end{equation}
where $x_u:[0,1]\to\R^n$ is the solution of
\eqref{eq:ctrl_Cau} corresponding to the control
$u$.

\begin{remark} \label{rmk:end_cost_Gamma}
A situation relevant for applications
occurs 
when the set $D$ is reduced to a single point,
i.e., $D=\{ x_1 \}$ with $x_1\in \R^n$.
Indeed, in this case the
minimization of the limiting functional $\F_\rho$   
is equivalent to find a horizontal 
energy-minimizing path that connect 
$x_0$ (i.e., the Cauchy datum of the control
system \eqref{eq:ctrl_Cau}) to $x_1$. This
in turn coincides with the problem of finding a 
sub-Riemannian length-minimizing curve
that connect $x_0$ to $x_1$
(see \cite[Lemma~3.64]{ABB}).
\end{remark}

We now prove the $\Gamma$-convergence result,
i.e., we show that 
$\Fb_\rho\to_\Gamma\F_\rho$ as $\be \to \infty$
with respect to the
weak topology of $\U$. 

\begin{theorem} \label{thm:G_conv}
Let us assume that the vector fields
$F^1,\ldots,F^k$ defining the control system
\eqref{eq:ctrl_Cau} satisfy the
Lipschitz-continuity condition 
\eqref{eq:Lipsch_const_F}, and that the
function $a:\R^n\to\R_+$ designing the end-point 
cost is continuous.
Given $\rho>0$, 
let us consider $\Fb_\rho:\U_\rho\to\R_+$
with $\be>0$. Let 
$\F_\rho:\U_\rho\to\R_+\cup\{+\infty\}$ be defined as in
\eqref{eq:def_F}. Then the functionals 
$(\Fb_\rho)_{\beta \in \R_+}$ 
$\Gamma$-converge to $\F_\rho$ as $\be\to+\infty$ with 
respect to the weak topology of $\U$.
\end{theorem}

\begin{remark} \label{rmk:controllability}
If $\rho>0$ is not large enough, it may happen that
no control in $\U_\rho$ steers $x_0$ to $D$, i.e.,
$x_u(1) \not \in D$
 for every $u\in\U_\rho$. In this 
case the $\Gamma$-convergence
 result is still valid,
and the $\Gamma$-limit satisfies 
$\F_\rho \equiv +\infty$. 
We can easily avoid this uninteresting
situation when
system \eqref{eq:ctrl_sys} is controllable.
 Indeed,
using the controllability assumption, we deduce that
there exists a control $\tilde u \in \U$ such that
the corresponding trajectory $x_{\tilde u}$ 
satisfies
$x_{\tilde u}(1) \in D$.
On the other hand, we have 
that 
\[
\inf_{u\in \U} \Fb(u) \leq \Fb(\tilde u) 
\]
for every $\beta>0$. Moreover, using the fact that
$x_{\tilde u}(1) \in D$
and recalling the definition of
$\Fb$ in \eqref{eq:cost}, we have that
\[
\Fb(\tilde u) = \frac12 
||\tilde u||^2_{L^2}
\]
for every $\beta>0$. The fact that
the end-point cost $a:\R^n\to\R_+$ 
is non-negative implies that 
$\Fb(u)> \Fb(\tilde u)$
whenever $||u||_{L^2} > ||\tilde u||_{L^2}$.
Setting $\rho = ||\tilde u||_{L^2}$, we deduce that
\[
\inf_{u\in \U} \Fb(u) = \inf_{u\in \U_\rho}\Fb_\rho(u).
\]
Moreover, this choice of $\rho$ guarantees that
the $\Gamma$-limit $\F_\rho \not \equiv +\infty$,
since we have that $\F_\rho (\tilde u) < +\infty$.
\end{remark}

\begin{proof}[Proof of Theorem~\ref{thm:G_conv}]
We begin with the $\limsup$ condition
\eqref{eq:Gamma_limsup}. If 
$\F_\rho(u)=+\infty$, the inequality is trivially 
satisfied. Let us assume that $\F_\rho(u)<+\infty$.
Then setting $u_\be = u$ for every $\be>0$,
we deduce that
$x_u(1) = x_{u_\be}(1)\in D$,
where $x_u:[0,1]\to\R^n$ is the solution of the
Cauchy problem \eqref{eq:ctrl_Cau} corresponding
to the control $u$.
Recalling that $a|_D\equiv 0$, we have that
\begin{equation*} 
\Fb_\rho(u_\be) = \frac12 ||u||_{L^2}^2 = \F_\rho(u)
\end{equation*}
for every $\be>0$.
This proves the $\limsup$ condition.

We now prove the $\liminf$ condition 
\eqref{eq:Gamma_liminf}.
Let us consider $( u_\be)_{\be\in\R_+}
\subset \U_\rho$ such
that $u_\be \weak_{L^2} u$ as $\be \to \infty$, and
such that
\begin{equation} \label{eq:liminf_eq}
\liminf_{\be \to +\infty} \Fb_\rho(u_\be) = C.
\end{equation}
We may assume that $C<+\infty$. If
this is not the case, then \eqref{eq:Gamma_liminf}
trivially holds.
Let us extract $( \be_m )_{m\geq0}$ such that
$\be_m \to +\infty$ and
\begin{equation} \label{eq:subseq_liminf}
\lim_{m\to \infty} \F_\rho^{\be_m}(u_{\be_m}) =
  \liminf_{\be \to +\infty} \Fb_\rho(u_\be) = C.
\end{equation}
For every $m\geq0$,
let $x_{\be_m}:[0,1]\to \R^n$ be the curve defined
as the solution of the Cauchy problem 
\eqref{eq:ctrl_Cau} corresponding to the
control $u_{\be_m}$, and let
$x_u:[0,1]\to\R^n$ be the solution corresponding
to $u$.
Using Lemma~\ref{lem:conv_traj}, we deduce that
$x_{\beta_m}\to_{C^0}x_u$ as $m\to\infty$. 
In particular, we obtain that 
$x_{\be_m}(1)\to x_u(1)$ as $m\to \infty$.
On the other hand, the limit in 
\eqref{eq:subseq_liminf} implies that there 
exists $\bar m\in\NN$ such that
\[
\beta_m a(x_{\be_m}(1)) \leq \F^{\be_m}_\rho
(u_{\be_m}) \leq C + 1,
\]
for every $m\geq \bar m$. Recalling that
$\be_m\to \infty$ as $m\to\infty$, 
the previous inequality yields
\[
a(x_u(1))=\lim_{m\to\infty}a(x_{\be_m}(1))=0,
\]
i.e., that $x_u(1)\in D$.
This argument proves that, if 
$u_\beta \weak_{L^2} u$ as $\beta \to\infty$
and if the quantity at the
right-hand side of \eqref{eq:liminf_eq}
is finite, then the limiting control $u$ steers
$x_0$ to $D$. In particular, this shows
that $\F_\rho(u)<+\infty$.
Finally, in order to establish 
\eqref{eq:Gamma_liminf}, we observe that
\[
\F_\rho(u) = \frac12 ||u||_{L^2}^2 \leq \liminf_{n\to \infty}
\frac12 || u_{\be_n}||_{L^2}^2\leq \liminf_{n\to \infty}
\F_\rho^{\be_n}(u_{\be_n}) = \liminf_{\be \to +\infty}
\F_\rho^\be(u_\be). 
\]
\end{proof}

The theorem that we present below 
motivates the interest in the $\Gamma$-convergence
result just established.
Indeed, we can investigate the asymptotic  
the behavior of
$(\inf_{\U_\rho}\Fb_\rho)_{\be\in\R_+}$ 
as $\be\to+\infty$. Moreover, it turns out that
the minimizers of $\Fb_\rho$ provide 
approximations of the minimizers of 
the limiting functional $\F_\rho$, 
with respect to the {\it strong topology}
of $L^2$.
The first part of Theorem~\ref{thm:fund_cons}
 holds for every 
$\Gamma$-convergent sequence of
equi-coercive functionals 
(see, e.g., 
\cite[Corollary 7.20]{D93}).
On the other hand, the conclusion of the second
part relies on the particular structure of
$(\Fb)_{\be\in\R_+}$. 

\begin{theorem}  \label{thm:fund_cons}
Under the same assumptions of 
Theorem~\ref{thm:G_conv},
given $\rho>0$ we have that
\begin{equation} \label{eq:inf_seq}
\lim_{\be\to\infty}\, \inf_{\U_\rho}\Fb_\rho =
\inf_{\U_\rho}\F_\rho.
\end{equation}
Moreover, under the further assumption
 that $\F_\rho\not \equiv
+\infty$,
for every $\be>0$ let $\hat u_\be$ 
be a minimizer of $\Fb_\rho$.
Then, for every non-decreasing sequence
$(\be_m)_{m\geq 1}$
such that $\be_m\to +\infty$ as 
$m\to\infty$, $(\hat u_{\be_m})_{m\geq 1}$ is
pre-compact with respect to the strong topology
of $\U_\rho$, and every limiting point
of $(\hat u_{\be_m})_{m\geq 1}$ is a minimizer
of $\F_\rho$.
\end{theorem}

\begin{proof}
For every $\be>0$ let $\hat u_\be$ 
be a minimizer of $\Fb_\rho$, that exists
in virtue of Proposition~\ref{prop:dir_met}. 
Let us consider a  non-decreasing sequence
$(\be_m)_{m\geq 1}$
such that $\be_m\to +\infty$ as $m\to\infty$
and such that
\begin{equation} \label{eq:lim_inf_aux}
\lim_{m\to\infty} \F^{\be_m}_\rho(\hat u_{\be_m})
= \lim_{m\to\infty} 
\, \inf_{U_\rho}\F^{\be_m}_\rho
= \liminf_{\be\to+\infty} \, \inf_{U_\rho} \Fb_\rho.
\end{equation}
Recalling that
$(\hat u_{\be_m})_{m\geq1}\subset \U_\rho$,
we have that there exists
$\hat u_\infty\in \U_\rho$ and
a sub-sequence
$(\be_{m_j})_{j\geq1}$ such that
$\hat u_{\be_{m_j}}\weak_{L^2} \hat u_\infty$
as $j\to\infty$.
Since $\Fb_\rho\to_\Gamma\F_\rho$ as
$\be\to+\infty$, 
the inequality \eqref{eq:Gamma_liminf_seq}
derived in Remark~\ref{rmk:rec_sequence}
implies that
\begin{equation} \label{eq:lim_inf_aux_2}
\F_\rho(\hat u_\infty) \leq 
\lim_{j\to\infty} \F^{\be_{m_j}}_\rho(u_{\be_{m_j}})
= \liminf_{\be\to+\infty} \, \inf_{U_\rho} \Fb_\rho,
\end{equation}
where we used \eqref{eq:lim_inf_aux} in the 
last identity.
On the other hand, 
for every $u\in\U_\rho$
let $(u_\be)_{\be\in\R_+}$
be a recovery sequence for $u$, i.e.,
a sequence that satisfies the $\limsup$ condition
\eqref{eq:Gamma_limsup}.
Therefore, we have that 
\begin{equation} \label{eq:lim_sup_aux}
\F_\rho (u) \geq
\limsup_{\be\to+\infty}\Fb_\rho(u_\be)
\geq \limsup_{\be\to+\infty}
\, \inf_{\U_\rho}\Fb_\rho.
\end{equation}
From \eqref{eq:lim_inf_aux_2} and
\eqref{eq:lim_sup_aux}, we deduce that
\begin{equation*}
\F_\rho(u) \geq \F_\rho(\hat u_\infty)
\end{equation*}
for every $u\in \U_\rho$, i.e., 
\begin{equation} \label{eq:minimiz_G_lim}
\F_\rho(\hat u_\infty) = \inf_{\U_\rho} \F_\rho.
\end{equation}
Finally, setting $u=\hat u_\infty$ in 
\eqref{eq:lim_sup_aux}, we obtain
\begin{equation} \label{eq:limit_inf_aux}
\F_\rho (\hat u_\infty) =
\lim_{\be\to\infty} \, \inf_{\U_\rho}\Fb_\rho.
\end{equation}
From \eqref{eq:minimiz_G_lim} and
\eqref{eq:limit_inf_aux}, it follows that
\eqref{eq:inf_seq} holds.

We now focus on the second part of the thesis.
For every $\be>0$ let $\hat u_\be$ 
be a minimizer of $\Fb_\rho$, as before. 
Let $(\be_m )_{m\geq1}$ be a non-decreasing sequence 
such that $\be_m\to+\infty$ as $m\to\infty$,
and let us consider $(\hat u_{\be_m})_{m\geq1}$.
Since $(\hat u_{\be_m})_{m\geq1}$ is 
$L^2$-weakly pre-compact, there exists 
$\hat u\in \U_\rho$ and a sub-sequence 
$(\hat u_{\be_{m_j}})_{j\geq1}$
 such that
$\hat u_{\be_{m_j}}\weak_{L^2}\hat u$
as $j\to\infty$. From the first part of the thesis
it descends that $\hat u$ is a minimizer 
of $\F_\rho$.
Indeed, in virtue of \eqref{eq:Gamma_liminf_seq},
we have that
\begin{equation*}
\F_\rho(\hat u) \leq 
\liminf_{j\to\infty} 
\F_\rho^{\be_{m_j}}(\hat u_{\be_{m_j}}) 
= \lim_{j\to\infty}\, \inf_{\U_\rho}\F_\rho^{\be_{m_j}} = \inf_{\U_\rho}\F_\rho,
\end{equation*}
where we used $
\F_\rho^{\be_{m_j}}(\hat u_{\be_{m_j}})=\inf_{\U_\rho}\F_\rho^{\be_{m_j}}$ 
and the identity \eqref{eq:inf_seq}.
The previous relation guarantees that
\begin{equation} \label{eq:strong_conv_aux1}
\F_\rho(\hat u) = \inf_{\U_\rho}\F_\rho,
 = \lim_{j\to\infty}
\F_\rho^{\be_{m_j}}(\hat u_{\be_{m_j}}).
\end{equation}
 To conclude we have to show that
\begin{equation} \label{eq:strong_conv_ths}
\lim_{j\to\infty}||\hat u_{\be_{m_j}}-
\hat u||_{L^2} = 0.
\end{equation}
Using the assumption $\F_\rho \not \equiv +\infty$,
from the minimality of $\hat u$
we deduce that $\F_\rho(\hat u) = 
\frac12||\hat u||_{L^2}^2$.
Hence, \eqref{eq:strong_conv_aux1}
implies that
\begin{equation} \label{eq:strong_conv_aux2}
\frac12 ||\hat u||_{L^2}^2 
 = \lim_{j\to\infty}
\F^{\be_{m_j}}_\rho(\hat u_{\be_{m_j}})
\geq 
\limsup_{j\to\infty} \frac12 ||u_{\be_{m_j}}||_{L^2}^2,
\end{equation}
where we used that 
$\Fb_\rho(u)\geq \frac12 ||u||_{L^2}^2$
for every $\be>0$ and
for every $u\in \U_\rho$.
From \eqref{eq:strong_conv_aux2} and from the 
weak convergence $\hat u_{\be_{m_j}}\weak_{L^2}
\hat u$ as $j\to\infty$, we deduce that 
\eqref{eq:strong_conv_ths} holds. 
\end{proof}

\end{section}

\section*{Conclusions}
In this paper we have considered an optimal control
problem in a typical framework of sub-Riemannian
geometry. In particular, we have studied 
the functional given by the weighted sum
of the energy of the admissible trajectory 
(i.e., the squared $2$-norm of the control)
and of an end-point cost. \\
We have written the gradient flow
induced by the functional on the Hilbert space
of admissible controls. 
We have proved that,
when the data of the problem are
real-analytic, the gradient flow trajectories
converge to stationary points of the functional
as soon as the starting point has Sobolev
regularity.  \\
The $\Gamma$-convergence result bridges
the functional
considered in the first part of the paper
with the problem of joining two assigned points with 
an admissible length-minimizer path. This fact may 
be of interest for designing methods to approximate
sub-Riemannian length-minimizers. 
Indeed, a natural approach could be to project  
the gradient flow onto a proper finite-dimensional
subspace of the space of admissible controls, and
to minimize the weighted functional restricted
to this subspace.
We leave further development of these ideas for 
future work.

\end{document}